\documentclass[11pt]{amsart}

\usepackage[utf8]{inputenc}

\usepackage{amssymb}
\usepackage{amsthm}
\usepackage{stmaryrd}

\usepackage[shortlabels]{enumitem}
\setenumerate{itemsep=0.005ex}
\usepackage[bookmarks,colorlinks=true,citecolor=blue]{hyperref}

\usepackage{microtype}

\theoremstyle{plain}
\newtheorem{thm}{Theorem}[section]
\newtheorem{mthm}{Theorem}
\newtheorem*{BTDet}{Borel-Turing determinacy}
\newtheorem*{thm1.1}{Theorem 1}
\newtheorem*{thm1.1a}{Theorem 1}
\newtheorem*{thm1.2}{Theorem 2}
\newtheorem*{thm1.2a}{Theorem 2}

\newtheorem{lem}[thm]{Lemma}
\newtheorem*{lem*}{Lemma}
\newtheorem*{mlem*}{Main Lemma}
\newtheorem{cor}[thm]{Corollary}
\newtheorem*{cor*}{Corollary}
\newtheorem{prop}[thm]{Proposition}
\newtheorem*{prop*}{Proposition}

\theoremstyle{definition}
\newtheorem{defn}[thm]{Definition}
\newtheorem*{defn*}{Definition}

\theoremstyle{remark}
\newtheorem{remark}[thm]{Remark}

\numberwithin{equation}{section}

\title{Effective Randomness for Continuous Measures}

\author{Jan Reimann}

\address{Department of Mathematics  \\
Pennsylvania State University, University Park
}
\email{reimann@math.psu.edu}
\thanks{Reimann was partially supported by NSF grants DMS-0801270 and DMS-1201263.}

\author{Theodore A.\ Slaman}

\address{Department of Mathematics  \\
University of California, Berkeley
}
\email{slaman@math.berkeley.edu}

\thanks{Slaman was partially supported by NSF grants DMS-0501167 and DMS-1001551.}

\newcommand{\Bit}{\ensuremath{\{0,1\}}}
\newcommand{\Nat}{\ensuremath{\omega}}

\newcommand{\Rat}{\ensuremath{\mathbb{Q}}}
\newcommand{\Real}{\ensuremath{\mathbb{R}}}
\newcommand{\Cant}{\ensuremath{2^{\omega}}}

\newcommand{\Str}[1][<\omega]{\ensuremath{2^{#1}}}

\newcommand{\Sle}{\ensuremath{\subset}}
\newcommand{\Sleq}{\ensuremath{\subseteq}}

\newcommand{\Sgeq}{\ensuremath{\supseteq}}

\newcommand{\Cl}[1]{\ensuremath{#1}}

\newcommand{\Cyl}[1]{\ensuremath{\llbracket #1 \rrbracket}}
\newcommand{\ACyl}[1]{\ensuremath{\llbracket #1 \rrbracket}}

\newcommand{\Meas}{\ensuremath{\mathcal{M}(\Cant)}}

\newcommand{\ML}{Martin-L\"of}

\newcommand{\Rest}[1]{\ensuremath{ \lceil_{#1}}}
\newcommand{\Op}[1]{\ensuremath{\operatorname{#1}}}

\newcommand{\Conc}{\ensuremath{\mbox{}^\frown}}
\newcommand{\Estr}{\ensuremath{\varnothing}}
\newcommand{\Tup}[1]{\ensuremath{\langle #1 \rangle}}

\newcommand{\join}[1][\mbox{}]{\ensuremath{\oplus_{#1}}}

\newcommand{\Leb}{\ensuremath{\lambda}}

\renewcommand{\and}{\ensuremath{\:\wedge \:}}

\newcommand{\Ax}[1]{\ensuremath{\mathsf{#1}}}

\newcommand{\ZFC}{\ensuremath{\mathsf{ZFC}}}
\newcommand{\ZF}{\ensuremath{\mathsf{ZF}}}

\renewcommand{\P}{\ensuremath{\mathcal{P}}}
\newcommand{\PC}{\ensuremath{\mathcal{PC}}}

\newcommand{\bbP}{\ensuremath{\mathbb{P}}}
\newcommand{\Z}{\ensuremath{\vec{Z}}}

\newcommand{\nmodels}{\ensuremath{\nvDash}}

\newcommand{\Lb}[1]{\ensuremath{L_{\beta_{#1}}}}
\newcommand{\JA}[2][\alpha]{\ensuremath{\Tup{J_{\rho^{#2}_{#1}},A^{#2}_{#1}}}}
\newcommand{\EM}[2][\alpha]{\ensuremath{\Tup{X^{#2}_{#1},M^{#2}_{#1}}}}

\newcommand{\JP}[2][\alpha]{\ensuremath{J_{\rho^{#2}_{#1}}}}
\newcommand{\XM}{\ensuremath{\Tup{X,M}}}
\newcommand{\vecJ}{\ensuremath{\vec{J}}}
\newcommand{\Struc}[1]{\ensuremath{\Tup{F_{{#1}}, E_{{#1}}}}}

\DeclareMathOperator{\NCR}{NCR}
\DeclareMathOperator{\T}{T}
\DeclareMathOperator{\TT}{tt}

\DeclareMathOperator{\Ord}{Ord}

\DeclareMathOperator{\Def}{DEF}

\DeclareMathOperator{\Set}{Set}

\DeclareMathOperator{\Dom}{dom}
\DeclareMathOperator{\Rud}{rud}
\DeclareMathOperator{\Fld}{Field}

\begin{document}

\begin{abstract}
	We investigate which infinite binary sequences (reals) are effectively random
  with respect to some continuous (i.e., non-atomic) probability measure. We
  prove that for every $n$, all but countably many reals are $n$-random for such
  a measure, where $n$ indicates the arithmetical complexity of the Martin-Löf
  tests allowed. The proof rests upon an application of Borel determinacy. Therefore, the proof
  presupposes the existence of infinitely many iterates of the power set of the
  natural numbers. In the second part of the paper we present a metamathematical
  analysis showing that this assumption is indeed necessary. More precisely,
  there exists a computable function $G$ such that, for any $n$, the statement
  ``All but countably many reals are $G(n)$-random with respect to a continuous
  probability measure'' cannot be proved in $\ZFC^-_n$. Here $\ZFC^-_n$ stands
  for Zermelo-Fraenkel set theory with the Axiom of Choice, where the Power Set
  Axiom is replaced by the existence of $n$-many iterates of the power set of
  the natural numbers. The proof of the latter fact rests on a very general
  obstruction to randomness, namely the presence of an internal definability
  structure.
\end{abstract}

\maketitle

%===========================================================
%
% Introduction
%
\section{Introduction}~\label{sec:intro}

The goal of this paper is study under what circumstances an infinite binary
sequence (real) is random with respect to some probability measure. We use the
framework of Martin-L\"of randomness to investigate this question. Given a
measure $\mu$, a Martin-L\"of test is an effectively presented $G_\delta$
$\mu$-nullset in which the measure of the open sets converges effectively to
zero. As there are only countably many such tests, only measure-zero many reals
can be covered by a Martin-Löf test for $\mu$. The reals that cannot be covered
are called Martin-Löf random for $\mu$. Obviously, if a real $X$ is an atom of
a measure $\mu$, then $X$ is random for $\mu$. If we rule out this
trivial way of being random, the task becomes harder: Given a real $X$, does
there exist a probability measure on the space of all infinite binary sequences
such that $X$ is not an atom of $\mu$ but $X$ is $\mu$-random?

In~\cite{reimann-slaman:tams},  we were able  to show that  if a real  $X$ is
not computable, then such a measure exists. It is not hard
to see that if a real $X$ is computable, then the only way that $X$ is random
with respect to a measure $\mu$ is for it to be an atom of $\mu$.
Hence having non-trivial random content (with respect to any measure at all) in
the sense of Martin-L\"of is equivalent to being non-computable. 
Besides Martin-Löf randomness, various other notions of algorithmic
randomness have been thoroughly investigated, such as Schnorr randomness or Kurtz randomness. Two recent books on algorithmic
randomness~\cite{downey-hirschfeldt:2010,nies:2009} provide a good overview
over the various concepts. They all have in common that they use algorithmic
features to separate non-randomness from randomness. Moreover, in terms of the
arithmetic hierarchy, the complexities of the underlying test notions usually fall within two or three quantifiers of each other. 

This suggests that in order to study the random content of a real from the point of view of algorithmic randomness in general, we should look at how this content behaves when making tests more powerful by giving them access to  oracles (or equivalently, considering nullsets whose definitions are more complicated). For Martin-Löf tests, this means the test has to be effectively $G_\delta$ only in some parameter
$Z$. This enlarges the family of admissible nullsets and, correspondingly,
shrinks the set of random reals. If the parameter $Z$ is an instance
$\emptyset^{(n)}$ of the Turing jump, i.e., $Z$ is real that can decide all
$\Sigma_n$ statements about arithmetic, we speak of $(n+1)$-randomness. 

\medskip
Our goal is to understand the nature of the set of reals that are not $n$-random with respect to \emph{any continuous probability measure}. In particular, we want to understand how this set behaves as $n$ grows larger (and more reals will have this property).

The restriction to continuous measures makes sense for the following reasons.
By a result of Haken~\cite[Theorem~5]{Haken:2014a}, if a real is $n$-random, $n>2$, with
respect to some (not necessarily continuous) probability measure and not an atom of the measure, it is
$(n-2)$-random with respect to a continuous probability measure. Thus considering arbitrary probability measures would only shift the question of how random a real is by a couple of quantifiers. And the core problem of finding a measure that makes a real random without making the real an atom of the measure remains.
While we ignore features of randomness for arbitrary measures at lower levels, we develop insights into randomness for continuous measures. 
At the level of $1$-randomness, there is an interesting connection with computability theory: In~\cite{reimann-slaman:tams}, drawing on a result of Woodin~\cite{woodin:2008},
we showed that if a real $X$ is not hyperarithmetic, then there
exists a continuous probability measure for which $X$ is $1$-random.

\medskip
Our first main result concerns the size of the set of reals that are not
$n$-random with respect to any continuous measure. The case $n=1$  follows of course from the result in~\cite{reimann-slaman:tams} mentioned above.

\begin{mthm}\label{thm:first-main}
	For any $n \in \Nat$, all but countably many reals are $n$-random with respect
  to some continuous probability measure.
\end{mthm}

The proof features a metamathematical argument. Let us denote by $\NCR_n$ the
set of all reals that are not $n$-random with respect to any continuous
probability measure. We show that for each $n$, $\NCR_n$ is contained in a
countable model of a fragment of set theory. More precisely, this fragment is
$\ZFC^{-}_{n+3}$, where $\ZFC^{-}_{n+3}$ denotes the axioms of Zermelo-Fraenkel set
theory with the Axiom of Choice, with the power set axiom replaced by a
sentence that assures the existence of $n+3$ iterates of the power set of the
natural numbers.

One may wonder whether this metamathematical argument is really necessary to
prove the countability of a set of reals, in particular, whether one needs the
existence of infinitely many iterates of the power set of $\Nat$ to prove
Theorem~\ref{thm:first-main}, a result about sets of reals. It turns out that
this is indeed the case. This is the subject of our second main result.

\begin{mthm}\label{thm:second-main}
  There exists a computable function $G(n)$ such that for every $n \in \Nat$, the statement
	\begin{quote}
		``There exist only countably many reals that are not $G(n)$-random with
    respect to some continuous probability measure.''
	\end{quote}
	is not provable in $\ZFC^{-}_n$.
\end{mthm}

This metamathematical property of $\NCR$ is reminiscent of Borel
determinacy~\cite{martin:1975}. Even before Martin proved that every Borel
game is determined, Friedman~\cite{friedman:1970} had shown that any proof of
Borel determinacy had to use uncountably many iterates of the power set of
$\Nat$. Borel determinacy is a main ingredient in our proof of
Theorem~\ref{thm:first-main}. Theorem~\ref{thm:second-main} establishes that
this use is, in a certain sense, inevitable.  

Theorem~\ref{thm:second-main} is proved via a fine structure analysis of the
countable models used to show $\NCR_n$ is countable. These models are certain
levels $L_\beta$ of Gödel's constructible hierarchy. In these $L_\beta$ (or
rather Jensen's version, the $J$-hierarchy) we exhibit sequences of
non-random reals with Turing degrees cofinal among those of the model. These reals are \emph{master codes}~\cite
{boolos-putnam:1968, jensen:1972}, reals that code initial segments of the $J$-hierarchy in a way that
arithmetically reflects the strong stratification of $L$. The main feature of this proof is a very general principle that
manifests itself in various forms: an internal stratified definability structure
forms a strong obstruction to randomness. This principle works for both iterated
Turing jumps as well as certain levels of the $J$-hierarchy.

\medskip
Before we proceed, we make one more comment on the restriction to continuous measures. Note that Theorem~\ref{thm:first-main} is a stronger statement for continuous measures than for arbitrary measures. Furthermore, by Haken's result~\cite{Haken:2014a}, Theorem 2 holds for arbitrary measures if we replace $G(n)$ by $G(n)+2$.

\medskip 
The paper is organized as follows. In Section 2, we introduce
effective randomness for arbitrary (continuous) probability measures. We also
prove some fundamental facts on randomness. In particular,
we will give various ways to obtain reals that are random for \emph{some}
continuous measure from standard Martin-L\"of random reals (i.e., random with
respect to Lebesgue measure). We also consider the definability strength of
random reals.
Section 3 features the proof that for any $n$, all but countably many reals
are $n$-random with respect to some continuous measure (Theorem~\ref{thm:first-main}). 
Finally, Section 4 is devoted to the metamathematical analysis of
Theorem~\ref{thm:first-main}. In particular, it contains a proof of
Theorem~\ref{thm:second-main}.

\medskip We expect the reader to have basic knowledge in mathematical logic and
computability theory, including some familiarity with forcing, the constructible
universe, and the recursion theoretic hierarchies.

\subsection*{Acknowledgments}
We would like to thank Sherwood Hachtman, Carl Jockusch Jr., Alexander Kechris,
Donald Martin, and W.~Hugh Woodin for many helpful discussions and suggestions.
We would also like to thank the anonymous referees for their very careful
reading of the manuscript and for their much-needed suggestions on how to improve the
paper.

% \newpage

%===========================================================
%
% Randomness for Continuous Measures
%
\section{Randomness for Continuous Measures}~\label{sec-rand-cont-meas}

In this section we review effective randomness on 
Cantor space $\Cant$ for arbitrary probability
measures. We then prove some preliminary facts about random reals.

The \emph{Cantor space} $\Cant$ is the set of all infinite binary
sequences, also called \emph{reals}. The topology generated by the
\emph{cylinder sets}
\[
\Cyl{\sigma} = \{ x : \: x\Rest{|\sigma|} = \sigma \},
\]
where $\sigma$ is a finite binary sequence, turns $\Cant$ into a
compact Polish space.
$\Str$ denotes the set of all finite binary sequences. If
$\sigma, \tau \in \Str$, we use $\Sleq$ to denote the usual prefix
partial ordering. This extends in a natural way to $\Str \cup
\Cant$. Thus, $x \in \Cyl{\sigma}$ if and only if $\sigma \Sle
x$. Finally, given $U \subseteq \Str$, we write $\Cyl{U}$ to denote
the open set induced by $U$, i.e. $\Cyl{U} = \bigcup_{\sigma \in U}
\Cyl{\sigma}$.

\subsection{Turing functionals} % (fold)
\label{sub:turing_functionals}

While the concept of a Turing functional is standard, we will later define a
forcing partial order based on it, and for this purpose we give a rather
complete formal definition here. The definition
follows~\cite{shore-slaman:1999}, with the one difference that we require Turing
functionals to be recursively enumerable.

A \emph{Turing functional} $\Phi$ is a computably enumerable set of triples
$(m,k,\sigma)$ such that $m$ is a natural number, $k$ is either $0$ or $1$, and
$\sigma$ is a finite binary sequence. Further, for all $m$, for all $k_{1}$ and
$k_{2}$, and for all compatible $\sigma_{1}$ and $\sigma_{2}$, if
$(m,k_{1},\sigma_{1}%
)\in\Phi$ and $(m,k_{2},\sigma_{2})\in\Phi$, then $k_{1}=k_{2}$ and
$\sigma_{1}=\sigma_{2}$.

We will refer to a triple $(m,k,\sigma)$ as a \emph{computation} in $\Phi$, and we will say it is 
a \emph{computation along $X$} when every $X$ is an extension of $\sigma$.

In the following, we will also assume that Turing functionals $\Phi$ are
\emph{use-monotone}, which means the following hold.
\begin{enumerate}
\item  For all $(m_{1},k_{1},\sigma_{1})$ and $(m_{2},k_{2},\sigma_{2})$ in
$\Phi$, if $\sigma_{1}$ is a proper initial segment of $\sigma_{2}$, then
$m_{1}$ is less than $m_{2}$.

\item  For all $m_{1}$ and $m_{2}$, $k_{2}$ and $\sigma_{2}$, if $m_{2}>m_{1}$
and $(m_{2},k_{2},\sigma_{2})\in\Phi$, then there are $k_{1}$ and $\sigma_{1}$
such that $\sigma_{1}\Sleq\sigma_{2}$ and $(m_{1},k_{1},\sigma_{1})\in
\Phi$.
\end{enumerate}

We write $\Phi^\sigma(m)=k$ to indicate that there is a $\tau$ such that
$\tau$ is an initial segment of $\sigma$, possibly equal to $\sigma$, and
$(m,k,\tau)\in\Phi$. In this case, we also write $\Phi^\sigma(m)\downarrow$,
as opposed to $\Phi^\sigma(m) \uparrow$, indicating that for all $k$ and all
$\tau \Sleq \sigma$, $(m,k,\tau) \not\in \Phi$. If, moreover, $(m,k,\tau)$ is
enumerated into $\Phi$ by time $s$, we write $\Phi_s^\sigma(m)=k$.

If $X\in\Cant$, we write $\Phi^X(m)=k$ (and $\Phi_s^X(m)=k$, respectively) to
indicate that there is an $l$ such that $\Phi^{X\Rest{l}}(m)=k$ (and this is
enumerated by time $s$, respectively). This way, for given $X \in
\Cant$, $\Phi^X$ defines a partial function from $\omega$ to $\{0,1\}$
(identifying reals with sets of natural numbers). If this function is total, it
defines a real $Y$, and in this case we write $\Phi(X) = Y$ and say that $Y$ is
Turing reducible to $X$ via $\Phi$, $Y \leq_{\T} X$.

By use-monotonicity, if $\Phi^\sigma(m) \downarrow$, then
$\Phi^\sigma(n)\downarrow$ for all $n < m$. If we let $\overline{m}$ be maximal
such that $\Phi^\sigma(\overline{m}) \downarrow$, $\Phi^\sigma$ gives rise to a
string $\tau$ of length $\overline{m}+1$, \[ \tau = \Phi^\sigma(0) \dots
  \Phi^\sigma(\overline{m}). \] If $\Phi^\sigma(n)\uparrow$ for all $n$, we put
$\tau = \Estr$. On the other hand, if $\overline{m}$ does not exist, then
$\Phi^\sigma$ gives rise to a real $Y$. We write $\Phi(\sigma) = \tau$ or
$\Phi(\sigma) = Y$, respectively. This way a Turing functional induces a
function from $\Str$ to $\Str\cup \Cant$ that is \emph{monotone}, that is,
$\sigma \Sleq \tau$ implies $\Phi(\sigma) \Sleq \Phi(\tau)$. Note that
$\Phi(\sigma)$ is not necessarily a computable function, but we can effectively
approximate it by prefixes. More precisely, there exists a computable mapping
$(\sigma,s) \mapsto \Phi_s(\sigma)$ so that $\Phi_s(\sigma) \Sleq
\Phi_{s+1}(\sigma)$, $\Phi_s(\sigma) \Sleq \Phi_s(\sigma\Conc i)$ for $i \in
\{0,1\}$, and $\lim_s \Phi_s(\sigma) = \Phi(\sigma)$.

If, for a real $X$, $\lim_n |\Phi(X\Rest{n})| = \infty$, then $\Phi(X) = Y$,
where $Y$ is the unique real that extends all $\Phi(X\Rest{n})$. In this way,
$\Phi$ also induces a partial, continuous function from $\Cant$ to $\Cant$. We
will use the same symbol $\Phi$ for the Turing functional, the monotone function
from $\Str$ to $\Str$, and the partial, continuous function from $\Cant$ to
$\Cant$. It will be clear from the context which $\Phi$ is meant. $\Phi$ is
called \emph{total} if $\Phi(X)$ is a real for all $X \in \Cant$. If $\Phi$ is
total and $\Phi(X) = Y$, then $Y$ is called \emph{truth-table reducible} to $X$,
$Y \leq_{\TT} X$.

Turing functionals can be relativized with respect to a parameter $Z$,
by requiring that $\Phi$ is r.e.\ in $Z$. We call such functionals
\emph{Turing $Z$-functionals}. This way we can consider relativized
Turing reductions. A real $X$ is Turing reducible to a real $Y$
relative to a real $Z$, written $X \leq_{\T(Z)} Y$, if there exists a
Turing $Z$-functional $\Phi$ such that $\Phi(X) = Y$.

% -------------------------------------------------
%
% Probability measures
%
\subsection{Probability measures}

By the Carathéodory extension theorem, a Borel probability measure $\mu$ on
$\Cant$ is completely specified by its values on \emph{clopen sets}, i.e., on
finite unions of basic open cylinders. In particular, $\mu\Cyl{\Estr} = 1$, and
the additivity of $\mu$ implies that for all $\sigma \in \Str$,
\begin{equation} \label{equ-probability measure}
   \mu\Cyl{\sigma} = \mu\Cyl{\sigma \Conc 0} + \mu\Cyl{\sigma \Conc 1}.
\end{equation}
An \emph{additive premeasure} is a function
$\eta: \Str \to \Real^{\geq 0}$ with $\eta(\Estr) = 1$ and
$\eta(\sigma) = \eta(\sigma\Conc 0)+\eta(\sigma\Conc 1)$ for all
$\sigma \in \Str$. Any additive premeasure induces a Borel probability
measure, and if we restrict a Borel probability measure $\mu$ to its
values on cylinders, we obtain an additive premeasure whose
Carathéodory extension is $\mu$. We can therefore identify a Borel
probability measure on $\Cant$ with the additive premeasure it
induces. We will exclusively deal with Borel probability measures and
in the following simply write \emph{measure} to denote a Borel
probability measure on $\Cant$.

The \emph{Lebesgue measure} $\Leb$ on $\Cant$ is obtained by
distributing a unit mass uniformly along the paths of $\Cant$, i.e.,
by setting $\Leb\Cyl{\sigma} = 2^{-|\sigma|}$.
A \emph{Dirac measure}, on
the other hand, is defined by putting a unit mass on a single real,
i.e., for $X \in \Cant$, let
\[
\delta_X \Cyl{\sigma} = \begin{cases}
  1 & \text{if } \sigma \Sle X, \\
  0 & \text{otherwise.}
\end{cases}
\]
If, for a measure $\mu$ and $X \in \Cant$, $\mu(\{X\}) > 0$, then $X$
is called an \emph{atom} of $\mu$. Obviously, $X$ is an atom of
$\delta_X$. A measure that does not have any atoms is called
\emph{continuous}.

\subsection{Representation of measures and Martin-L\"of
  randomness}\label{ssec-randomness}

To incorporate measures into an effective test for randomness we represent them
as reals. This can be done in various ways (for example, identify them with the
underlying premeasure and code that), but in order for the main arguments
in~\cite{reimann-slaman:tams} to work, the representation has to reflect some of
the topological properties of the space of probability measures.

Let $\Meas$ be the set of all Borel probability measures on $\Cant$. With the
weak-* topology, this becomes a compact Polish space (see~\cite[Theorem
17.23]{kechris:1995}). It is possible to choose a countable dense subset
$\mathcal{D}\subseteq \Meas$ so that every measure in $\Meas$ is the limit of an
effectively converging Cauchy sequence of measures in $\mathcal{D}$. Moreover,
the structure of the measures in $\mathcal{D}$ is such that they give rise to a
canonical continuous surjection $\rho: \Cant \to \Meas$ with the additional
property that for every $R \in \Cant$, $\rho^{-1}(\{\rho(R)\})$
is $\Pi^0_1(R)$. For details on the construction of $\rho$,
see~\cite[Section 2]{Day-Miller:2013a}. If
$\rho(R) = \mu$, then $R$ is called a \emph{representation} of $\mu$. A measure
may have several distinct representations with respect to $\rho$. If
$\mu$ is given, $R_\mu$ will always denote a representation of $\mu$. 

Working with representations, we can apply computability theoretic notions to
measures. The following two observations appeared as Propositions~2.2 and 2.3,
respectively, in~\cite{reimann-slaman:tams}.

\begin{prop} \label{prop:repres-relation}
  Let $R \in \Cant$ be a representation of a measure $\mu \in \Meas$. Then the relations
  \[
    \mu\Cyl{\sigma} < q \quad \text{ and } \quad \mu\Cyl{\sigma} > q \qquad
    (\sigma \in \Str, q \in \Rat)
  \] 
  are r.e.\ in $R$.
\end{prop}

It follows that the representation of a measure can effectively approximate
its values on cylinders to arbitrary
precision.

\begin{prop} \label{prop:computation-measure}
  Let $R \in \Cant$ be a representation of a measure $\mu \in \Meas$. Then $R$
  computes a function $g_\mu : \Str \times \Nat \to \Rat$ such that for all
  $\sigma \in \Str$, $n \in \Nat$,
$$
  |g_\mu(\sigma,n) - \mu\Cyl{\sigma}| \leq 2^{-n}.
$$
\end{prop}

We say a real $X$ is \emph{recursive in $\mu$} if $X \leq_{\T} R_\mu$
for every representation $R_\mu$ of $\mu$. On the other hand, we say a real
computes a measure if its computes \emph{some} representation of it.
A measure does not necessarily have a representation of least
Turing degree~\cite[Theorem 4.2]{Day-Miller:2013a}.

We will later show that the question of whether a real is random with respect to
a continuous measure can be reduced to considering only continuous dyadic
measures. A measure $\mu$ is \emph{dyadic} if every measure of a cylinder is of the form
$\mu\Cyl{\sigma} = m/2^n$ with $m,n$ non-negative integers.

For dyadic measures,
it makes sense to speak of \emph{exact computability}: A dyadic measure $\mu$ is
\emph{exactly computable} if the function $\sigma \mapsto \mu\Cyl{\sigma}$ is a
computable mapping from $\Str$ to $\Rat$. Note that for exactly computable
measures, the relation $\mu\Cyl{\sigma} > q$, $q$ rational, is
decidable, whereas in the general case for $\mu$ with a computable
representation we only know it is
$\Sigma^0_1$.

If we encode a dyadic measure $\mu$ by collecting the ternary expansions
of its values on cylinders in a single real $Y$, we obtain a representation not in
the sense of $\rho$, but that is minimal in the following sense: Any real that
can compute an approximation function to $\mu$ in the sense of
Proposition~\ref{prop:computation-measure} can compute $Y$.

\medskip We can now give the definition of a general \ML\ test. The definition
is a generalization of Martin-Löf $n$-tests and Martin-Löf 
$n$-randomness for Lebesgue measure. We relativize both with respect to a
representation of the measure and an additional parameter.

\begin{defn}\label{def-test-relative-to-measure}
  Suppose $\mu$ is a probability measure on $\Cant$, and $R$ is a representation
  of $\mu$. Suppose further that $Z \in \Cant$ and $n\geq 1$.
\begin{enumerate}[(1)]
	\item An \emph{$(R,Z,n)$-test} is a set
	  $W \subseteq \Nat \times \Str$ which is recursively enumerable in
	  $(R\oplus Z)^{(n-1)}$, the $(n-1)$st Turing jump of $R\oplus Z$ such that
	  \begin{equation*} \label{equ-correct-test}
	    \sum_{\sigma \in W_n} \mu\Cyl{\sigma} \leq 2^{-n},
	  \end{equation*}
	where $W_n = \{ \sigma : \: (n, \sigma) \in W \}$

	\item A real $X$ \emph{passes} a test $W$ if $X \not\in
  \bigcap_n \ACyl{W_n}$. If $X$ does not pass a test $W$, we also say $X$ is
  \emph{covered} by $W$ (or $(W_n)$, respectively).

	\item A real $X$ is
	   \emph{$(R, Z, n)$-random} if it passes all $(R,Z,n)$-tests.

   \item \label{def:mu-random} A real $X$ is \emph{Martin-L\"of $n$-random for
       $\mu$ relative to $Z$}, or simply \emph{$(\mu, Z, n)$-random} if there
     exists a representation $R_\mu$ such that $X$ is $(R_\mu,Z,n)$-random. In
     this case we say $R_\mu$ \emph{witnesses} the $\mu$-randomness of $X$.
\end{enumerate}
\end{defn}

If the underlying measure is Lebesgue measure $\Leb$, we often drop reference to
the measure and simply say $X$ is \emph{$(Z,n)$-random}. We also drop the index
$1$ in case of $(\mu, Z,1)$-randomness and simply speak of
\emph{$\mu$-randomness relative to $Z$} or \emph{$\mu$-$Z$-randomness}. If $Z =
\emptyset$, on the other hand, we speak of $(\mu,n)$- or
$\mu$-$n$-\emph{randomness}. Note also that if $\mu$ is $Z$-computable, say
$R_\mu \leq_{\T} Z$, then $(R_\mu,Z,n)$-randomness is the same as
$(R_\mu,Z^{(n-1)},1)$-randomness.

\begin{remark}
  The original definition of $n$-randomness for Lebesgue measure $\lambda$ given
by Kurtz~\cite{kurtz:1981} uses tests based on $\Sigma_n$ classes. However, it
is possible to approximate $\Sigma^0_n$ classes from outside in measure by open
sets. Kurtz~\cite{kurtz:1981} and Kautz~\cite{kautz:1991} showed that such an
approximation in measure can be done effectively for classes of the lightface
finite Borel hierarchy, in the sense that a $\Sigma^0_n$ class can be
approximated in measure by a $\Sigma^{0,\emptyset^{(n-1)}}_1$ class. Therefore, while
the definitions based on $\Sigma^0_n$ nullsets and on $\Sigma^{0,\emptyset^{(n-1)}}_1$
nullsets do not give the same notion of test, they yield the same class of
random reals (see~\cite[Section~6.8]{downey-hirschfeldt:2010} for a complete
presentation of this argument).

The proof that the two approaches yield the same notion of $n$-randomness
relativizes. Moreover, the approximation in measure by open sets is possible for
any Borel probability measure, as any finite Borel measure on a metric space is
regular. Finally, using Proposition~\ref{prop:repres-relation} one can show
inductively that, given a representation $R$ of $\mu$, the relations $\mu(S) >
q$ and $\mu(S) < q$ (for $q$ rational) are uniformly $\Sigma^{0,R}_n$ for any
$\Sigma^0_n$ class $S$. The latter fact is a key ingredient in the equivalence
proof. Therefore, $(\mu,Z,n)$-randomness could alternatively be defined via
$\Sigma^{0,R\oplus Z}_n$ tests. We prefer the approach given in
Definition~\ref{def-test-relative-to-measure}, because open sets are usually
easier to work with, and because most techniques relativize.
\end{remark}

\medskip Levin~\cite{levin:1976} introduced the alternative concept of a
\emph{uniform test} for randomness, which is representation-independent (see
also~\cite{gacs:2005}). Day and Miller~\cite[Theorem 1.6]{Day-Miller:2013a} have
shown that for any measure $\mu$ and for any real $X$, $X$ is $\mu$-random in
the sense of Definition~\ref{def-test-relative-to-measure} if and only if $X$ is
$\mu$-random for uniform tests.

\medskip Since, for fixed $R_\mu$, $Z$, and $n$, there are only countably many $(R_\mu,Z,n)$-tests, it follows
from countable additivity that the set of $(\mu,Z,n)$-random reals
for any $\mu$ and any $Z$ has $\mu$-measure $1$. Hence there always exist
$(\mu,Z,n)$-random reals for any measure $\mu$, any real $Z$, and any $n \geq
1$.

However, $\mu$, $Z$, and $n$ put some immediate restrictions on the relative
definability of any $(\mu,Z,n)$-random real.

\begin{prop}\label{pro:random-not-delta}
	If $\mu$ is a continuous measure, $X$ is $(\mu, Z,n)$-random via a representation $R_\mu$, then $X$ cannot be
  $\Delta^0_{n}(R_\mu \join Z)$.
\end{prop}

\begin{proof}
	If $X$ is $\Delta^0_{n}(R_\mu \join Z)$, then $X \leq_{\T} (R_\mu \join Z)^{(n-1)}$, and we can build a $(\mu,Z,n)$-test covering $X$ by using the cylinders given by its initial segments.
\end{proof}

\medskip It is also immediate from the definition of randomness that any atom of
a measure is random with respect to it. This is a trivial way for a real to be
random. The proposition below (a straightforward relativization of a result by
Levin~\cite{zvonkin-levin:1970}, see also~\cite[Proposition
3.3]{reimann-slaman:tams}) shows that atoms of a measure are also
computationally trivial (relative to the measure).

\begin{prop}[Levin]
	If for a measure $\mu$ and a real $X$, $\mu \{X\} > 0$, then $X \leq_{\T}
  R_\mu$ for any representation $R_\mu$ of $\mu$.
\end{prop}

Since we are interested in randomness for continuous measures, the case of
atomic randomness is excluded a priori.

% -------------------------------------------------
%
% Transformations, image measures, and semimeasures
%
\subsection{Image measures and transformation of
  randomness}\label{ssec-trans-meas}

Let $f: \Cant \to
\Cant$ be a Borel measurable function. If $\mu$ is a measure on $\Cant$, the
\emph{image measure} $\mu_f$ is defined by
\[
\mu_f(A) = \mu(f^{-1}(A)).
\]
It can be shown that every probability measure is the image measure of Lebesgue
measure $\Leb$ for some $f$. Oxtoby~\cite[Theorem 1]{oxtoby:1970} proved that any continuous, positive probability
measure on $2^\omega$ can be transformed into Lebesgue measure on $2^\omega$ via a homeomorphism. Here, a measure is \emph{positive}
if $\mu\Cyl{\sigma} > 0$ for $\sigma \in \Str$.

Levin~ (see~\cite[Theorem 4.3]{zvonkin-levin:1970}), and independently
Kautz~\cite[Corollary IV.3.18]{kautz:1991} (see
also~\cite{Bienvenu-Porter:2012a}) proved effective analogs of these results.
For a computable measure $\mu$ on $\Cant$ there exists a Turing
functional $\Phi$ defined on almost every real such that $\mu$ is the image
measure of $\Leb$ under $\Phi$. If $\mu$ is, moreover,
continuous and positive, then $\Phi$ has an inverse that transforms $\mu$ into
$\lambda$.

A consequence of the Levin-Kautz theorem is that every non-recursive real that
is random with respect to a computable probability measure is Turing equivalent
to a $\lambda$-random real. We will show now that for continuous measures, this
can be strengthened to \emph{truth-table equivalence}.

\begin{prop}\label{prop:char-cont-rand}
	Let $X$ be a real. For any $Z \in \Cant$ and any $n \geq 1$, the following are equivalent.
	\begin{enumerate}[(i)]
		\item $X$ is $(\mu,Z,n)$-random for a continuous measure $\mu$ recursive in $Z$.
		\item $X$ is $(\nu,Z,n)$-random for a continuous, positive, dyadic measure
      $\nu$ exactly computable in $Z$.
		\item There exists a Turing $Z$-functional $\Phi$ such that $\Phi$ is an
      order-preserving homeomorphism of $\Cant$, and $\Phi(X)$ is
      $(\lambda,Z,n)$-random.
		\item $X$ is truth-table equivalent relative to $Z$ to a $(\lambda,Z,n)$-random real.
   \end{enumerate}
\end{prop}

Here, the order on $\Cant$ is the \emph{lexicographical order} given by
\[
	X < Y \quad : \Leftrightarrow \quad X(N) < Y(N) \; \text{ where } N = \min\{n \colon X(n) \neq Y(n)\}.
\]

\begin{proof} We give a proof for $Z = \emptyset$ and $n=1$. It is routine to check that the proof relativizes and generalizes to higher levels of randomness.

	(i) $\Rightarrow$ (ii): Let $X$ be $\mu$-random, where $\mu$ is a continuous,
  computable measure. We construct a continuous, positive, dyadic, and exactly
  computable measure $\nu$ such that $X$ is random with respect to $\nu$, too.
  The construction is similar to Schnorr's rationalization of
  martingales~\cite{schnorr:1971} (see
  also~\cite[Proposition~7.1.2]{downey-hirschfeldt:2010}).

  We define $\nu$ by
  recursion on the full binary tree $\Str$. To initialize, let
  $\nu^*\Cyl{\Estr}=2$. Now assume $\nu^* \Cyl{\sigma}$ is defined such that,
  \[
    \mu \Cyl{\sigma} < \nu^* \Cyl{\sigma} < \mu\Cyl{\sigma} + 2^{-|\sigma|+1}.
  \]
A simple case
  distinction shows that
  \begin{equation*}
    \max\{\mu\Cyl{\sigma\Conc 0},\nu^*\Cyl{\sigma}-\mu\Cyl{\sigma\Conc 1} -2^{-|\sigma|}\} < \min\{\mu\Cyl{\sigma\Conc 0} + 2^{-|\sigma|},\nu^*\Cyl{\sigma}-\mu\Cyl{\sigma\Conc 1}  \}.
  \end{equation*}
  As the dyadic rationals are dense in $\Real$, there exists a dyadic rational
  $r$ in this interval, and by Proposition~\ref{prop:computation-measure} we can
  find such an $r$ effectively in $\sigma$. Put
  \begin{equation*}
    \nu^*\Cyl{\sigma\Conc 0} = r, \quad \nu^*\Cyl{\sigma\Conc 1} = \nu^*\Cyl{\sigma}-r.
  \end{equation*}
  Then clearly
  \begin{equation*}
    \nu^*\Cyl{\sigma\Conc 0} + \nu^*\Cyl{\sigma\Conc 1} = \nu^*\Cyl{\sigma},
  \end{equation*}
  and by the choice of $r$,
  \begin{align*}
		& \mu \Cyl{\sigma \Conc 0} < \nu^*\Cyl{\sigma\Conc 0} < \mu\Cyl{\sigma \Conc
     0} + 2^{-(|\sigma|)}, \\
    & \mu \Cyl{\sigma \Conc 1} <
    \nu^*\Cyl{\sigma\Conc 1} < \mu\Cyl{\sigma \Conc 1} + 2^{-(|\sigma|)}.
  \end{align*}
  We normalize by letting $\nu = \nu^*/2$. By construction of $\nu^*$, the
  measure $\nu$ is dyadic and exactly computable. It is also clear from
  the construction that for all $\sigma$, $\mu\Cyl{\sigma} < 2\nu\Cyl{\sigma}$.
  In particular, $\nu$ is positive. Finally, if $(V_n)$ is a test for $\nu$, by
  letting $W_n = V_{n+1}$ we obtain a $\mu$-test that covers every real covered
  by $(V_n)$. Hence if $X$ is $\mu$-random, then $X$ is also $\nu$-random.

 	(ii) $\Rightarrow$ (iii): Suppose $\nu$ is an exactly computable, continuous,
  positive, dyadic measure. Since $\nu$ is continuous and $\Cant$ is compact,
  for every $m$ there exists a least $l_m \in \Nat$ such that whenever $|\sigma|
  \geq l_m$, then $\nu\Cyl{\sigma} \leq 2^{-m}$. Without loss of generality, we
  can assume that $l_m < l_{m+1}$. As $\nu$ is exactly computable, the mapping
  $m \mapsto l_m$ is computable.
	
  We define inductively a mapping $\varphi: \Str \to \Str$ that will induce the
  desired homeomorphism. In order to do so, we first define, for every $\tau \in
  \Str$, an auxiliary finite, non-empty set $E_\tau \subseteq \Str$. It will
  hold that
  \begin{enumerate}[(a)]

  \item all strings in $E_\tau$ are of  length $l_{2n}$, where $n$ is the length of $\tau$;

  \item if $\sigma \Sleq \tau$, then every string in $E_{\tau}$ is an
    extension of some string in $E_{\sigma}$;
      
  \item if $\sigma$ and $\tau$ are incomparable, then $E_\sigma$ and $E_\tau$
    are disjoint; moreover, if $|\sigma| = |\tau|$ and $\sigma$ is
    lexicographically less than $\tau$, then all strings in $E_\sigma$ are
    lexicographically less than any string in $E_\tau$;
      
  \item for all $n$, $\bigcup_{|\tau| = n} \Cyl{E_\tau} = \Cant$;

  \item for all $\tau$, $0 < \nu\Cyl{E_\tau} \leq 2^{-|\tau|} ( 2-2^{-|\tau|})$.

  \end{enumerate}

  Put $E_{\Estr} = \{\Estr\}$. Suppose now that $E_\tau$ is defined for all
  strings $\tau$ of length at most $n$, and that for these sets $E_\tau$, (a)-(e)
  are satisfied.
  
  Given any $\tau$ of length $n$, let
  \[
    F_\tau = \{ \sigma: |\sigma| = l_{2(n+1)} \; \& \; \sigma \text{
      extends some string in $E_\tau$} \}.
  \]
  Find the least (with respect to the usual lexicographic ordering) $\sigma \in
  F_\tau$ such that
  \[
    \sum_{\substack{\eta \leq \sigma \\ \eta \in F_\tau}} \nu\Cyl{\eta} \geq
    \nu\Cyl{E_\tau}/2.
  \]
  Let $E_{\tau\Conc 0} = \{ \eta \in F_\tau \colon \eta < \sigma \}$ and put the
  remaining strings of $F_\tau$ into $E_{\tau\Conc 1}$. This ensures that
  $E_{\tau\Conc 0}$ and $E_{\tau\Conc 1}$ satisfy (a), (b), (c), and (d).
  Moreover, by the choice of the length of strings in $F_\tau$ and property
  (e) for $E_\tau$, both $E_{\tau\Conc 0}$ and $E_{\tau\Conc 1}$ are non-empty.
  As $\nu$ is positive, this implies $\nu\Cyl{E_{\tau\Conc i}} > 0$ for each $i
  \in \Bit$. Moreover, for each $i \in \Bit$, we can use the induction
  hypothesis for $E_\tau$ and deduce that
  \begin{align*}
    \nu \Cyl{E_{\tau\Conc i}} & \leq \nu\Cyl{E_\tau}/2 + 2^{-2|\tau|-1} \\
                              & \leq 2^{-|\tau|-1}(2 - 2^{-|\tau|}) + 2^{2(|\tau|+1)} \\
                              & =  2^{-|\tau|-1}(2 - 2^{-|\tau|-1}),
  \end{align*}
  which yields the bound in (d).

  Now we define the mapping $\varphi$: Put $\varphi(\Estr) = \Estr$. Suppose now
  $\varphi(\sigma)$ is defined for all $\tau$ of length less than or equal to
  $n$. Given $\tau$ of length $n$, map all strings in $E_{\tau\Conc 0}$ to $\tau
  \Conc 0$, and all strings in $E_{\tau\Conc 1}$ to $\tau \Conc 1$. To make
  $\varphi$ defined on all strings, map any string that extends some string in
  $E_\tau$ but is a true prefix of some string in $F_\tau$ to $\tau$.

  It is clear from the construction that $\varphi$ induces a total, order
  preserving mapping $\Phi: \Cant \to \Cant$ by letting
  \[
  	\Phi(X) = \lim_n \varphi(X\Rest{n}).
  \]
  $\Phi$ is onto since for every $\sigma$, $E_\sigma$ is not empty. We claim
  that $\Phi$ is also one-to-one. Suppose $\Phi(X) = \Phi(Y)$. This implies that
  for all $n$, $\varphi(X\Rest{n}) = \varphi(Y\Rest{n})$, that is, for all $n$,
  $X\Rest{n}$ and $Y\Rest{n}$ belong to the same $E_\sigma$. Since $\nu$ is
  positive, the diameter of the $E_\sigma$ goes to $0$ along any path. Hence
  $X=Y$.

  It remains to show that $\Phi(X)$ is Martin-L\"of random. Suppose not, then
  there exists an $\Leb$-test $(W_n)$ that covers $\Phi(X)$. Let
  \[
  	V_n = \bigcup_{\sigma \in W_{n+2}} E_\sigma.
  \]
  Then $(V_n)$ covers $X$. Furthermore, the $(V_n)$ are uniformly enumerable
  since the mapping $\sigma\mapsto E_\sigma$ is computable by the construction
  of the $E_\sigma$. Finally,
  \[
  	\sum_{\tau \in V_n} \nu\Cyl{\tau} = \sum_{\sigma \in W_{n+2}}
    \nu\Cyl{E_\sigma} \leq \sum_{\sigma \in W_{n+2}} 2^{-|\sigma|+2} \leq
    2^{-n}.
  \]
  thus $X$ is not $\nu$-random, contradiction.

	(iii) $\Rightarrow$ (iv): This is immediate.

	 (iv) $\Rightarrow$ (i): This follows from Theorem 5.7 in~\cite{reimann-slaman:tams}.

\end{proof}

The result also suggests that if we are only interested in whether a real is
random with respect to a continuous measure, representational issues do not
really arise. We can restrict ourselves to dyadic measures, which have a minimal
representation.

\begin{remark} \label{rem:dyadic} We will henceforth, unless explicitly noted,
  assume that \emph{any measure is a dyadic measure}. With respect to effectiveness considerations, we identify $R_\mu$ with $\mu$  and write $\mu'$ instead of $R'_\mu$ and so forth.
\end{remark}

\subsection{Continuous randomness via Turing reductions}

While Proposition~\ref{prop:char-cont-rand} gives a necessary and sufficient
criterion for reals being random for a continuous measure, we will later need
further techniques to show that a given real is random with respect to a
continuous measure. As many of our arguments will involve arithmetic
definability, it will be helpful to know to what extent
randomness for continuous measures can be ``transfered'' via Turing reductions
instead of truth-table reductions.
The key ingredients are a theorem by Demuth~\cite{demuth:1988} and a result by
Kurtz~\cite{kurtz:1981}.

Demuth~\cite[Theorem~17]{demuth:1988} showed that every non-recursive real
truth-table below a Martin-L\"of random real measure is Turing equivalent to a
Martin-L\"of-random real. The proof relativizes (as can be seen from the
presentation in~\cite[Theorems~6.12.9 and 8.6.1]{downey-hirschfeldt:2010}) and
yields the next proposition.

Recall that we only consider dyadic measures and hence drop reference to a
representation. Nevertheless, the results in this section are not dependent on
the existence of a minimal representation and can be reformulated accordingly.

\begin{prop}[Demuth] \label{prop:relative-demuth} Suppose $Y$ is $(\mu,Z,n)$-random $(n
  \geq 1)$ and $X$ is truth-table reducible to $Y$ relative to $(\mu \join
  Z)^{(k)}$ for some $k \leq n-1$ $($i.e., $X \leq_{\TT((\mu\join Z)^{(k)})}
  Y$$)$. Further suppose $X$ is not recursive in $(\mu \join Z)^{(k)}$. Then $X$
  is Turing equivalent relative to $(\mu \join Z)^{(k)}$ to a $(\lambda,\mu
  \join Z, n)$-random real.
\end{prop}

Kurtz~\cite[Theorem~4.3]{kurtz:1981} observed that $2$-random reals are
$\emptyset'$-dominated. More precisely, there exists a $\emptyset'$-computable
function dominating every function computable from a $2$-random real.

The proof is based on the following idea
(see~\cite[Proposition~5.6.28]{nies:2009}): Given a Turing functional $\Phi$,
$\emptyset'$ can decide, given rational $q$ and $n \in \Nat$, whether
\[
	\lambda \: \{Y \colon \text{$\Phi(Y)(k)$ is defined for all $k \leq n$} \} > q.
\]
For each $n$, let $q_n$ be maximal of the form $i\cdot 2^{-n}$ so that the above
holds, and let $t_n$ be such that $\Phi$ converges on at least measure
$q_n$-many strings of length $t_n$ by time $t_n$. Construct a function $f
\leq_{\T} \emptyset'$ such that $f(n)$ dominates all function values $\Phi(Y)$
computed with use $t_n$ and within $t_n$ steps. Then the set of all $Y$ for
which $\Phi(Y)$ is total and not dominated by $f$ has Lebesgue measure $0$ and
can be captured
by a $\emptyset'$-\ML-test.
The argument relativizes to other measures and parameters,
and we obtain the following.

\begin{prop}[Kurtz] \label{prop-ae-unif-dominating} Given a measure $\mu$ and a real
  $Z$, there exists a function $f \leq_{\T} (\mu\oplus Z)'$ such that for every
  $(\mu,Z,2)$-random $X$, if $g \leq_{\T(Z\oplus \mu)} X$, then $g$ is dominated
  by $f$.
\end{prop}

\medskip Together with Proposition~\ref{prop:relative-demuth} this yields a
sufficient criterion for continuous randomness.

\begin{lem} \label{prop:rand-via-turing} Suppose $n \geq 3$ and $Y$ is
  $(\mu,Z,n)$-random. If $X \leq_{\T(\mu\oplus Z)} Y$ and $X \nleq_{\T} (\mu \oplus Z)^{'}$, then $X$ is $(\nu,(\mu\oplus Z)'',n-2)$-random for
  some continuous measure $\nu \leq_{\T} (\mu\join Z)^{''}$.
\end{lem}

\begin{proof}
  We assume $Z = \emptyset$ to keep notation simple. 
  Suppose $X \leq_{\T(\mu)}
  Y$ via a Turing reduction $\Phi$. By
  Proposition~\ref{prop-ae-unif-dominating}, the use and the convergence time of
  $\Phi$ on $Y$ are dominated by some function recursive in $\mu'$. We can
  modify $\Phi$ to $\widetilde{\Phi}$ such that $\widetilde{\Phi}$ is a
  truth-table reduction relative to $\mu'$ and $\widetilde{\Phi}(Y) = X$.

	By Proposition~\ref{prop:relative-demuth}, $X$ is Turing equivalent relative
  to $\mu'$ to a $(\lambda,\mu, n)$-random real $R$. Any
  $(\lambda,\mu, n)$-random real is also $(\lambda,\mu',
  n-1)$-random, and so we can apply Proposition~\ref{prop-ae-unif-dominating} to
  $X$ and $R$ to conclude that they are truth-table equivalent relative to
  $\mu''$. This in turn means that $X$ is truth-table equivalent
  relative to $\mu''$ to a $(\lambda,\mu'', n-2)$-random
  real, which by Proposition~\ref{prop:char-cont-rand} implies that $X$ is
  $(\nu, \mu'', n-2)$-random for a continuous measure recursive in
  $\mu''$.
\end{proof}

\subsection{The definability strength of randomness}

Lemma~\ref{prop:rand-via-turing} shows that sufficiently high randomness
for continuous measures propagates downward under Turing reductions (losing some
of the randomness strength, however). This result was partly based on the fact
that, for $n \geq 2$, $n$-random reals cannot compute fast-growing functions
(beyond what is computable by $\emptyset'$). There is further evidence that the
computational strength of $n$-random reals is rather limited.

For example, random reals are generalized low (relative to the measure). This is
a generalization of a result due to Kautz~\cite[Theorem~III.2.1]{kautz:1991}.

\begin{prop}[Kautz] \label{pro:random-join-jump} Let $\mu$ be a continuous measure, and
  suppose $X$ is $\mu$-$(n+1)$-random, where $n \geq 1$. Then
	\[
		(X \join \mu)^{(n)} \equiv_{\T} X \join \mu^{(n)}
	\]
\end{prop}

The generalization works for the same reasons that $n$-randomness can be defined
equivalently in terms of $\Sigma^{0,\mu}_n$-tests or
$\Sigma^{0,\mu^{(n-1)}}_1$-tests: Borel probability measures are regular, and if $S$ is a $\Sigma^0_n$ class,
the relations $\mu(S) > q$ and $\mu(S) < q$ (for $q$ rational) are uniformly
$\Sigma^{0,\mu}_n$ and $\Pi^{0,\mu}_n$, respectively. 

Furthermore, one can generalize a result of Downey, Nies, Weber, and Yu~\cite
{downey-nies-weber-yu:2006}, who show that every weakly $2$-random real forms a
minimal pair with $0'$. This will be of central importance in Section~\ref{sec-meta}.
For our purposes, it suffices to consider randomness instead of weak
randomness, which we do in the following lemma.

\begin{lem} \label{prop-non-accel} 
Suppose $\mu$ is a continuous measure and $Y$
  is $\mu$-$n$-random, $n \geq 2$. If $X \leq_{\T} \mu^{(n-1)}$ and $X \leq_{\T}
  Y \oplus \mu$, then $X \leq_{\T} \mu$.
\end{lem}

The structure of the proof is as follows: Following Downey, Nies, Weber, and Yu, we first show that the upper cone by $\Phi$ is $\Pi^0_2$
  (relative to $\mu^{(n-2)}$). Next, we argue that the upper cone cannot
  have measure zero since it contains a random real. Finally, one uses this fact to isolate $X$ as a path in a $\mu$-r.e.\ tree. The last step is a generalized version of the result that if the Turing upper cone of a real has positive Lebesgue measure, then the real must be computable \cite{Leeuw:1956a,Sacks:1963a}. Our presentation follows~\cite{nies:2009}.

\begin{proof}[Proof of Lemma~\ref{prop-non-accel}]
Suppose $X \leq_{\T} Y \oplus \mu$ via a Turing functional $\Phi$ and $Y
\leq_{\T} \mu^{(n-1)}$. Note that $Y$ is $\Delta^0_2$ relative to $\mu^{(n-2)}$. Let
$Y(n,s)$ be a $\mu^{(n-2)}$-recursive approximation of $Y$, i.e., $\lim_s Y(n,s)
= Y(n)$. Given $i,s \in \Nat$, put
\[
	U_{i,s} = \{ X \colon \exists t > s \; \bigl( \Phi_t^{X\oplus \mu}(i) = Y(i,t)  \bigr)\}.
\]
The set $U_{i,s}$ is $\Sigma^{0,\mu^{(n-2)}}_1$ uniformly in $i,s$ and hence $P
= \bigcap_{i,s} U_{i,s}$ is $\Pi^{0,\mu^{(n-2)}}_2$. Note that $P$ is the upper
cone of $X$ under $\Phi$,
\begin{equation*}
  P = \{ A \colon \Phi(A) = X\}.
\end{equation*}
$P$ cannot have $\mu$-measure $0$: If it had then,
since Borel probability measures on $\Cant$ are regular, for the sequence of
open sets $(V_k)_{k \in \Nat}$ given by $V_k = \bigcap_{\Tup{i,s} \leq k}
U_{i,s}$, we have $\mu V_k \searrow 0$. Since each $V_k$ is
$\Sigma^{0,\mu^{(n-2)}}_1$, $\mu^{(n-1)}$ can decide whether $\mu V_k \leq
2^{-l}$ for given $l$. Hence, we can convert $(V_k)$ into a $(\mu,n)$-test.
Since $\bigcap_k V_k = P$ and $P$ contains $Y$, this contradicts the fact that
$Y$ is $\mu$-$n$-random.

Hence pick $r$ rational such that $\mu P > r >0$, where $r$ is rational. Define
a tree $T$ by letting
\[
	\sigma \in T \; : \Leftrightarrow \; \mu \{\tau \colon \Phi(\tau\oplus \mu)
  \Sgeq \sigma \} > r,
\]
and closing under initial segments. $T$ is r.e.\ in $\mu$ and $X$ is an infinite
path through $T$.

Since $\mu$ is a probability measure, any antichain in $T$ contains at most $\lceil 1/r\rceil$
strings. Choose $\sigma = X\Rest{n}$ such that no $\tau \Sgeq
\sigma$ incompatible with $X$ is in $T$. Such $\sigma$ exists for otherwise we could find an antichain of more than $\lceil 1/r\rceil$ strings branching off $X$. To compute $X\Rest{m}$ from $\mu$, it
suffices to enumerate $T$ above $\sigma$ until a long enough extension shows up.
\end{proof}

We will later need the following relativization of the previous lemma. The proof is similar.

\begin{lem} \label{lem:stair-with-jump}
	Suppose $\mu$ is a continuous measure and $Y$ is $\mu$-$(k+n)$-random, $k \geq 0, n \geq  2$. If $X \leq_{\T} \mu^{(k+n-1)}$ and $X \leq_{\T} Y \oplus \mu^{(k)}$, then $X \leq_{\T} \mu^{(k)}$.
\end{lem}

One interpretation of Lemmas~\ref{prop-non-accel} and \ref{lem:stair-with-jump} is that $\mu$-random reals are not
helpful in computing (defining) reals arithmetic in $\mu$. For example, if a real is
properly $\Delta^0_n$ relative to a measure $\mu$, then it cannot be
$\Delta^0_k$ relative to $\mu\oplus Y$ where $k < n$ and $Y$ is
$\mu$-$(n+1)$-random. 

In Section~\ref{sec-meta}, we will also need a result similar to the previous lemmas regarding initial segments of linear orders, namely, that random reals are not helpful in the recognizing well-founded initial segments. The following lemma may appear technical at this point, but its importance will become clear towards the end of Section~\ref{sec-meta}, in the proof of Theorem~\ref{thm:mastercodes-in-NCR}. We use it to separate, in a definable way, well-founded from non-well-founded models.  

\begin{lem} \label{lem:recog-wf-segment}
  Let $j \geq 0$.  Suppose $\mu$ is a continuous measure and $\prec$ is a linear order on a subset of $\omega$ such that the relation $\prec$ and the field of $\prec$ are both recursive in $\mu^{(j)}$. Suppose further $Y$ is $(j+5)$-random relative to $\mu$, and $I \subseteq \omega$ is the longest well-founded initial segment of $\prec$. If $I$ is recursive in $(Y \join \mu)^{(j)}$, then $I$ is recursive in $\mu^{(j+4)}$.
\end{lem}

\begin{proof}
  Suppose $I \leq_{\T} (Y \join \mu)^{(j)}$, $Y$ is $(j+5)$-random relative to $\mu$, but $I \nleq_{\T} \mu^{(j+4)}$. By Lemma~\ref{prop:rand-via-turing} (with $Z = \mu^{(j)}$, $n=5$, and $X=I$), there is a continuous measure $\mu_I \leq_{\T} \mu^{(j+2)}$ such that $I$ is $(\mu_I, \mu^{(j+2)}, 3)$-random.

  For given $a \in \Op{Field}(\prec)$, let $\mathcal{I}(a)$ be the set of all reals $A \subseteq \Nat$ such that $A$ is an initial segment of $\prec$, and all elements of $A$ are bounded by $a$. $\mathcal{I}(a)$ is a $\Pi^0_1(\mu^{(j)})$ class. Let $T_a$ be a tree recursive in $\mu^{(j)}$ such that $[T_a] = \mathcal{I}(a)$. Given $n \in \Nat$, let $T_a\Rest{n} = \{\sigma \in T_a \colon |\sigma| = n\}$ be the $n$-th level of $T_a$. We have $\mathcal{I}(a) = \bigcap_n \Cyl{T_a\Rest{n}}$.

  Now, if $a \in I$, then $\mathcal{I}(a)$ is countable (since in this case each element of $\mathcal{I}(a)$ is an initial segment of the well-founded part of $\prec$ and there are at most countably many such initial segments). Since $\mu_I$ is continuous, it follows that $\mathcal{I}(a)$ has $\mu_I$-measure zero.

  If, on the other hand, $a \not\in I$, then $I \in \mathcal{I}(a)$. Since $I$ is $(\mu_I, \mu^{(j+2)}, 3)$-random and  $\mathcal{I}(a)$ is $\Pi^0_1(\mu^{(j)})$,  $\mathcal{I}(a)$ does not have $\mu_I$-measure zero: Otherwise we could recursively in $\mu^{(j+2)}$, compute a sequence $(l_n)$ such that $\mu_I \Cyl{T_a\Rest{l_n}} \leq 2^{-n}$. This would be a $(\mu_I,\mu^{(j+2)},1)$-test that covers $I$, but $I$ is $(\mu_I, \mu^{(j+2)}, 3)$-random.

  We obtain the following characterization of $I$.
  \[
    a \in I \quad \Leftrightarrow \quad \forall n \: \exists l \: (\mu_I \Cyl{T_a\Rest{l}} \leq 2^{-n})
  \]
  Since $\mu_I \leq_{\T} \mu^{(j+2)}$, the property on the right hand side is $\Pi^0_2(\mu^{(j+2)})$, hence $I$ is recursive in $\mu^{(j+4)}$, contradicting our initial assumption.
\end{proof}

\medskip
We conclude this section by establishing that Turing jumps cannot be
$\mu$-$n$-random, $n \geq 2$, for any measure $\mu$.  
While strictly speaking the following two results are not needed later, they are prototypical for a type of argument that will be important in Section~\ref{sec-meta}, where we construct long sequences of reals with an internal definability hierarchy that are not random with respect to any continuous measure.

\begin{prop} \label{pro:jump-non-random-1}
	For any $k \geq 0$, if $X \equiv_{\T} \emptyset^{(k)}$, then $X$ is not $2$-random with respect to any continuous measure.
\end{prop}

\begin{proof}
  The case $k = 0$ is clear, so assume $k > 0$. Suppose $X \equiv_{\T}
  \emptyset^{(k)}$ is $\mu$-$2$-random for some $\mu$. Then $\emptyset'
  \leq_{\T} X$ and also $\emptyset' \leq_{\T} \mu'$. It follows from Lemma~\ref{prop-non-accel} that $\emptyset'$ is recursive in $\mu$. Applying the same
  argument inductively to $\emptyset^{(i)}$, $i
  \leq k$,  yields $\emptyset^{(i)} \leq_{\T} \mu$, in particular $X
  \equiv_{\T} \emptyset^{(k)} \leq_{\T} \mu$, which is impossible if $X$ is
  $\mu$-$2$-random.
\end{proof}

It may be helpful to picture the preceding argument as a ``stair trainer machine'': Using the supposedly random $X$, each step, that is, Turing jump, ``sinks down'' to $\mu$, and eventually, $X \leq_{\T} \mu$, yielding a contradiction.  

The non-randomness property of the jumps extends to infinite jumps, too.

\begin{prop} \label{pro:jump-non-random-2}
	If $X \equiv_{\T} \emptyset^{(\omega)}$, then $X$ is not $3$-random with respect to a continuous measure.
\end{prop}

\begin{proof}
  Assume for a contradiction that $X$ is $\mu$-$3$-random for continuous $\mu$.
  By the inductive argument of the previous proof, $\emptyset^{(k)} \leq_{\T} \mu$ for
  all $k \in \Nat$. By a result of Enderton and Putnam~\cite{enderton-putnam:note-hyperarithmetical_1970}, if $Y$ is a $\leq_{\T}$-upper bound for
  $\{\emptyset^{(k)} \colon k \in \Nat \}$, then $\emptyset^{(\omega)} \leq_{\T} Y''$.
  Therefore, $X \leq_{\T} \mu''$, contradicting that $X$ is $\mu$-$3$-random.
\end{proof}

%===========================================================
%
% The Countability Theorem
%
\section{The Countability Theorem} \label{sec-countability}

In this section we will prove Theorem~\ref{thm:first-main}, which we restate here for convenience.

\begin{thm1.1} \label{thm-countability}
	Let $n \in \omega$. Then the set
	\[
		\NCR_n = \{ X \in \Cant: \: X \text{ is not $n$-random for any continuous measure}\}
	\]
	is countable.
\end{thm1.1}

As mentioned in the introduction, the case $n=1$ was proved
in~\cite{reimann-slaman:tams}. The basic outline for the proof for $n > 1$ is as follows. By Borel-Turing determinacy~\cite{martin:1968}, the set of degrees of reals $X$ such that $X$ is $n$-random for some continuous measure recursive in $X$ contains an upper cone in the Turing
degrees. The base of the cone is given by the winning strategy in a certain Borel game $\mathcal{G}(B)$. The 
constructive nature of Martin's proof of Borel determinacy yields that a winning strategy is contained in a countable level $L_{\beta_{n+4}}$ of the
constructible hierarchy. We use a forcing notion due to
Kumabe and Slaman (see~\cite{shore-slaman:1999}) to show that given $X \notin
L_{\beta_{n+4}}$, there exists a forcing extension $L_{\beta_{n+4}}[\Phi]$ such that $\Phi$ is a real number and every
real in $L_{\beta_{n+4}}[\Phi]$ is Turing reducible to $X$ relative to $\Phi$. In particular, $X \oplus \Phi$ is in the upper cone of random reals above the winning strategy for the game $\mathcal{G}(B)$ (in $L_{\beta_{n+4}}[\Phi]$). Finally, we argue that $X$ itself is random with respect to a continuous measure.

%--------------------------------------
%
% Relativizing upper cones
%
\subsection{About Borel determinacy}

We will apply Martin's theorem about the determinacy of Borel games in the following special case.

\begin{BTDet}[\cite{martin:1968, martin:1975}] \label{thm-borel-turing-det}
  If
  $\Cl{A} \subseteq \Cant$ is a Turing invariant Borel set, then either $\Cl{A}$ or $\Cant \setminus \Cl{A}$ contains a Turing cone.
\end{BTDet}

Martin's proof of Borel determinacy is explicit. It locates a winning strategy for a given Borel game within the constructible hierarchy.

\begin{defn}
	Given $n \in \Nat$, $\ZFC^-_n$ denotes the axiom of $\ZFC$, where the power
  set axiom is replaced by the sentence
	\begin{center}
		``There exist $n$-many iterates of the power set of $\omega$''.
	\end{center}
\end{defn}

Hence, in $\ZFC^-_0$, for instance, we have the existence of the set of all
natural numbers (since the Axiom of Infinity holds and $\omega$ is absolute),
and various other subsets of $\omega$ as given by applications of separation or
replacement, but we lack the guaranteed existence of the set of all such
subsets.

Models of $\ZFC^-_n$ will play an important role throughout this paper. In
particular, we are interested in models inside the constructible universe. As
usual, $L$ will denote the constructible universe, the limit of the cumulative hierarchy of
sets obtained by iterating the power set operation restricted to definable
subsets. For any ordinal $\alpha$, $L_\alpha$ denotes the $\alpha$-th level of
the hierarchy. A key property of this hierarchy is that $|L_\alpha| = |\alpha|$ for any infinite ordinal $\alpha$.
For more background on $L$, see~\cite[Chapter~13]{jech:2003} or~\cite[Chapters~V,VI]{kunen:2011}.
For an in-depth account, see~\cite{devlin:1984}.

\begin{defn}
	Given $n \in \Nat$, let $\beta_n$ be the least ordinal such that
	\[
		L_{\beta_n} \models \ZFC^{-}_n.
	\]
\end{defn}

By the L\"owenheim-Skolem theorem and the Gödel condensation lemma, 
$L_{\beta_n}$, and hence $\beta_n$, is countable.

\begin{lem} \label{lem:det-constructive}
  If $A \subseteq \Cant$ is $\Sigma^0_n$, then the Borel game $\mathcal{G}(A)$
  with winning set $A$ has a winning strategy $S$ in $L_{\beta_n}$.
\end{lem}

The proof given by Martin~\cite{martin:1985} is inductive. A key concept is the
\emph{unraveling of a game}. Simply speaking, a tree $T$ over some set $B$
unravels $\mathcal{G}(A)$ if there exists a continuous mapping $\pi\colon [T]
\to \Cant$ such that $\pi^{-1}(A)$ is clopen in $[T]$, and there is a continuous
correspondence between strategies on $T$ and strategies on $\Str$.

Martin first shows that $\Pi^0_1$ games can be unraveled. The argument is
completely constructive, hence can be carried out in $L$. The unraveling tree
$T$ is given by the legal moves of some auxiliary game whose moves correspond to
strategies in the original game on $\Str$, that is, reals. To be able to collect
all these legal moves requires the existence of the power set of $\omega$.

The inductive step then shows how to unravel a given $\Sigma^0_n$ set $A$.
Suppose $A = \bigcup A_i$, where each $A_i$ is $\Pi^0_{n-1}$. By induction
hypothesis, each $A_i$ can be unraveled by some $T_i$ via some mapping $\pi_i$.
Martin proves that the unravelings $T_i$ can be combined into a single one,
$T_\infty$, that unravels each $A_i$ via some $\pi_\infty$. Since each of the
sets $\pi^{-1}_\infty(A_i)$ is clopen, their union $\bigcup \pi^{-1}_\infty(A_i)
= \pi^{-1}_\infty(A)$ is open, and can in turn be unraveled by some $T$. Again,
the proof is constructive. The last step in the construction (unraveling
$\pi^{-1}_\infty(A)$) passes to a tree of higher type -- its nodes correspond to
strategies over $T_\infty$. Hence one more iterate of the power set of $\omega$
is introduced.

Therefore, $\Sigma^0_n$ determinacy is provable in $\ZFC^-_n$, and Martin's proof constructs a winning strategy $S$ in $\Lb{n}$, relative to $\Lb{n}$. The property of being a winning strategy for a given Borel game is
$\pmb{\Pi}^1_1$. By Mostowski's absoluteness theorem
(see~\cite[Theorem~25.4]{jech:2003}, which applies because in a model of $\ZFC^-_n$, $n\geq 1$, every well-founded tree contained in $\omega^{<\omega}$ has a rank function), this means that $S$
actually wins on \emph{all} plays, not just the ones in $L_{\beta_{n}}$.

\medskip 
In order to prove Theorem~\ref{thm:first-main}, we need 
 a Posner-Robinson style theorem for reals not contained in
$L_{\beta_n}$ and an application of absoluteness. We state
Lemma~\ref{thm-posrob-L} only in the case that $n$ is greater than zero. Under
this restriction, we can avoid class forcing and reduce to standard facts about
set forcing. The case $n$ equals zero is not needed for our argument.

\begin{lem}\label{thm-posrob-L} 
  Suppose that $n$ is a natural number greater than zero and $X$ is a real
  number not in $L_{\beta_n}$. Then there exists a real $\Phi \subseteq \Nat$ such that  $L_{\beta_n}[\Phi]$ is a model of
  $\ZFC^{-}_n$ and every real in $L_{\beta_n}[\Phi]$ is Turing reducible to
  $X \oplus \Phi$.
\end{lem}

\medskip
\subsection{Kumabe-Slaman forcing}

This subsection is devoted to proving Lemma~\ref{thm-posrob-L}. We derive
$\Phi$ from a generic $G$ which we construct by means of a notion of forcing due to Kumabe and Slaman. The forcing was an
essential ingredient in the proof of the definability of the Turing jump by
Shore and Slaman~\cite{shore-slaman:1999}. It allows for extending the
Posner-Robinson Theorem to iterated applications of the Turing jump. 

In the following, we use the conventions and vocabulary of Section~\ref{sub:turing_functionals}.

\begin{defn}\label{def-kumbae-slaman-po}
  Let ${\bbP}$ be the following partial order.

  \begin{enumerate}
  \item The elements $p$ of ${\bbP }$ are pairs $(\Phi_{p},\Z_{p})$ in which $\Phi_{p}$ is a finite,\
    use-monotone Turing functional and $\Z_{p}$ is a finite set of subsets of $\omega$. As
    usual, we identify subsets of $\omega$ with elements of $\Cant$.

  \item  If $p$ and $q$ are elements of ${\bbP }$, then $p\geq q$ if and only if
    \begin{enumerate}
    \item
      \begin{enumerate}
      \item $\Phi_{p}\subseteq\Phi_{q}$ and
      \item for all $(x_{q},y_{q},\sigma_{q})\in\Phi_{q}\setminus\Phi_{p}$ and all
        $(x_{p},y_{p},\sigma_{p})\in\Phi_{p}$, the length of $\sigma_{q}$ is greater than the length of $\sigma_{p}$,
      \end{enumerate}
    \item $\Z_{p}\subseteq\Z_{q}$,

    \item  for every $x$, $y$, and $Y\in\Z_{p}$, if $\Phi_{q}(x,Y)=y$ then
      $\Phi_{p}(x,Y)=y$.
    \end{enumerate}
  \end{enumerate}
\end{defn}

In short, a stronger condition than $p$ can add computations to $\Phi_{p}$, provided that they are
longer than any computation in $\Phi_{p}$ and that they do not apply to any element of ${\Z}_{p}$.

% Revised section on forcing.
\newcommand{\Pn}{{\bbP_{n}}}
\newcommand{\Ln}{{L_{\beta_n}}}

Let $\Pn$ denote the partial order $\bbP$ as defined in $L_{\beta_n}$.  By standard arguments, we
show that if $G\subseteq\Pn$ is a generic filter in the sense of $\Ln$, then $L_{\beta_{n}}[G]$ is a
model of $\ZFC^-_n$.  By inspection of $\Pn$, any such $G$ naturally gives rise to a functional
$\Phi_{G} =\bigcup\{\Phi_{p}:p\in{G}\}$.  To prove Lemma~\ref{thm-posrob-L}, given $X$ not in
$\Ln$, we will exhibit a particular $G$ so that $G$ is $\Pn$-generic over $\Ln$ and so that every
element in $L_{\beta_n}[G]$ is computable from $X \oplus \Phi_G$.

\begin{defn}[Definition III.3.3, \cite{kunen:2011}] Let $\bbP^*$ be a partially ordered set.  Then,
  $p,q\in\bbP^*$ are \emph{compatible} iff they have a common extension.  An \emph{antichain} is a
  subset of $\bbP^*$ whose elements are pairwise incompatible.  $\bbP^*$ has the \emph{countable
    chain condition} (ccc) iff, in $\bbP^*$ every antichain is countable.
\end{defn}

\begin{lem}\label{forcing1}
  Let $n$ be a natural number greater than zero.
  \begin{enumerate}
  \item $\Pn$ is an element of $\Ln$.
  \item $\Ln\models \text{$\Pn$ has the ccc.}$
  \end{enumerate}
\end{lem}

\begin{proof}
  Since $n\geq1$, by definition of $\beta_n$, the power set of $\Nat$ as defined in $\Ln$ is
  a set in $\Ln$.  Since $\bbP$ is defined directly from the power set of $\Nat$, $\Pn$ is a set in $\Ln$.

  If $p$ and $q$ are incompatible elements of $\Pn$, then $\Phi_p$ and $\Phi_q$ must be different.
  Since there are only countably many possibilities for $\Phi_p$ and $\Phi_q$, any antichain in
  $\Pn$ must be countable in $\Ln$.  
\end{proof}

\begin{defn} Let $G$ be a subset of $\Pn$.
  \begin{enumerate}
  \item $G$ is a \emph{filter} on $\Pn$ iff
  \begin{enumerate}
  \item $G$ is not empty.
  \item $\forall p,q\in G\;\exists r\in G[p\geq r\text{ and }q\geq r]$.
  \item $\forall p,q\in G[\text{if } p\geq q\text{ and }q\in G,\text{ then }p\in G].$
  \end{enumerate}
\item $G$ is \emph{$\Pn$-generic over $\Ln$} iff for all $D$ such that $D\subseteq\Pn$ is dense and
  $D \in \Ln$, $G\cap D$ is not empty.  Here, $D$ is dense iff for
  every $p\in\Pn$ there is a $q\in D$ such that $p\geq q$.
  \end{enumerate}
\end{defn}

\begin{lem}
  If $G$ is $\Pn$-generic over $\Ln$,  then $\Ln[G]\models\ZFC^-_n.$ Similarly, $\Ln[\Phi_G]$ is a model of $\ZFC^-_n$.
\end{lem}

\begin{proof}
  By Lemma~IV.2.26 of \cite{kunen:2011}, it follows that if $G$ is $\Pn$ generic over $\Ln$, then
  $\Ln[G]$ satisfies the axioms of $\ZF^-$ except for possibly Replacement.  It remains only to
  observe that $\Ln[G]$ satisfies Replacement, Choice and that there are $n$-many uncountable
  cardinals.  The verification that $\Ln[G]$ satisfies Replacement is the same as given in
  Theorem~IV.2.27 of \cite{kunen:2011}.  That $\Ln[G]$ satisfies Choice follows from Replacement and
  the usual proof that the order of constructibility is a $\Ln$-definable well-order of $\Ln$ applies
  relative to $G$.  Finally, that $\Ln[G]$ has the same uncountable cardinals as $\Ln$ does follows
  from $\Pn$'s having the ccc in $\Ln$ by the argument given in the proof of Theorem IV.3.4 of
  \cite{kunen:2011}.  Although this theorem is stated for models of $\ZFC$, its proof does not
  invoke the power set axiom.  
  
  Given that $\Ln[G]$ has $n$ many uncountable cardinals, we proceed to apply a variant of the G\"odel
  Condensation Lemma relative to $G$ to show that  $\Ln[G]$ has $n$-many iterates of the power set of
  $\omega$. 
  
  The existence of the power set of $\omega$ is a special case. Let $\omega_1^*$ denote $\omega_1^{\Ln[G]}$. We may view $G$ as a subset of $\omega_1^*$. We show that any subset of $\omega$ in $\Ln[G]$ is an element of $L_{\omega_1^*}(G)$, the structure obtained by constructing relative to $G$ as a predicate to height $\omega_1^*$. Assume $x \in L_\gamma[G]$, $\gamma < \beta_n$, and $\gamma > \omega_1^*$. Let $H$ be a countable elementary substructure of $L_\gamma[G]$ in $\Ln[G]$ that includes $x$. Let $\pi: H \to \overline{H}$ be the transitive collapse of $H$. Note that $\pi$ is the identity on all countable ordinals in $H$. Hence $\pi(G)$ is $G \cap \pi(\omega^*_1)$ and $\pi(x) = x$. By the Condensation Lemma, $\overline{H}$ is isomorphic to some $L_{\gamma^*}[G \cap \pi(\omega_1^*)]$ and $\gamma^* < \omega_1^*$. It follows that $x \in L_{\gamma^*}[G \cap \pi(\omega_1^*)]$, and thus it is an element of $L_{\omega_1^*}(G)$.
  
  The case for the higher order power sets is the standard Condensation Lemma argument.
  
  \medskip
  Finally, $\Ln[\Phi_G]$ is a model of $\ZFC^-_n$ because it is a definable inner model of $\Ln[G]$, so it satisfies $\ZFC^-$, and the Condensation Lemma applies, yielding the same number of iterates of the power set of $\omega$.
\end{proof}

Next, we show that every dense set in $\Lb{n}$ can be met via an extension adding no computations
along $X$. This is crucial for the construction in~\cite{shore-slaman:1999}.

\begin{lem}\label{lem-meet-dense}
  Let $D\in\Ln$ be dense in $\Pn$ and $X \in 2^\omega$, $X \notin \Ln$. For any $p \in \Pn$, there exists a
  $q \leq p$ such that $q \in D$ and $\Phi_q$ does not add any new computation along $X$.
\end{lem}

\begin{proof}
  Suppose $p = (\Phi_p,\Z_p)$ is in $\Pn$.  We say a string $\tau$ is \emph{essential} for $(p,D)$
  if, whenever $q < p$ and $q \in D$, there exists a triple
  $(x,y,\sigma) \in \Phi_q \setminus \Phi_p$ such that $\sigma$ is compatible with $\tau$.  In other
  words, whenever one meets $D$ by an extension of $p$, a computation relative to some string
  compatible with $\tau$ is added.  Note that $\tau$'s being essential for $(p,D)$ is definable in
  $\Ln$.

  If $\tau$ is a binary sequence and $\tau_0$ is an initial segment of $\tau$,
  then any sequence $\sigma$ which is compatible with $\tau$ is also compatible
  with $\tau_0$.  Thus, being essential is closed under taking initial segments.
  So,
  \[
    T(p,D) = \{ \tau\colon \text{$\tau$ essential for $(p,D)$ } \}
  \]
  is a binary tree in $\Ln$.

Assume now for a contradiction that a $q$ as postulated above does not exist.
  This means that for any $r\leq p$, either $r\notin D$ or $\Phi_r$ adds a
  computation along $X$. It follows that every initial segment $\tau\Sle X$ is
  essential for $(p,D)$.  Thus $T(p,D)$ is infinite.  Since $\Ln$ satisfies
  K\H{o}nig's Lemma (equivalently, compactness of the Cantor set), there exists
  a real $Y \in \Cant \cap \Lb{n}$ such that $Y$ is an infinite path through
  $T(p,D)$.

  Now consider the condition $p_1= (\Phi_p, \Z_p \cup \{Y\})$.  As $\Phi_q =
  \Phi_p$ and $Y\in\Ln$, we trivially have $p_1 \leq p$ in $\Pn$.  Since every
  initial segment of $Y$ is essential for $(p,D)$, any extension of $p$ in $D$
  must add a computation along $Y$.  Since no extension of $p_1$ can add a
  computation along $Y$ and every extension of $p_1$ is an extension of $p$, no
  extension of $p_1$ is in $D$.  This contradicts the density of $D$.
\end{proof}

  The previous lemma also holds in the case $X \in \Ln$, but we are only interested in the case $X \notin \Ln$.

We can now finish the proof of Lemma~\ref{thm-posrob-L}.  It is sufficient to construct a
$\Pn$-filter $G$ that is generic over $\Ln$ such that for every $A:\omega\to2$ in $\Ln[G]$ there is
a $k$ such that for all $m$, $\Phi^X_G((k,m))=A(m)$.  We fix countings of $\Pn$, of the set of terms
$\tau$ in the forcing language for functions from $\omega$ to $2$ in $\Ln$ and of the dense subsets
$D$ of $\Pn$ in $\Ln$.  We let $(\tau_i:i\in\Nat)$ and $(D_i:i\in\Nat)$ denote the latter two
countings.  We proceed by recursion on $s$ to define $G$.  At stage $s$, we will define a condition
$p_s=(\Phi_{p_s},\vec{Z}_{p_s})$ and an integer $k_s$.  We will ensure
that for all $i$ and $m$, $\Phi^X_G((k_i,m))$ will have the same value as the interpretation of $\tau_i(m)$
in $\Ln$.  We define $p_0$ to be the empty condition.  Suppose that $p_s$ and $k_0,\dots, k_{s-1}$
are given.  First, we define $k_s$ to the be least integer $k$ such that for all
$(x,y,\sigma)\in\Phi_{p_s}$, $x$ is not of the form $(k,m)$ for any $m$.  Next, let $q$ be the least
condition, according to our counting of $\Pn$, such that $q$ extends $p_s$, $q\in D_s$, $q$ decides
the values of $\tau_i(j)$ for each $i \leq s$ and $j\leq s$, and $\Phi_{p_s}^X=\Phi_q^X$. By
Lemma~\ref{lem-meet-dense}, there is such a $q$. For $i$ and
$j$ less than or equal to $s$, let $m_{i,j}$ be the value decided for $\tau_i(j)$ by $q$.    Since $X\not\in \Ln$ and $\vec{Z}_q$ is finite,
there is an $\ell$ such that for each element $Z\in\vec{Z}_q$, $Z$ and $X$ disagree on some number
less than $\ell$.  By increasing $\ell$ if necessary, we may further assume that $\ell$ is greater
than the length of any $\sigma$ mentioned in an element of $\Phi_q$.  Let \[\Phi_{p_{s+1}}=
\Phi_q\cup\{(k_i,m_{i,j},X\restriction\ell):i,j\leq s \text{ and $\Phi_q^X((i,j))$ is not defined}\}.  \]
Let $p_{s+1}$ be $(\Phi_{p_{s+1}},\vec{Z}_q)$.  By construction, $p_{s+1}$ is an extension of $q$
in $\Pn$.  Finally, let $G$ be the filter of conditions in $\Pn$ which are extended by
some $p_s$ and let $\Phi_G$ be the union of the set of $\Phi_p$ such that $p\in G$.

We conclude by observing that the set $\Phi_G$, viewed as a subset of $\Nat$, satisfies the two
conclusions of Lemma~\ref{thm-posrob-L}.  First, we argue that $\Ln[\Phi_G]$ is model of $\ZFC^-_n$. Note that
  the cardinals of $\Ln[G]$ are the cardinals of $\Ln$, and since $\Phi_G\in\Ln[G]$ they are also the cardinals of $\Ln[\Phi_G]$. By Gödel's proof of the GCH, applied to $\Ln[\Phi_G]$, $\Ln[\Phi_G]$ has $n$ iterates of the power set of $\Nat$. That means $\Ln[\Phi_G]$ satisfies $\ZFC^-_n$. 

Second, for each $Z:\omega\to2$, if
$Z\in\Ln[\Phi_G]$ then there is an $i$ such that $Z$ is the denotation of $\tau_i$ in $L_{\beta_n}[G]$.  By
direct induction on the construction, for this $i$, for all $m$, $Z(m)=\Phi^X_G((k_i,m))$.

%--------------------------------------
%
% Completing the proof of Theorem~\ref{thm-countability}
%
\subsection{The proof of
  Theorem~\ref{thm:first-main}} \label{sub:completion-thm1}

We now put the pieces together to show that every real outside of
$L_{\beta_{n+4}}$ is $n$-random with respect to a continuous probability
measure. As $L_{\beta_{n+4}}$ is countable, this will complete the proof of Theorem~\ref{thm:first-main}.

Given $X \not\in L_{\beta_{n+4}}$, choose $\Phi$ as in
Lemma~\ref{thm-posrob-L}, that is, every real in $L_{\beta_{n+4}}[\Phi]$ is Turing reducible to $X \oplus \Phi$. 

Consider the game with the following winning set

\begin{equation*}
  B^\Phi =  \{ W \in \Cant: \: \exists Z \, \exists R \, (W\oplus \Phi \equiv_{\T} Z \oplus \Phi \oplus R \;\; \& \;\; R \text{ is $(Z\oplus \Phi,n+3)$-random})\}.
\end{equation*}
Clearly, $B^\Phi$ is Turing invariant. To see that it is Borel, note that $W \oplus \Phi$ is in $B^\Phi$ if and only if
	\begin{align*}
    \exists e, d \; (\text{$e,d$ are indices of } & \text{Turing functionals such that } \\
    & \Phi_d(\Phi_e(W\oplus \Phi)) = W\oplus \Phi, \\
    & \text{$\Phi_e(W \oplus \Phi)$ is of the form $Z \oplus \Phi \oplus R$,} \\
    & \text{and $R$ is $(n+3)$-random relative to $Z \oplus \Phi$}).
	\end{align*}
This is a prima facie arithmetic definition of $B^\Phi$. By counting quantifiers, we see that $B^\Phi$ is $\Sigma^0_{n+4}(\Phi)$. We also know that the Turing degrees of $B^\Phi$ are cofinal in the Turing degrees (since it is always possible to ``add on'' another random).  

By Lemma~\ref{lem:det-constructive}, $L_{\beta_{n+4}}[\Phi]$ contains a winning strategy $S$ for the relativized game
$\mathcal{G}(B^\Phi)$.  Since $B^\Phi$ is cofinal in the Turing degrees, by Borel-Turing determinacy $S$ has to be a winning strategy for the player who wins if the result of the game is in $B^\Phi$. If any strategy is played against a real that computes the strategy, the result of game is Turing equivalent to the real the strategy was played against.

Therefore, since $X \oplus \Phi$ computes $S$, $B^\Phi$ is Turing invariant, and $(X\oplus \Phi) \oplus \Phi$ is Turing equivalent to $X \oplus \Phi$, $X \oplus \Phi$ is in $B^\Phi$. Let $Z$ and $R$ witness the condition that  $X\oplus \Phi \in B^\Phi$. Since $R$ is $(Z\oplus \Phi, n+3)$-random, $X \oplus \Phi$ is not recursive in $(Z \oplus \Phi)'$. By Lemma~\ref{prop:rand-via-turing}, $X \oplus \Phi$ is $(\nu, (Z\oplus \Phi)'', n+1)$-random for some continuous measure $\nu$ recursive in $(Z\oplus \Phi)''$.

Let $\widetilde{\nu}$ be the measure on $\Cant$ given by 
\[
  \widetilde{\nu}(A) = \nu(\{W \oplus \Phi \colon W \in A \})
\]
Since $X \oplus \Phi$ is $\nu$-random, and it belongs to the $\Pi^0_1(\Phi)$ set
\[
  \Cant \oplus \Phi = \{W \oplus \Phi \colon W \in \Cant\},
\]
$\nu(\Cant \oplus \Phi) > 0$, and therefore $\widetilde{\nu}(\Cant) > 0$. Furthermore, since $\nu$ is continuous, so is $\widetilde{\nu}$. 

For any $\sigma \in \Str$, $\widetilde{\nu}\Cyl{\sigma}$ is recursive in $(\nu \oplus \Phi)'$, uniformly in $\sigma$. It follows that the normalization $\widetilde{\nu}_N$ of $\widetilde{\nu}$ $$\widetilde{\nu}_N = \frac{1}{\widetilde{\nu}(\Cant)}\: \widetilde{\nu}$$ is recursive in $(\nu \oplus \Phi)'$, too. Further, for any $\sigma \in \Str$ and any rational $\varepsilon > 0$, we can compute, recursively in $(\nu \oplus \Phi)'$, an open cover $U$ of 
\[
  \sigma \oplus \Phi =\{W  \oplus \Phi \colon W \in \Cyl{\sigma} \}
\]
so that
\[
  | \nu(\sigma \oplus \Phi) - \nu(U) | < \varepsilon,
\]
uniformly in $\sigma$ and $\varepsilon$.

Thus, for any $k \geq 1$, there is uniformly recursive in $(\nu \oplus \Phi)^{(k)}$ procedure converting a Martin-Löf test for $(\widetilde{\nu}_N)$ relative to $(\nu \oplus \Phi)^{(k)}$ into a Martin-Löf test for $\nu$ relative to $(\nu \oplus \Phi)^{(k)}$. Since $X \oplus \Phi$ is $(\nu, (Z\oplus \Phi)'', n+1)$-random and $\nu$ is recursive in $(Z\oplus \Phi)''$, $X\oplus \Phi$ is $(\nu, (\nu \oplus \Phi), n+1)$-random. Therefore, $X$ is $(\widetilde{\nu}_N, (\nu\oplus \Phi)', n)$-random.
Finally, note that every $(\widetilde{\nu}_N, (\nu\oplus \Phi)', n)$-random real is $(\widetilde{\nu}_N, n)$-random.

We have shown that every real not contained in $L_{\beta_{n+4}}$ is $n$-random for a continuous measure. As $L_{\beta_{n+4}}$ is countable, this completes the proof of Theorem~\ref{thm:first-main}.

%===========================================================
%
% The Metamathematics of Randomness
%
\section{The Metamathematics of Randomness} \label{sec-meta}

In this section, we will show that the metamathematical ingredients used to prove the countability of $\NCR_n$ are necessary. More precisely, we will prove the Theorem~\ref{thm:second-main}, which we restate here for convenience.

\begin{thm1.2a}\label{thm:second-main-var}
	There exists a computable function $G(n)$ such that for every $n \in \Nat$,
	\[
		\ZFC^{-}_n \nvdash \text{ ``$\NCR_{G(n)}$ is countable.'' }
 	\]
\end{thm1.2a}

\medskip
Before starting the proof, we outline its basic idea. For given $n$, we will show that in the model $L_{\beta_n}$ of $\ZFC^{-}_n$, $\NCR_{G(n)}$ is not countable. To this end, we find a sequence $(Y_\alpha)$ of reals that satisfies
\begin{enumerate}[(1)]

	\item $(Y_\alpha)$ is cofinal in the Turing degrees of $L_{\beta_n}$, hence not countable in $L_{\beta_n}$,

	\item no $Y_\alpha$ is $G(n)$-random for a continuous measure in $L_{\beta_n}$.

\end{enumerate}

As we have seen in Propositions~\ref{pro:jump-non-random-1} and~\ref{pro:jump-non-random-2}, iterating the Turing jump produces an increasing sequence of non-random reals. It makes sense therefore to look for a set-theoretic analogue of the jump hierarchy that is cofinal in the hierarchy of constructibility within the countable structure $L_{\beta_n}$. Each initial segment of this hierarchy has a canonical representation in $L$ as a countable object. These are called \textit{master codes} for the structures that they represent. They induce a well-ordered, increasing sequence under Turing reducibility, just like the Turing jump. The `Stair Trainer''-argument to come will proceed along this hierarchy, as it did along the jump hierarchy in Propositions~\ref{pro:jump-non-random-1} and~\ref{pro:jump-non-random-2}. 

We will show that none of these master codes can be (sufficiently) random with respect to a continuous measure $\mu$. In the Stair Trainer method, we compare the sets recursive in a fixed jump of $\mu$ to the elements of the hierarchy on which we are implementing the method. There are new challenges to be met here that go beyond iterates of the Turing jump. For example, we can have representations which are not canonical representations and we can have representations of non-wellfounded versions of the constructible hierarchy. Most of the technical work we have to do is to organize the representations of structures that are definable from $\mu$--to linearize them and to eliminate the non-wellfounded ones. The application of the Stair Trainer method then comes by comparing the canonical representations to the organized $\mu$-definable representations, and then showing that there is no first place where randomness can occur.   

\medskip
This section of the paper is laid out as follows. \ref{sub:fine_structure_and_jensen_s_} and \ref{sub:projecta} review some basic facts of Jensen's fine structure analysis of the constructible hierarchy. Readers familiar with fine structure theory can skip ahead to Subsection~\ref{ssec:arithmetic_copies}. In \ref{ssec:arithmetic_copies} and ~\ref{ssec:defining-copies}, we study representations (which we call \textit{$\omega$-copies}) of countable levels of the constructible hierarchy from a computability point of view. We show how we can computationally retrieve the information coded in an $\omega$-copy, and how we can propagate representations of lower levels to obtain representations of higher levels. In~\ref{ssec:arith-master-codes}, we extract an arithmetic property which captures much of the combinatorics of the fine structure. Representations with this weaker property will be called \textit{pseudocopies}. In~\ref{ssec:comparing-copies}, we implement the organization of pseudocopies recursive in a given real as described above. In~\ref{sub:ncr-not-rand} and~\ref{sub:finishing_the_proof_of_theorem_2}, we implement the Stair Trainer technique and complete the proof of Theorem~\ref{thm:second-main-var}.

\subsection{Fine structure and Jensen's J-hierarchy} % (fold)
\label{sub:fine_structure_and_jensen_s_}

Fine structure provides a level-by-level, quantifier-by-quantifier
analysis of how new sets are generated in $L$. Jensen defines the new
constructible hierarchy, the \emph{$J$-hierarchy}
$(J_\alpha)_{\alpha \in \Ord}$ that has all the important properties
of the $L$-hierarchy (in particular, $L = \bigcup_\alpha
J_\alpha$). In addition to this, each level $J_\alpha$ has closure
properties (such as under pairing functions) that $L_\alpha$ may be
lacking.  While it is not strictly necessary for this paper to work
with $J_\alpha$ (we could work with
$(L_{\omega\alpha})_{\alpha \in \Ord})$, the $J$-hierarchy is the
established framework for fine structure analysis, and we will adopt
its basic concepts and terminology.

The sets $J_{\alpha}$ are obtained by closing under a scheme of \emph{rudimentary functions}. In contrast to $L_{\alpha+1}$,
$J_{\alpha+1}$ contains sets of rank up to $\omega(\alpha+1)$, not
just subsets of $J_{\alpha}$, e.g.\ ordered pairs. The
rudimentary functions are essentially a scheme of \emph{primitive set
  recursion}~\cite{Jensen-Karp:1971a}.

For transitive $X$, $\Rud(X)$ denotes the smallest set $Y$ that
contains $X \cup \{X\}$ and is closed under rudimentary functions (rud
closed). The inclusion of $\{X\}$ when taking the rudimentary closure
guarantees that new sets are introduced even if $X$ is closed under
rudimentary functions.

The $J$-hierarchy is introduced as a cumulative hierarchy induced by
the $\Rud$-operation:

\begin{align*}
  J_0 & = \emptyset \\
  J_{\alpha+1} & = \Rud(J_\alpha) \\
  J_{\lambda} & = \bigcup_{\alpha < \lambda} J_\alpha \quad \text{ for $\lambda$ limit}.
\end{align*}

A fine analysis of the rudimentary functions reveals that the
$\Rud$-operation can be completed by iterating some or all of
\emph{nine basic rudimentary functions}.

\begin{prop}[Jensen~\cite{jensen:1972}] \label{prop:rud-base-functions} Every
  rudimentary function is a composition of the following nine
  functions:
\begin{align*}
  F_0(x,y) & = \{x,y\}, \\
  F_1(x,y) & = x\setminus y, \\
  F_2(x,y) & = x \times y, \\
  F_3(x,y) & = \{ (u,z,v) \colon z \in x \; \wedge \; (u,v) \in y \}, \\
  F_4(x,y) & = \{ (u,v,z) \colon z \in x \; \wedge \; (u,v) \in y \}, \\
  F_5(x,y) & = \bigcup x, \\
  F_6(x,y) & = \Dom(x), \\
  F_7(x,y) & =  \: \in \cap \: (x\times x), \\
  F_8(x,y) & = \{ \{ x(z)\} \colon z \in y \}.
\end{align*}
\end{prop}

The $S$-operator is defined as taking a one-step application of
any of the basic functions,
\begin{equation} \label{equ:S-operator}
	S(X) = [X \cup \{X\}] \cup \left[ \bigcup_{i = 0}^8 F_i[X \cup \{X\}] \right].
\end{equation}
For transitive $X$, it holds that~\cite[Corollary~1.10]{jensen:1972}
\[
	\Rud(X) = \bigcup_{n \in \omega} S^{(n)}(X).
\]
The $S$-hierarchy is defined as the cumulative hierarchy induced by
the $S$-operator and refines the $J$-hierarchy.

\begin{align*}
  S_0 & = \emptyset, \\
  S_{\alpha+1} & = S(S_\alpha), \\
  S_{\lambda} & = \bigcup_{\alpha < \lambda} S_\alpha \quad \text{ for $\lambda$ limit}.
\end{align*}

We obviously have
\[
  J_\alpha = \bigcup_{\beta < \omega\alpha} S_\beta =
  S_{\omega\alpha}.
\]

We list a few basic properties of the sets $J_{\alpha}$. For details
and proofs, see~\cite{jensen:1972} or~\cite{devlin:1984}.

\begin{itemize}

\item Each $J_{\alpha}$ is transitive and is
  a model of a sufficiently large fragment of set theory (more
  precisely, it is a model of \Ax{KP}-set theory without
  $\Sigma_0$-collection).

\item The hierarchy is \emph{cumulative}, i.e., $\alpha \leq \beta$
  implies $J_{\alpha} \subseteq J_{\beta}$.

\item
  $\operatorname{rank}(J_{\alpha+1}) =
  \operatorname{rank}(J_\alpha)+\omega$. Each successor step adds
  $\omega$ new ordinals. $J_{\alpha} \cap \Ord = \omega \alpha$, in
  particular, $J_{1} = V_\omega$ and $J_1 \cap \Ord = \omega$.

\item $(J_{\alpha})_{\alpha \in \Ord}$ and
  $(L_{\alpha})_{\alpha \in \Ord}$ generate the same universe:
  $L = \bigcup_{\alpha} J_{\alpha}$. Moreover,
  $L_{\alpha} \subseteq J_{\alpha} \subseteq L_{\omega\alpha}$, and
  $J_{\alpha} = L_{\alpha}$ if and only if $\omega \alpha =
  \alpha$. Finally,
  $J_{\alpha+1} \cap \P(J_\alpha) = \P_{\Def}(J_\alpha)$, that is,
  $J_{\alpha+1}$ contains precisely those subsets of $J_\alpha$ that
  are first order definable over $J_\alpha$.

\item The $\Sigma_n$-satisfaction relation over $J_\alpha$,
  $\models^{\Sigma_n}_{J_\alpha}$, is $\Sigma_n$-definable over
  $J_\alpha$, uniformly in $\alpha$.

\item The mapping $\beta \mapsto J_\beta$ ($\beta < \alpha$) is
  uniformly $\Sigma_1$-definable over any $J_\alpha$, and hence so is the set $\{J_\beta \colon \beta < \alpha\}$.

\item There is a $\Pi_2$ formula $\varphi_{\Ax{V=L}}$ such that for
  any transitive set $M$,
  \[
    M \models \varphi_{\Ax{V=L}} \; \Leftrightarrow \; \exists \alpha
    \; M = J_\alpha.
  \]
  In particular, $\varphi_{\Ax{V=L}}$ says that for every set $a$ there exists an ordinal $\alpha$ such that $a$ is in the rud-closure of $J_\alpha$.

\end{itemize}

The $J$-hierarchy shares all important metamathematical features with
the $L$-hierarchy.  We cite the two most important facts. The
$L$-versions of the two propositions together constitute the core of
Gödel's proof that \Ax{GCH} and \Ax{AC} hold in $L$.

\begin{prop} \label{prop:definability-J}
  There exists a $\Sigma_1$-definable well-ordering $<_J$ of $L$ and
  for any $\alpha > 1$, the restriction of $<_J$ to $J_\alpha$ is
  uniformly $\Sigma_1$-definable over $J_{\alpha}$.
\end{prop}

\begin{prop}[The condensation lemma for
  $J$] \label{prop:condensation-lemma-J} For any $\alpha$, if
  $X \preceq_{\Sigma_1} J_\alpha$, then there is an ordinal $\beta$
  and an isomorphism $\pi$ between $X$ and $J_\beta$. Both $\beta$ and
  $\pi$ are uniquely determined.
\end{prop}

For proofs of these results for the $J$-hierarchy again refer to~\cite{jensen:1972} or~\cite{devlin:1984}.

\subsection{Projecta and master codes}
\label{sub:projecta}

The definable well-ordering $<_J$ together with the definability of
the satisfaction relation can be used to show that each $J_\alpha$ has
\emph{definable Skolem functions}, essentially by selecting the
$<_J$-least witness that satisfies an existential formula. The
definable Skolem functions can in turn be used to define an 
indexing of $J_\alpha$~\cite[Lemma~2.10]{jensen:1972}.

\begin{prop}[Jensen~\cite{jensen:1972}] \label{prop:Jensen-map-onto_Jalpha}
  For each $\alpha$, there exists a $\Sigma_{1}(J_\alpha)$-definable
  surjection from $\omega\alpha$ onto $J_{\alpha}$.
\end{prop}

While a simple cardinality argument yields that
$|J_\alpha| = |\omega\alpha|$, Jensen's result shows that an
$\omega\alpha$-counting of $J_\alpha$ already exists in $J_{\alpha+1}$. The
indexing is obtained by taking (essentially) the Skolem hull of
$\omega\alpha$ under the canonical $\Sigma_{1}$-Skolem function. The
resulting set $X$ is a $\Sigma_{1}$-elementary substructure of
$J_{\alpha}$, hence by the condensation lemma is isomorphic to some
$J_{\beta}$. The isomorphism taking $X$ to $J_\beta$ is the identity
on all ordinals below $\omega\alpha$, and one can show that this in
turn implies that the isomorphism must be the identity on $X$, i.e.,
$X = J_\beta$.

Boolos and Putnam~\cite{boolos-putnam:1968} first observed that if a new real is
defined in $L_{\alpha+1}$, i.e., if
\[
	\P(\omega) \cap (L_{\alpha+1} \setminus L_\alpha) \neq \emptyset,
\]
then the strong absoluteness properties of $L$ can be used to get a
\emph{definable} $\omega$-counting of $L_\alpha$ (instead of just an
$\alpha$-counting as above). Because, if a new subset $Z$ of $\omega$ is constructed
in $L_{\alpha+1} \setminus L_\alpha$, one can take the Skolem hull of $\omega$
instead of $\omega\alpha$. The resulting $X \cong L_{\beta}$ is still equal to
$L_{\alpha}$, since the definition of the new real applies in the
elementary substructure $L_{\beta}$. If $\beta < \alpha$, then this would
contradict the fact that $Z \not\in L_{\alpha}$.

\begin{prop}[Boolos and Putnam~\cite{boolos-putnam:1968}]
	If $\P(\omega) \cap (L_{\alpha+1} \setminus L_\alpha) \neq \emptyset$, then
  there exists a surjection $f: \omega \to L_\alpha$ in $L_{\alpha+1}$.
\end{prop}

Of course, at some stages no new reals are constructed. Boolos and
Putnam~\cite{boolos-putnam:1968} showed that the first such stage is
precisely the ordinal $\beta_0$, i.e., the least ordinal $\beta$ such
that $L_{\beta} \models \ZF^-$. By Gödel's work, on the other hand, we
know that no new real is constructed after stage $\omega^L_1$.

Jensen~\cite{jensen:1972} vastly extended these ideas into the framework of
\emph{projecta} and \emph{master codes}, which form the core concepts of
\emph{fine structure theory}.

\begin{defn} \label{defn:projectum} For natural numbers $n > 0$ and ordinals
  $\alpha > 0$, the \emph{$\Sigma_n$-projectum} $\rho^n_\alpha$ is equal to the
  least $\gamma \leq \alpha$ such that $\P(\omega\gamma) \cap ( \Sigma_n
  (J_{\alpha}) \setminus J_\alpha ) \neq \emptyset$.
\end{defn}

We put $\rho^0_\alpha = \alpha$. Hence $1 \leq \rho^n_\alpha \leq \alpha$ for
all $n$. As $\rho^n_\alpha$ is non-increasing in $n$, we can also define
\[
  \rho_\alpha = \min_n \rho^n_\alpha \quad \text{ and } \quad n_\alpha = \min \{
  k \colon \rho^k_\alpha = \rho_\alpha \}.
\]

Jensen~\cite[Theorem~3.2]{jensen:1972} proved that the projectum $\rho^n_\alpha$
is equal to the least $\delta \leq \alpha$ such that there exists a function $f$
that is $\Sigma_n(J_\alpha)$-definable over $J_\alpha$ such that $f(D) =
J_\alpha$ for some $D \subseteq \omega\delta$, establishing the analogy with the
Boolos-Putnam result. From this it follows that if $\rho^n_\alpha < \alpha$, it
must be a cardinal in $J_\alpha$, for all $n$.

Jensen gave another characterization of the projectum, which in fact he used as
his original definition in~\cite{jensen:1972}. Suppose $\Tup{M,\in}$ is a
set-theoretic structure. We can extend this structure by adding an additional
relation $A \subset M$. If we do this, we would like the structure to satisfy
some basic set theoretic closure properties. For instance, we would like our
universe to satisfy the comprehension axiom with respect to the new relation,
that is, whenever we pick an $x \in M$, the collection of elements in $x$ that
satisfy $A$ should be in $M$. Such structures are called \emph{amenable}.

\begin{defn}
  Given $A \subseteq M$, the structure $\Tup{M, A}$ is called \emph{amenable},
  if $M$ is a transitive set and
  \[
    \forall x \in M \; [  x \cap A \in M].
  \]
\end{defn}

Jensen~\cite[Theorem~3.2]{jensen:1972} showed that
\begin{align*}
	\rho^n_\alpha =  & \text{ the largest ordinal $\gamma \leq \alpha$ such that} \\
                   & \text{ $\Tup{J_\gamma,A}$ is amenable for any $A \subseteq J_\gamma$ that is in $\Sigma_n(J_\alpha)$}.
\end{align*}

This means the projectum $\rho^n_\alpha$ identifies the ``stable'' core of
$J_\alpha$ with respect to $\Sigma_n$ definability over $J_\alpha$.

Being amenable with rud-closed domain can also be characterized via relative
rud-closedness. This will be important later.

\begin{defn}[Jensen~\cite{jensen:1972}]
	A function $f$ is \emph{$A$-rud} if it can be obtained as a combination of the
  basis functions $F_1, \dots, F_8$ and the function
	\[
	 	F_A(x,y) = x \cap A.
	\]
	A structure $\Tup{M,A}$ is \emph{rud closed} if $f[M^n] \subseteq M$ for all $A$-rud functions $f$.
\end{defn}

\begin{prop}[Jensen~\cite{jensen:1972}]
 	A structure $\Tup{M,A}$, $A \subseteq M$, is rud closed if and only if $M$ is rud closed and $\Tup{M,A}$ is amenable.
\end{prop}

\medskip The existence of a definable surjection between (a subset of)
$\omega\rho^n_\alpha$ and $\Sigma_n(J_\alpha)$ allows for coding
$\Sigma_n(J_\alpha)$ into its projectum. One way this can be implemented is via
so-called \emph{master codes}.

\begin{defn}
  A $\Sigma_n$ \emph{master code} for $J_\alpha$ is a set $A \subseteq
  J_{\rho^n_\alpha}$ that is $\Sigma_n(J_\alpha)$, such that for any $m \geq 1$,
  \[
    \Sigma_{n+m}(J_\alpha) \cap \P(J_{\rho^n_\alpha}) =
    \Sigma_m(\Tup{J_{\rho^n_\alpha}, A}).
  \]
\end{defn}

A $\Sigma_n$ master code does two things:
\begin{enumerate}
\item It ``accelerates'' definitions of new subsets of $J_{\rho^n_\alpha}$ by
  $n$ quantifiers.

\item It replaces parameters from $J_\alpha$ in the definition of these new sets
  by parameters from $J_{\rho^n_\alpha}$ (and the use of $A$ as an ``oracle'').
\end{enumerate}

The existence of master codes follows rather easily from the existence of a
$\Sigma_{n}(J_\alpha)$-mapping from $\omega \rho^n_\alpha$ onto $J_{\alpha}$.
However, for $n > 1$, this mapping is \emph{not uniform}. Jensen exhibited a
uniform, canonical way to define master codes, by iterating
$\Sigma_1$-definability.

Put
\[
	A^0_\alpha = \emptyset, \quad p^0_\alpha = \emptyset.
\]

Assuming that $A^n_\alpha$ is a $\Sigma_n$ master code, it is not hard to see that
every set $x \in J_{\rho^n_\alpha}$ is $\Sigma_1$-definable over
$\Tup{J_{\rho^n_\alpha}, A^n_\alpha}$ with parameters from
$J_{\rho^{n+1}_\alpha}$ and one parameter from $J_{\rho^n_\alpha}$ (used to
define a surjection from $\omega\rho^{n+1}_\alpha$ onto $J_{\rho^n_\alpha}$).
Hence we can put
\begin{align*}
	p^{n+1}_\alpha = & \text{ the $<_J$-least $p \in J_{\rho^n_\alpha}$ such that every $u \in J_{\rho^n_\alpha}$ is $\Sigma_1$ definable } \\
                   &  \text{ over $\Tup{J_{\rho^n_\alpha}, A^n_\alpha}$ with parameters from $J_{\rho^{n+1}_\alpha} \cup \{p\}$ }.
\end{align*}
The $p^{n}_\alpha$ are called the \emph{standard parameters}.

Using $p^{n+1}_\alpha$, we can code the $\Sigma_1$ elementary diagram of the
structure $\Tup{J_{\rho^n_\alpha}, A^n_\alpha}$ into a set $A^{n+1}_\alpha$:
\[
	A^{n+1}_\alpha := \{ (i,x) \colon i \in \omega \: \wedge \: x \in
  J_{\rho^{n+1}_\alpha} \: \wedge \: \Tup{J_{\rho^n_\alpha}, A^n_\alpha} \models
  \varphi^{(2)}_i(x,p^{n+1}_\alpha) \},
\]
where $(\varphi^{(k)}_i)$ is a standard Gödel numbering of all $\Sigma_1$
formulas with $k$ free variables. It is not hard to verify that $A^{n+1}_\alpha$
is a $\Sigma_{n+1}$ master code for $J_\alpha$. Furthermore, the structure
$\Tup{J_{\rho^n_\alpha},A^n_\alpha}$ is amenable for each $\alpha > 1$, $n \geq
0$. We will call the structure $\JA{n}$ the \emph{standard $\Sigma_n$
  $J$-structure} for $J_\alpha$.

\begin{defn}
	We denote the standard $J$-structure over $J_\alpha$ at the `ultimate'
  projectum $n_\alpha$ by
	\[
		\Tup{J_{\rho_\alpha}, A_\alpha} := \JA{n_\alpha}. 	
	\] 
\end{defn}

One consequence of the $A^n_\alpha$ being master codes is that we can obtain the
sequence of projecta of an ordinal by iterating taking $\Sigma_1$-projecta
relative to $\Tup{J_{\rho^n_\alpha},A^n_\alpha}$. Given an amenable structure
$\Tup{J_\alpha, A}$, the \emph{$\Sigma_n$-projectum $\rho^n_{\alpha,A}$} of
$\Tup{J_\alpha, A}$ is defined to be the largest ordinal $\rho \leq \alpha$ such
that $\Tup{J_\rho,B}$ is amenable for any $B \subseteq J_\rho$ that is in
$\Sigma_n(\Tup{J_\alpha, A})$.

\begin{prop}[Jensen~\cite{jensen:1972}] \label{prop:relative-proj}
	For $\alpha > 1, n\geq 0$,
	\[
		\rho^{n+1}_\alpha = \rho^1_{\rho^n_\alpha, A^n_\alpha}.
	\]
\end{prop}

In particular, the standard $\Sigma_{n+1}$ $J$-structure for $J_\alpha =
\Tup{J_\alpha, \varnothing}$ is the standard $\Sigma_1$ $J$-structure for
$\JA{n}$.

\subsection{\texorpdfstring{$\omega$-}\ Copies of J-structures} % (fold)
\label{ssec:arithmetic_copies}

We later want to apply the recursion theoretic techniques of Section~\ref{sec-rand-cont-meas} to countable $J$-structures.
We therefore have to code them as subsets of $\omega$. If the projectum $\rho^n_\alpha$ is equal to $1$, all set-theoretic information about the $J$-structure $\JA{n}$ is contained in the master code $A^n_\alpha$, which is simply a real, and hence lends itself directly to recursion theoretic analysis. Starting with the work by Boolos and Putnam~\cite{boolos-putnam:1968}, this has been studied in a number of papers (e.g.\ \cite{Jockusch-Simpson:1976a}, \cite{Hodes:1980a}).

In this subsection we give a recursion theoretic analysis of the internal workings of a countable presentation of a $J$-structure. 

\begin{defn} \label{def:relational_structure}
	Let $X \subseteq \omega$. The \emph{relational structure} induced by $X$ is $\Tup{F_X, E_X}$,
	where
	\[
		x E_X y \Leftrightarrow \Tup{x,y} \in X
	\]
	and
	\[
		F_X = \Fld(E_X) = \{ x \colon \exists y \: (x E_X y \text{ or } yE_X x)  \}.
	\]
\end{defn}

The idea is that a number $x$ is a code for the set whose elements are the sets coded by the numbers $y$ with $y E_X x$,
\[
	\Set_X(x) = \{y \colon y E_X x \}.
\]
For ease of notation, $\Set_X(x)$ will also denote the structure obtained by restricting $E_X$ to $\Set_X(x)$.

The relational structure $\Tup{F_X, E_X}$ is \emph{extensional} if
\[
	\forall x,y \in F_X \: [ (\forall z \: z E_X x \Leftrightarrow z E_X y) \; \Rightarrow x = y ],
\]
that is
\[
	\forall x,y \in F_X \: (x \neq y  \Rightarrow \Set_X (x) \neq \Set_X(y)).
\]
Mostowski's Collapsing Theorem states that if $\Tup{F_X, E_X}$ is extensional and well-founded, it is isomorphic to a unique structure $(M,\in)$, where $M$ is a transitive set. In this sense we can speak of a countable set theoretic structure \emph{coded} by $X$. If $\varphi(v_1, \dots, v_n)$ is a formula in the language of the set theory, we can interpret it over $\Tup{F_X, E_X}$ and write
\[
	X \models \varphi[a_1, \dots, a_n]
\]
for $\Tup{F_X,E_X} \models \varphi[a_1, \dots, a_n]$ with $a_i \in F_X$.

$J$-structures have an additional set $A$, and we capture this on the coding side via pairs $\Tup{X,M}$, where $M \subseteq F_X$. Semantically, $A$ and $M$ are seen as interpreting a predicate added to the language. This way we can consider the satisfaction relation $\Tup{X,M} \models \varphi$, where $\varphi$ is a set-theoretic formula with an additional unary predicate.

We are particularly interested in relational structures that code countable standard $J$-structures. The following is a generalization of the definition due to Boolos and Putnam~\cite{boolos-putnam:1968}

\begin{defn} \label{def:omega-copy}
    \begin{enumerate}[(1)]
	\item An \emph{$\omega$-copy} of a countable, extensional, set-theoretic structure $\Tup{S,A}$, $A \subseteq S$, is a pair $\Tup{X,M}$ of subsets of $\omega$ such that $X$ codes the structure \Tup{F_X,E_X} in the sense of Definition~\ref{def:relational_structure}, and such that there exists a bijection $\pi: S \to F_X$ such that
	\begin{equation} \label{equ:arith-iso}
		\forall x, y \in S \: [ \: x \in y \; \iff \; \pi(x) E_X \pi(y) \: ],
	\end{equation}
	and
	\begin{equation} \label{equ:arith-mastercode}
		M = \{ \pi(x) \colon x \in A \}.
	\end{equation}
	\item When $\pi$ is an isomorphism between a $J$-structure $\JA{n}$ and the structure coded by $\Tup{X,M}$ in the sense of (1), we say that $\Tup{X,M}$ is an \textit{$\omega$-copy of $\JA{n}$ via $\pi$}.
	\end{enumerate}
\end{defn}

The definition thus means an $\omega$-copy $\XM$ of $\Tup{S,A}$ is isomorphic to $\Tup{S,A}$ when seen as structures over the language of set theory.

If $A= \varnothing$, then necessarily $M = \varnothing$, and in this case we say $X$ is an $\omega$-copy of $S$.

\medskip
We will now consider $\omega$-copies of standard $J$-structures.

\medskip
If $\rho^n_\alpha = 1$, we have $J_{\rho^n_\alpha} = L_\omega = V_\omega$, i.e., the \emph{hereditarily finite sets}. In this case, we obtain an $\omega$-copy by fixing a recursive bijection $\pi_\omega$ from $J_1$ to $\omega$, by which we mean that $\Delta_1(J_1)$ relations are mapped to recursive relations on $\omega$ uniformly. Such a mapping can be found, for instance, in~\cite{Ackermann:1937a}.

We let $\Tup{x,y} \in X$, i.e., $x \, E_X \, y$, if and only if $\pi^{-1}_\omega(x) \in \pi^{-1}_\omega(y)$ and $x \in M$ if and only if $\pi^{-1}_\omega(x) \in A^{n+1}_\alpha$. 
Then $\Tup{X,M}$ is an $\omega$-copy of $\JA{n+1}$ via $\pi_\omega$.

\begin{defn}\label{def:canonical-copy}
  When $\rho^n_\alpha = 1$, the \emph{canonical copy} of $\JA{n}$ is the $\omega$-copy defined above.
\end{defn}

We will now show that from a canonical copy of $\JA{n+1}$, we can extract $\omega$-copies of all $\JA{i}$, $i \leq n$, in an effective and uniform way.

By choice of $p^{n+1}_\alpha$, for every $u\in J_{\rho^n_\alpha}$, there exists a $\Sigma_1$-formula $\psi(v_0,v_1,v_2)$ and $x \in J_{\rho^{n+1}_\alpha}$ such that $u$ is the only solution over $\Tup{J_{\rho^n_\alpha}, A^n_\alpha}$ to $\psi(v_0, x, p^{n+1}_\alpha)$. 

Recall that $(\varphi^{(k)}_i)$ is a standard Gödel numbering of all $\Sigma_1$ formulas with $k$ free variables. We can assume that every formula $\varphi_i^{(k)}$ has a lead existential quantifier (by adding dummy variables if necessary). Let $\theta_i^{(k+1)}$ be such that  $\varphi_i^{(k)}$ is  $\exists v_0\, \theta_i^{(k+1)}$.

\begin{defn}
A pair $(i,x)$, $i \in \omega, x \in \JP{n+1}$ is an $n$-\emph{code} if there exists a $u \in \JP{n}$ such that $u$ is the unique solution to
	\[
		\JA{n} \models \theta_i^{(3)}(v_0, x, p^{n+1}_\alpha).
	\]
\end{defn}

We can check the property of being an $n$-code using $\Sigma_1$ formulas: $(i,x)$ is an $n$-code if and only if
\begin{equation} \label{equ:code-1}
	\Tup{J_{\rho^n_\alpha}, A^n_\alpha} \models  \exists v_0 \, \theta^{(3)}_i(v_0,x,p^{n+1}_\alpha) 
\end{equation}
and
\begin{equation} \label{equ:code-2}
	\Tup{J_{\rho^n_\alpha}, A^n_\alpha} \nmodels \exists v_0, v_1 \: ( \theta^{(3)}_i(v_0, x, p^{n+1}_\alpha) \wedge \theta^{(3)}_i(v_1, x, p^{n+1}_\alpha) \wedge v_0 \neq v_1).
\end{equation}

This means a standard code has the information necessary to sort out $n$-codes among its elements. Relative to a canonical copy of $\JA{n+1}$, it is decidable whether some number is the image of an $n$-code, due to the effective way we translate between finite sets and their codes.

If $\Tup{X,M}$ is an arbitrary $\omega$-copy of $\JA{n+1}$ via $\pi$, and $(i,x) \in \JP{n+1}$ is an $n$-code, then we call $\pi( (i,x) )$ a $\pi$-$n$-code.
Being able to decide relative to $\XM$ whether a number is a $\pi$-$n$-code of a pair hinges on knowledge of the following two functions\begin{enumerate}[(i)]
	\item the mapping $n_{X}: n \mapsto \pi(n)$ ($n \in \omega$),
	\item the mapping $h_{X}: (\pi(x), \pi(y)) \mapsto \pi((x,y))$.
\end{enumerate}
These mappings may not be computable relative to $\XM$, but they are definable, as follows.

If $\alpha > 1$, then $\omega \in J_\alpha$, and an $\omega$-copy of any such $J_\alpha$ must contain a witness for $\omega$. In this case, we can recover $n_X$ recursively in $X'$.

\begin{lem} \label{pro:recover-nat-num}
	If $X$ is an $\omega$-copy of $J_\alpha$, $\alpha \geq 1$, then the function $n_X$ is uniformly computable in $X'$.
\end{lem}

\begin{proof}
    We define $n_X$ by recursion. $n_X(0)$ is the unique element of $X$ that represents the empty set in $X$, i.e.\ $\forall x\:\neg \, x E_X n_X(0)$.
    
    Similarly, $n_X(i+1)$ is the unique element of $X$ such that
    $$\forall a (a E_X n_X(i+1) \: \to \: a = n_X(i) \vee a E_X n_X(i) ).$$
    
    These properties are $\Pi^0_1$ uniformly in $X$, hence the function $n_X$ is recursive in $X'$.
\end{proof}

This argument applies more generally as follows.

\begin{lem} \label{pro:compute-canonical-copy}
	If $\XM$ is an $\omega$-copy of $\JA{n}$,  and $\rho^n_\alpha = 1$, then $\XM^{'}$ uniformly computes the canonical copy of $\JA{n}$.
\end{lem}

\begin{proof}

    We want to decide whether a natural number $a$ is in the image of $A^n_\alpha$ under $\pi_\omega$. Given $a$, we first compute the diagram of the transitive closure of the finite set represented by $a$ under $\pi_\omega$. For any reasonable coding $\pi_\omega$ this will be recursive, but in the worst case it will be recursive in $\emptyset'$. 
    
    Next, we take the diagram and search inside $X$ for the natural number $b$ representing this set. This is a finite recursion that can performed recursively in $X'$.
    
    Then $a \in \pi_\omega[A^n_\alpha]$ if and only $b \in M$.
    Therefore, $\pi_\omega[A^n_\alpha]$ is recursive in $\XM'$.
\end{proof}

We can also recover the function $h_X$ arithmetically in $X$.

\begin{lem} \label{pro:compute-pair-function}
	If $X$ is an $\omega$-copy of $J_\alpha$, then the function $h_X$ is uniformly computable in $X'$.
\end{lem}

\begin{proof}
	We have 
	\begin{align*} 
	h_X(\pi(x), \pi(y)) = b \; \Leftrightarrow \; \exists c, d \:   \bigl [  & \forall z \, (z E_X c \: \Leftrightarrow z = \pi(x) ) \\
							& \and \forall z \, (z E_X d \Leftrightarrow z = \pi(x) \: \vee \: z = \pi(y) )  \\ 
              & \and \forall z \, ( z E_X b \: \Leftrightarrow \: z = c \: \vee \: z = d ) \bigr ]
	\end{align*}
	It follows that $h_X$ is a $\Sigma^0_2(X)$ function and as usual, a $\Sigma^0_2(X)$ function is $\Delta^0_2(X)$.
\end{proof}

\begin{defn} \label{def:effective-copy}
	Suppose $\Tup{X,M}$ is an $\omega$-copy via $\pi$ of a rud closed structure $\Tup{J,A}$. We say $\Tup{X,M}$ is \emph{effective} if the function $n_{X}$ is recursive $X\join M$ and the function
	$h_{X}$ is partial recursive in $X\join M$ and recursive in $X\join M$ on its domain $F_X \times F_X$.
\end{defn}

\begin{lem}
 If $\rho^n_\alpha = 1$, the canonical copy of $\JA{n}$ is effective.
\end{lem}

\begin{proof}
	The mapping $\pi^{-1}_\omega$ satisfies the conditions required in Definition~\ref{def:effective-copy} naturally.
\end{proof}

\begin{lem} \label{lem:effective-decompose}
	If $\Tup{X,M}$ is an effective copy of $\JA{n+1}$ via $\pi$, then the mapping
	\[
		\pi( (u,v)) \mapsto (\pi(u), \pi(v)),
	\]
	where $u,v \in \JP{n+1}$, is (partial) recursive in $X\join M$. The mapping
	\[
		\pi( (i,u)) \mapsto (i,\pi(u)),
	\]
	where $i \in \omega$, is also (partial) recursive in $X \join M$.
\end{lem}

\begin{proof}
	The first mapping can be computed by inverting $h_{X}$ (which must be one-to-one), the second mapping by additionally inverting $n_{X}$.
\end{proof}

\begin{lem} \label{lem:pi-n-code}
If $\Tup{X,M}$ is an effective copy of $\JA{n+1}$ via $\pi$, then it is decidable in $X \join M$ whether a number $y \in \omega$ is a $\pi$-$n$-code.
\end{lem}

\begin{proof}
	Suppose $y \in M$ (if not, it cannot be a $\pi$-$n$-code). Then $y = \pi ( (i,x))$ for some $(i,x) \in A^{n+1}_\alpha$. Since the copy is effective, we have $$\pi ( (i,x)) = h_X(\pi(i), \pi(x)),$$ and by Lemma~\ref{lem:effective-decompose} we can find $i$ and $\pi(x)$ recursively in $X\join M$.

	Since $(\varphi^{(2)}_i)$ is a standard Gödel numbering of the $\Sigma_1$ formulas with two free variables, there exist recursive functions $g_1, g_2$ such that $\varphi^{(2)}_{g_1(i)}$ and $\varphi^{(2)}_{g_2(i)}$ are $\Sigma_1$ formulas equivalent (over $\JA{n}$) to the formulas in \eqref{equ:code-1} and \eqref{equ:code-2}, respectively. Then $(i,x)$ is a $n$-code if and only if $(g_1(i), x) \in A^{n+1}_\alpha$ and $(g_2(i), x) \not\in A^{n+1}_\alpha$.

	The latter two conditions are equivalent to
	\[
		h_X(\pi(g_1(i)), \pi(x)) \in M \text{ and } h_X(\pi(g_2(i)), \pi(x)) \not\in M,
	\]
	which, since $\XM$ is an effective copy, is recursive in $X\join M$.
\end{proof}

Two $n$-codes $(i_0,x_0)$ and $(i_1,x_1)$ \emph{represent the same set} $u \in \JP{n}$ if $u$ is the unique solution to $$\JA{n} \models \theta^{(3)}_{i_0}(v_0, x_0, p^{n+1}_\alpha) \; \text{ and } \;  \JA{n} \models \theta^{(3)}_{i_1}(v_0, x_1, p^{n+1}_\alpha).$$ A similar property can be defined for $\pi$-$n$-codes.

\begin{lem} \label{lem:code-extensional}
If $\Tup{X,M}$ is an effective copy of $\JA{n+1}$ via $\pi$, then it is decidable in $X\join M$ whether two numbers are $\pi$-$n$-codes of the same set.
\end{lem}

\begin{proof}
	$(i_0,x_0)$ and $(i_1,x_1)$ represent different sets if and only if
	\begin{multline*}
		\Tup{J_{\rho^n_\alpha}, A^n_\alpha} \models \exists v_0, v_1, v_2 \: \bigl [ \theta^{(3)}_{i_0}(v_0, x_0, p^{n+1}_	\alpha) \: \wedge \: \theta^{(3)}_{i_1}(v_1, x_1, p^{n+1}_\alpha) \: \wedge \:  \\
	  	(v_2 \in v_0 \wedge v_2 \not\in v_1) \vee (v_2 \not\in v_0 \wedge v_2 \in v_1)  \bigr ].
	\end{multline*}
	Let $g_3(i_0,i_1)$ be a Gödel number for the $\Sigma_1$ formula
	\begin{equation*}
		\psi(x,p^{n+1}_\alpha) \equiv  \exists v_0, v_1 \: \bigl [ \theta^{(3)}_{i_0}(v_0, (x)_0, p^{n+1}_	\alpha) \: \wedge \: \theta^{(3)}_{i_1}(v_1, (x)_1, p^{n+1}_\alpha) \: \wedge \: 
	  	v_0 \neq v_1  \bigr ]
	\end{equation*}
 	Then $(i_0,x_0)$ and $(i_1,x_1)$ represent different sets if and only if
	\[
		(g_3(i_0,i_1), (x_0,x_1) ) \in A^{n+1}_\alpha,
	\]
	which in turn holds if and only if
	\[
		h_X(\pi(g_3(i_0,i_1)), h_X(\pi(x_0),\pi(x_1))) \in M.
	\]
	Since $g_3$ is computable, it follows from Lemma~\ref{lem:effective-decompose} that it is decidable in $X\join M$ whether two numbers are $\pi$-$n$-codes and whether they represent the same set.
\end{proof}

\begin{lem} \label{pro:compute-n-copy}
	If $\Tup{X,M}$ is an effective copy of $\JA{n+1}$ via $\pi, n_X,h_X$, it computes an $\omega$-copy $\Tup{Y,N}$ of $\JA{n}$. Furthermore, the computation is uniform, and $h_Y$ and $n_Y$ can be computed uniformly from $X \join M$.
\end{lem}

\begin{proof}
	By Lemmas~\ref{lem:pi-n-code} and~\ref{lem:code-extensional}, the set
	\[
		U = \{ y \in \omega \colon \exists \, u \in \JP{n} \: ( y \text{ is the $<_\omega$-least $\pi$-$n$-code for $u$}) \}
	\]
	is recursive in $X\join M$.

	Let $\sigma$ be the mapping
	\[
		\sigma : u \in J_{\rho^n_\alpha} \mapsto \text{ the unique $\pi$-$n$-code of $u$ in $U$ },
	\]
	and put
	\[
		Y = \{ \Tup{\sigma(x),\sigma(y)} \colon x \in y \in \JP{n} \}, \quad N = \{ \sigma(x) \colon x \in A^n_\alpha \}.
	\]
	Then $\Tup{Y,N}$ is clearly an $\omega$-copy of $\JA{n}$. To show that it is recursive in $X\join M$, we note that for $u,w \in \JP{n}$, if $(i,x)$ is an $n$-code for $u$ and $(j,y)$ is an $n$-code for $w$,
	\begin{equation} \label{equ:code-element}
	u \in w \; \Leftrightarrow \; \Tup{J_{\rho^n_\alpha}, A^n_\alpha} \models \exists v_0,v_1 \: (\theta^{(3)}_i(v_0, x, p^{n+1}_\alpha) \wedge \theta^{(3)}_j(v_1, y, p^{n+1}_\alpha) \wedge v_0 \in v_1).
	\end{equation}
	Moreover,
	\begin{equation} \label{equ:code-standcode}
	u \in A^n_\alpha \; \Leftrightarrow \; \Tup{J_{\rho^n_\alpha}, A^n_\alpha} \models \exists v_0 \: ( \theta^{(3)}_i(v_0, x, p^{n+1}_\alpha) \wedge v_0 \in A^n_\alpha).
	\end{equation}
	There are recursive functions $g_4, g_5$ that output Gödel numbers for $\Sigma_1$ formulas equivalent to the ones in \eqref{equ:code-element} and \eqref{equ:code-standcode}, respectively. Given two numbers $a,b \in U$, we can use Lemma~\ref{lem:effective-decompose} to find $(i,a_0)$ and $(j,b_0)$ such that $a = h_X(\pi(i),\pi(a_0))$, $b = h_X(\pi(j),\pi(b_0))$. Then
	\[
		a E_Y b \; \Leftrightarrow \; h_X(\pi(g_4(i,j)), h_X(a_0,b_0)) \in M.
	\]
	Likewise,
	\[
		a \in N \; \Leftrightarrow \; h_X(\pi(g_5(i)), \pi(a_0)) \in M.
	\]

	To see that the functions $h_Y$ and $n_Y$ are uniformly recursive in  $X \join M$ note that we can
	\begin{enumerate}[(i)]
		\item given $i \in \omega$, effectively compute the Gödel number of a $\Sigma_1$ formula that is satisfied by $u$ if and only if $u$ \emph{is} the natural number $i$,

		\item given $n$-codes $(i,x)$, $(j,y)$ for elements $u,w$ in $\JP{n}$, compute a Gödel number for a $\Sigma_1$ formula whose only solution is $(u,w)$.

	\end{enumerate}

\end{proof}

By iterating the procedure described above, we obtain the following.

\begin{cor} \label{cor:effective-copies-from-above}
	If $\Tup{X,M}$ is an effective copy of $\JA{n+1}$, then it computes $\omega$-copies of
	\[
		\JA{n}, \JA{n-1}, \dots, \text{ and } \JA{0} = \Tup{J_\alpha, \varnothing} = J_\alpha.
	\]
\end{cor}

If a copy is not effective, we can use Lemma~\ref{pro:compute-pair-function} to decode the predecessor $J$-structures.

\begin{cor} \label{cor:copies-from-above}
		If $\Tup{X,M}$ is an $\omega$-copy of $\JA{n+1}$, then $(X \join M)^{(2)}$ computes $\omega$-copies of
	\[
		\JA{n}, \JA{n-1}, \dots, \text{ and } \JA{0} = \Tup{J_\alpha, \varnothing} = J_\alpha.
	\]
\end{cor}

\begin{lem} \label{prop:copy-computes-below}
	If $X$ is an $\omega$-copy of $J_\alpha$, then $X$ computes an $\omega$-copy of $\JA[\beta]{n}$, for all $n \in \omega$, $\beta < \alpha$.
\end{lem}

\begin{proof}
	Both $J_{\rho^n_\beta}$ and $A^n_\beta$ are elements of $J_\alpha$. Let $\pi$ be the isomorphism between $J_\alpha$ and $X$, and let $x_\beta, a^n_\beta \in F_X$ be such that
	\[
		x_\beta = \pi(J_{\rho^n_\beta}), \quad a^n_\beta = \pi(A^n_\beta).
	\]
	Then $\Tup{\Set_X(x_\beta), \Set_X(a^n_\beta)}$, wherein $\Set_X(x_\beta)$ is regarded as a structure and $\Set_X(a^n_\beta)$ is regarded as a set. This is an $\omega$-copy of $\JA[\beta]{n}$, clearly recursive in $X$.
\end{proof}

A similar argument yields an analogous fact for the $S$-operator.

\begin{lem} \label{prop:copy-computes-S-below}
	If $X$ is an $\omega$-copy of $J_\alpha$, then $X$ computes an $\omega$-copy of $S^{(n)}(J_\beta)$, for all $n \in \omega$, $\beta < \alpha$.
\end{lem}

\bigskip

\subsection{Defining \texorpdfstring{$\omega$-}\ copies} \label{ssec:defining-copies}

In the previous section we saw how to effectively extract information from
$\omega$-copies of $J$-structures. Next, we describe how $\omega$-copies of new
$J$-structures can be defined from $\omega$-copies of given $J$-structures.

The $J$-hierarchy has two types of operations that we need to capture: defining
new sets using the $S$-operator, and taking projecta and defining standard
codes. We will analyze both operations from an arithmetic perspective.

\subsubsection*{An arithmetic analogue of the $S$-operator}
\label{ssub:an_arithmetic_analogue_of_the_s_operator}

The $S$-operator is defined by application of a finite number of explicit functions. This makes it possible to devise an arithmetic analogue, which we denote by $\overline{S}$, and which is the subject of the following lemma.

\begin{lem}\label{lem:arithmetic-S-copy}
  There exists an arithmetic function $\overline{S}(X)=Y$ such that, if $X$ is
  an $\omega$-copy of a transitive set $U$, $\overline{S}(X)$ is an
  $\omega$-copy of the transitive closure of $S(U)$. Further, $X$ is coded into a reserved column of $\omega$, that
  is,
    \[
      \Tup{x,y} \in X \Leftrightarrow \Tup{2^x,2^y} \in \overline{S}(X),
    \]
    and $3$ represents the element $\{U\}$ in $\overline{S}(X)$.
\end{lem}

\begin{proof}
  The elements of $S(U)$ are obtained by single applications of the functions
  $F_0,\ldots, F_8$. Thus each element of $S(U)$ is the denotation of a term
  consisting of one of the functions and finitely many elements of $U \cup
  \{U\}$. We denote the set of terms by $T_S(U)$. Arithmetically in $X$, we can define an $\omega$-copy $\overline{S}(X)$ of $S(U)$ from the natural $X$-recursive copy $T_S(X)$ of $T_S(U)$.
  Membership of the set denoted by one term in the set denoted by another term
  or equality between the sets denoted by terms is arithmetic in $X$, since
  these are defined by quantification over $X$. The same applies for elements
  of the transitive closure of the thus coded structure.  The additional
  uniformity condition on the coding of $X$ does not change the calculation.
\end{proof}

We can subject the $\overline{S}$-operator to an analysis similar to that of the
jump operator by Enderton and
Putnam~\cite{enderton-putnam:note-hyperarithmetical_1970}.

\begin{lem} \label{lem:recog-S-iterates}
  There exists an arithmetic predicate $Q(n,X,Y)$ such that if $A$ is an $\omega$-copy, $Q(n,A,Y)$ holds if and only if $Y = \overline{S}^{(n)}(A)$.
 \end{lem}

\begin{proof}
Suppose the relation $\overline{S}(X)=Y$ is definable by a $\Pi^0_k$ formula. Given $Y \subset \omega$, let 
\[
  Y_{[2]} = \{ \Tup{a,b} \colon \Tup{2^a, 2^b} \in Y\}.
\]
For each $n \geq 1$, the relation $\overline{S}^{(n)}(X)=Y$ is definable via the formula
\[
  \psi_n \equiv \; Y = \overline{S}(Y_{[2]}) \: \wedge \: Y_{[2]} = \overline{S}((Y_{[2]})_{[2]}) \: \wedge \: \dots \: \wedge \: (Y)^{(n-1)}_{[2]} = \overline{S}(X),
\]
where $(\, . \,)^{(m)}_{[2]}$ denotes the $m$-th iterate of the $(.)_{[2]}$-operator.	Each formula $\psi_n$ is also $\Pi^0_k$, and $\psi_n$ can be found recursively and uniformly in $n$. Let $P(n,X,Y)$ be a universal $\Pi^0_k$ predicate, that is, for every $\Pi^0_k$ predicate $R(X,Y)$ there exists an $n \in \omega$ such that 
\[
  P(n,X,Y) \; \Leftrightarrow \; R(X,Y),
\]
and $n$ can be found recursively from a formula defining $R$. Let $i(n)$ be the index with respect to $P$ of the $\Pi^0_k$ predicate defined by $\psi_n$. The predicate 
\[
  Q(n,X,Y) \; :\Leftrightarrow \; P(i(n),X,Y)
\]
has the property claimed in the statement of the Lemma.
\end{proof}

\begin{lem}
	If $Z$ is such that $Z \geq_{\T} \overline{S}^{(n)}(X)$ for all $n$, then
  $\bigoplus_n \overline{S}^{(n)}(X)$ is uniformly arithmetically definable from
  $Z$.
\end{lem}

\begin{proof}
	Define the predicate $\overline{Q}(n,e)$ as
	\[
		\overline{Q}(n,e) \: :\Leftrightarrow \: \Phi^Z_e \text{ is total and } Q(n,X,\Phi^Z_e).
	\]
	$\overline{Q}$ is an arithmetic predicate relative to $Z$ since $Z \geq_{\T} \overline{S}(X) \geq_{\T} X$. To decide whether $a \in \overline{S}^{(n)}(X)$, find, arithmetically in $Z$,
  the least $e$ such that $\overline{Q}(n,e)$ and compute $\Phi^Z_e(a)$.
\end{proof}

\begin{cor}\label{cor:successor-copy}
	If $X$ is an $\omega$-copy of $J_\alpha$ and $Z \geq_{\T}
  \overline{S}^{(n)}(X)$ for all $n$, then $Z$ uniformly arithmetically defines
  an $\omega$-copy of $J_{\alpha+1}$.
\end{cor}

\begin{proof}
	We can use $\bigoplus_n \overline{S}^{(n)}(X)$ to define a copy of
  $J_{\alpha+1}$ by `stacking' the elements of $\overline{S}^{(n+1)}(X)$ coded
  with base $3$ and higher at the next `available' prime column. Essentially
  this means that instead of moving $\overline{S}^{(n)}(X)$ into the column
  given by powers of $2$, we leave it unchanged and add new elements for
  $\overline{S}^{(n+1)}(X)$ starting at the smallest available prime column.
\end{proof}

\subsubsection*{An arithmetic version of the standard code}
\label{ssub:an_arithmetic_version_of_the_projectum}

To define an arithmetic copy the $\Sigma_n$-standard code for $J_\alpha$, we can
simply interpret the set theoretic definitions as formulas of arithmetic. More
precisely, suppose $P$ is a definable predicate over a $J$-structure $\JA{n}$,
and $\Tup{X,M}$ is an $\omega$-copy via $\pi$. Since the structure $\Tup{F_X,
  E_X, M}$ is isomorphic to $\Tup{\JP{n}, \in, A^n_\alpha}$, we can use the same
formula that defines $P$ over $\JA{n}$ and obtain a definition of $\pi[P]$
arithmetic in $\Tup{X,M}$. The problem, however, is that a bounded quantifier in
set theory will not necessarily correspond to a bounded quantifier in
arithmetic. This means the transfer of complexities between the Lévy-hierarchy
and the arithmetical hierarchy may not result in uniform bounds.

However, we will use only a fixed, finite number of set-theoretic definitions.
Most importantly, we use the uniform definability of the satisfaction relation
$\models$ over transitive, rud closed structures.

\begin{prop}[Jensen~\cite{jensen:1972}, Corollary~1.13] \label{pro:def-sat}
  For $n \geq 1$, the satisfaction relation $\models^{\Sigma_n}_{\Tup{M,A}}$ is
  uniformly $\Sigma_n(\Tup{M,A})$ for transitive, rud closed structures
  $\Tup{M,A}$.
\end{prop}

\begin{cor} \label{lem:successor-copy}
	Suppose $\XM$ is an $\omega$-copy of $\JA{n}$. 
  Then there exists an $\omega$-copy of $\JA{n+1}$ uniformly arithmetically definable in $\XM$.
\end{cor}

\begin{proof}
  The projectum $\rho^{n+1}_\alpha$ and the standard parameter are uniformly first-order definable over $\Tup{J_{\rho^n_\alpha}, A^n_\alpha}$ (uniformly in $n$). Then the standard code $A^{n+1}_\alpha$ is uniformly defined over the same structure using the parameter $\rho^{n+1}_\alpha$. These uniform first-order definitions become uniform arithmetic definitions in $\omega$-copies, which yields the corollary.
\end{proof}

\bigskip

\subsection{Recognizing J-structures}
\label{ssec:arith-master-codes}

Our goal is to show that there exists a recursive function $G$ such that, for
each $n$, no element of the sequence of canonical copies of $J$-structures with
projectum equal to one in $\Lb{n}$ can be $G(n)$-random with respect to a continuous
measure. 
In the proof of this result (Theorem~\ref{thm:mastercodes-in-NCR}), we need to consider the initial segment of $\omega$-copies computable in
(some fixed jump of) $\mu$. 

The problem is that we cannot arithmetically define the set of $\omega$-copies
of structures $J_\alpha$. We can, however, define a set of ``pseudocopies'',
subsets of $\omega$ that behave in most respects like $\omega$-copies of actual $J_\alpha$, 
but that may code structures that are not well-founded. We can, arithmetically in the pseudocopy, require that its version of the natural numbers (if any) is isomorphic to $\omega$, that is to say that it is an \emph{$\omega$-model}.

By comparing the structures coded by these pseudocopies, we can also linearly
order a subset of the latter (up to isomorphism), depending on whether a coded structure
embeds into another. This ordering will be developed in Section~\ref{ssec:comparing-copies}.

To define what it means to be a pseudocopy, we have to formulate ``pseudo''-versions of the properties that characterize the $J_\alpha$'s. These properties will stand in for the use of the Condensation Lemma, which we lack in non-wellfounded structures. One such property is rudimentary closure. A first-order approximation $\varphi_{\Op{Rud}}$ is the assertion that the formulas which specify the graphs of $F_0, \dots, F_8$ from Proposition~\ref{prop:rud-base-functions} define total functions.

We also need to require the presence of a linearly-ordered, internal $J$-hierarchy inside a pseudocopy. By Proposition~\ref{prop:definability-J}, for any $\beta$, the sequence of
$J_\alpha$ ($\alpha < \beta$) is uniformly $\Sigma_1$-definable over $J_\beta$.
Let $\varphi$ be the formula defining this collection.
For an $\omega$-copy $X$, 
we define
\[
	\vec{J}^X = \{z \colon X \models \varphi(z)\},
\]
the $J$-structure inside $X$.

We will also apply some aspects of the fine structure of the $J$-hierarchy by requiring that pseudocopies exhibit the same features: Elements of the $J$-hierarchy satisfy $V=L$ and that there exists a $\Sigma_1$-map from the ordinals of the structure onto the whole structure. 

\begin{lem}\label{lem:universal-Sigma_1}
If $X$ is an $\omega$-copy that satisfies $\varphi_{\Op{Rud}}$ and is an $\omega$-model, then the satisfaction predicate $\models^{\Sigma_1}$ for $\Sigma_1$-formulas is $\Sigma_1$-definable in $\Tup{F_X, E_X}$.
\end{lem}

\begin{proof}
The proof follows the argument given by Jensen for transitive, rud-closed sets~\cite[Lemma~1.12, Corollary~1.13]{jensen:1972}. The argument does not use the full well-foundedness of the $\in$-relation, but rather works with $\omega$-models.
\end{proof}

\begin{defn} \label{def:pseudocopy}
  A set $X \subseteq \omega$ is a \emph{pseudocopy} if the following hold.
  \begin{enumerate}[(1)]
  \item The relation $E_X$ is non-empty and extensional.

  \item $X$ codes an $\omega$-model: if in the structure coded by $X$ there is an upper bound on the codes of finite ordinals in $X$, then there is a code for a least upper bound in $X$.

  \item The structure $\Tup{F_X, E_X}$ satisfies $\varphi_{\Op{Rud}}$.

    \item The structure $\Tup{F_X, E_X}$ satisfies 
  \begin{enumerate}[(a)]
      \item $\varphi_{V=L}$,
      \item the ordinals in the structure are linearly ordered by $E_X$ and no two distinct ordinals are isomorphic, 
      \item there is a $\Sigma_1$-map from the ordinals of $\Tup{F_X, E_X}$ onto $F_X$ and for every $a \in F_x$ such that $$\Tup{F_X, E_X} \models \text{``$a$ is an infinite cardinal''},$$ it holds that
      \[
      \Tup{F_X, E_X} \models \text{``there is a bijection between $a$ and $a \times a$''}.
      \]
      \item the elements represented by $\vecJ^X$ are transitive and linearly ordered by the $\in$-relation.
      \item if $\varphi_<$ is the $\Sigma_1$-formula that defines $<_J$ as in Proposition~\ref{prop:definability-J}, this formula $\varphi_<$ defines a linear ordering in the structure $\Tup{F_X, E_X}$ such that
      \begin{enumerate}[(i)]
        \item every $\Sigma_3$-definable property has a least element with respect to this linear ordering,
        \item for every $z \in \vecJ^X$, the ordering defined by $\varphi_<$ inside $z$ is an initial segment of the ordering in $\Tup{F_X, E_X}$.
      \end{enumerate}
  \end{enumerate}
  Moreover, (a)-(e) hold for every set induced by some $z \in \vecJ^X$. (As $X$ codes an extensional structure, and $z$ is transitive in $X$, this condition is equivalent to saying that $X$ satisfies the five conditions hold with quantifiers bounded by $z$.) 

  \end{enumerate}
\end{defn}

Every $\omega$-copy of a $J_\alpha$ is a pseudocopy.
Below we use various properties of pseudocopies that are true of each $J_\alpha$ and are arithmetically definable for $\omega$-copies. For the sake of efficiency, the reader may assume that these properties are part of the definition of a pseudocopy.

\begin{lem}\label{lem:reflection}
  If $X$ is a pseudocopy and $z \in \vecJ^X$, then $\Set_X(z)$ is a pseudocopy.
\end{lem}

\begin{proof}
    (1) holds since, by item (4)(d), if $z \in \vecJ^x$, then $z$ represents a transitive set in $X$. Being extensional is  absolute to transitive elements of $X$ and $E_X$ is extensional. 
    
    For (2), note that $X$ and $\Set_X(z)$ have the same natural numbers.
    
    Item (3) holds for $\Set_X(z)$ since inside $X$, $z$ is rud-closed, and the graphs of the functions $F_0, \dots, F_8$ are absolute between $X$ and $\Set_X(z)$. 
    
    For (4), first note, by Lemma~\ref{lem:universal-Sigma_1}, item  (c) is a first-order property.
    $\Set_X(z)$ itself satisfies (a)-(e) specified in (4) because it is in $\vecJ^X$. Now suppose $w \in \vecJ^{\Set_X(z)}$. Membership in $\vecJ$ is determined by a (uniformly) $\Sigma_1$-formula and these formulas are upward absolute, so $w \in \vecJ^X$. Since (4) holds for $X$, $\Set_X(w)$ satisfies (a)-(e). Since $\Set_{\Set_X(z)}(w) = \Set_X(w)$, the condition that (a)-(e) hold in $\Set_X(w)$ is a formula in which all the quantifiers are bounded by $\Set_X(w)$, and such formulas are absolute between the structures coded by $X$ and by $z$ within $X$, $\Set_X(w)$ satisfies (a)-(e), as required.
\end{proof}

\begin{lem} \label{pro:pseudocopy-arithmetic}
  There exists an arithmetic formula
  $\varphi_{\Op{PC}}(X)$ such that $\varphi_{\Op{PC}}(X)$ holds for a real
  $X$ if and only if $X$ is a pseudocopy. Moreover, if $\Tup{F_X,E_X}$ is well-founded,
  then it is an $\omega$-copy of a countable $J_\beta$, $\beta > 1$.
\end{lem}

\begin{proof}
  Given Lemma~\ref{lem:reflection} and item (2), items (1), (3), and (4) are first-order properties and hence arithmetic relative to $X$. 

  With respect to (2), we can define $\omega$ using the usual $\Sigma_0$ set
theoretic formula (the least infinite ordinal).
In transitive, rud-closed sets, $\varphi_\omega(x)$ holds if and only if $x = \omega$. 
Interpreting $\varphi_\omega$ over $\Tup{F_X,E_X}$, we obtain an arithmetic in
$X$ property. We require a pseudocopy $\Tup{F_X,E_X}$ to satisfy
\[
	\exists x \; \varphi_\omega(x).
\]
This $x$ will be unique and define $\omega$ with respect to $\Tup{F_X,E_X}$. Let
us denote this unique number by $\omega_X$.

Given $\omega_X$, we can also recover the mapping $i \mapsto n_X(i)$ as in
Lemma~\ref{pro:recover-nat-num}. As the definition of $\omega_X$ is uniform, we
obtain that $i \mapsto n_X(i)$ is uniformly arithmetic in $X$. (2) holds exactly
when this map from $\omega$ to $\omega_X$ is a surjection, which is again
uniformly arithmetic relative to $X$.

The conjunction of the arithmetic formula that characterizes (2) with the arithmetic analogues of the first-order formulas characterizing (1), (3), and (4) yields the arithmetic formula $\varphi_{\Op{PC}}$.

\medskip 
By Mostowski's Collapsing Theorem, if $X$ satisfies (1) and $E_X$ is
well-founded, then $\Tup{F_X,E_X}$ is isomorphic to a transitive set structure
$\Tup{S,\in}$, and by (3) $S$ will be rud-closed. Finally, any rud-closed set that satisfies $\varphi_{V=L}$ is a $J_\alpha$.
\end{proof}

\bigskip

\subsection{Comparing pseudocopies} \label{ssec:comparing-copies}

If two pseudocopies $X$ and $Y$ define well-founded structures, they are
$\omega$-copies of sets $J_\alpha$ and $J_\beta$, respectively. Since $\alpha <
\beta$ implies $J_\alpha \in J_\beta$, it follows that one structure must embed
into the other as an initial segment.

We want to find an arithmetic formula that compares two pseudocopies in this
respect. The problem is that, in general, the isomorphism relation between countable
structures need not be arithmetic. In our case, however, we can make use of the
special set-theoretic structure present in the pseudocopies, by comparing the
subsets of the cardinals present, to show that the isomorphism relation is arithmetic.

The complexity of the arithmetic operations involved in these comparisons will
depend on the number of cardinals present in a pseudocopy.

Let us introduce the following notation. Recall that $\beta_N$ denotes the least
ordinal such that $L_{\beta_N} \models \ZF^-_N$. For any ordinal $\alpha$, let
\begin{equation} \label{equ:max-power}
	P_\alpha = \max \{ n \colon \P^{(n)}(\omega) \text{ exists in } J_\alpha \},
\end{equation}
if this maximum exists. We first note that for all $\alpha < \beta_N$, $P_\alpha
\leq N$. This is because, if $P_\alpha$ were greater than or equal to $n+1$ and
$\beta$ were the $(n+1)$st cardinal in $L_\alpha$, then $L_\beta$ would satisfy
$\ZF^-_N$, hence $\alpha > \beta_N$. Hence $P_\alpha$ is defined and uniformly
bounded by $N$ for all $\alpha < \beta_N$.

\medskip
Using the predicate $\varphi_\omega$, we can formalize the
(non-)existence of power sets of $\omega$ for pseudocopies.
For any $k \in \omega$, there exists a formula defining the predicate
$y = \P^{(k)}(\omega)$.

\begin{defn}
	A pseudocopy $X$ is an \emph{$n$-pseudocopy} if it satisfies the uniformly arithmetic in $X$ predicate
	\[
		\exists y ( y = \P^{(n)}(\omega) ) \and \forall z ( z \neq \P^{(n+1)}(\omega) ).
	\]
\end{defn}

We now use the fact that pseudocopies are $\omega$-models. Using the power sets
of $\omega$ in each pseudocopy, we can check whether two pseudocopies have the
same reals, sets of reals, etc.

First, we can check whether every real in $X$ has an analogue in $Y$:
\[
	\forall u \: ( X \models u \subseteq \omega \, \rightarrow \, \exists v ( Y
  \models v \subseteq \omega \, \wedge \, \forall i ( n_X(i) E_X u \,
  \leftrightarrow \, n_Y(i) E_Y v ))).
\]
By extensionality, such a $v$, if it exists, is unique. We can therefore define
the mapping $f_0^{X,Y}(u) = v$ which maps the representation of a real in
$\Tup{F_X,E_X}$ to its representation in $\Tup{F_Y,E_Y}$.
We can similarly check whether every real in $Y$ has an analogue in $X$.
This gives rise to a function $f_0^{Y,X}$. Let $\varphi^{(0)}_{\Op{comp}}(X,Y)$
be the arithmetic formula asserting that $X$ and $Y$ code the same subsets of
$\omega$.

We can continue this comparison through the iterates of the power set of $\omega$. 
This will yield arithmetic formulas $\varphi^{(n)}_{\Op{comp}}(X,Y)$ with the following property:
\begin{quote}
	If $X$ and $Y$ are pseudocopies in which $\P^{(n)}(\omega)$ exists, then
  $\varphi^{(n)}_{\Op{comp}}(X,Y)$ holds if and only if $X$ and $Y$ have
  (representations of) the same subsets of $\P^{(i)}(\omega)$, for all $0 \leq i
  \leq n$.
\end{quote}
Given two $n$-pseudocopies, the above formulas allow for an arithmetic definition of isomorphic pseudocopies.

\begin{lem} \label{lem:compare_well-founded}
	For given $n$ and for any two $n$-pseudocopies $X,Y$ that code well-founded
  structures $\Tup{F_X,E_X}$ and $\Tup{F_Y,E_Y}$, respectively, if
  $\varphi^{(n)}_{\Op{comp}}(X,Y)$, then $X$ and $Y$ code the same $J_\alpha$.
\end{lem}

\begin{proof}
	Assume for a contradiction $X$ and $Y$ are not isomorphic. Since they are
  well-founded pseudocopies, there must exist countable $\alpha, \beta$ such
  that $\Tup{F_X,E_X} \cong (J_\alpha, \in)$ and $\Tup{F_Y,E_Y} \cong (J_\beta,
  \in)$. Without loss of generality, $\alpha < \beta$. Since $\Tup{F_X,E_X}$ and
  $\Tup{F_Y,E_Y}$ code the same subsets of $\P^{(n)}(\omega)$ no new subset of
  $\P^{(n)}(\omega)$ is constructed between $\alpha$ and $\beta$. But this
  implies $\P^{(n+1)}(\omega)$ exists at $\alpha+1$, which is an immediate
  contradiction if $\beta = \alpha+1$. If $\beta > \alpha+1$, since
  $\P^{(n+1)}(\omega)$ does not exist in $J_\beta$, a new subset of
  $\P^{(n)}(\omega)$ must be constructed between $\alpha+1$ and $\beta$,
  contradiction.
\end{proof}

We will consider the comparison between ill-founded structures later.

\begin{cor} \label{cor:transfer-function}
  For given $n$, if $X$ and $Y$ are
  well-founded $n$-pseudocopies and $\varphi^{(n)}_{\Op{comp}}(X,Y)$, then the isomorphism from the 
  structure coded by $X$ to the structure coded by $Y$ is arithmetic in $(X,Y)$. Furthermore, the formula which defines this isomorphism depends only on $n$. 	
\end{cor}

\begin{proof}
  Under the given hypothesis, there exists an $\alpha$ such that both $X$ and
  $Y$ are isomorphic to $J_\alpha$. 
  
  Since $\varphi^{(n)}_{\Op{comp}}(X,Y)$ holds, we can arithmetically in the pair $(X,Y)$ match the power sets of $\omega$ up to $n$, as follows. First, observe that by Lemma~\ref{pro:recover-nat-num}, there is an order-isomorphism between the natural numbers in $X$ and in $Y$. By the proof of the GCH, there are bijections in $J_\alpha$ between the $k$-th uncountable cardinal in $J_\alpha$ and the power sets of its predecessor cardinal in $J_\alpha$. We can use these bijections to inductively find order isomorphisms between the cardinals in $X$ and the cardinals in $Y$. Finally, this yields an order isomorphims between the greatest cardinals in $X$ and $Y$, which is arithmetically definable in $(X,Y)$. 
  
  Next we show how to define a mapping between ordinals of these structures. Let $\gamma$ be an ordinal in $X$. Since $\omega_n$ is the biggest cardinal in $X$, there exists a mapping from $\omega_n$ onto $\gamma$ in $X$. Let $f_X$ be the $<_J$-least such mapping. The function $f_X$ induces a relation on $\omega_n$ in $X$. Since $<_J$ is uniformly $\Sigma_1(J_\alpha)$, arithmetically in $(X,Y)$ we can find the $<_J$-least function $f_Y$ in $Y$ which induces the same relation on $\omega_n$. The image of $\omega_n$ under $f_Y$ represents the ordinal $\gamma$ in $Y$. 
  The result of this process is the desired mapping. 
  
  By the proof of
  Proposition~\ref{prop:Jensen-map-onto_Jalpha} given by Jensen, there is a $\Sigma_1(J_\alpha)$-definable map from
  $\omega\alpha$, that is the set of ordinals in $J_\alpha$, onto $J_\alpha$. The definition of this map involves a parameter $p$ from $J_\alpha$ such that $J_\alpha$ is the $\Sigma_1$ Skolem hull of $p$ and the ordinals of $J_\alpha$. Arithmetically in $(X,Y)$, we can find the $X$-version and the $Y$-version of $p$ and the map from finite sequences of ordinals in $J_\alpha$ corresponding to forming the Skolem hull. From these we get the isomorphism between $X$ and $Y$. Note that the definition of the isomorphism is uniformly arithmetic in the pair $(X,Y)$ and depends only on the number of iterates of the power set of $\omega$ present. 
\end{proof}

We can use the transfer function of 
Corollary~\ref{cor:transfer-function} to translate also between copies of $S^{(n)}(J_\alpha)$.

\begin{cor} \label{cor:transfer-S-copy}
  For every $N$, there exists a number $d_N$, which can be computed uniformly from $N$, such that the following holds. Suppose $X$ is an $\omega$-copy of some
  $J_\alpha$ with $P_\alpha \leq N$. Suppose further that $Z$ is an $\omega$-copy
  of $S^{(n)}(J_\alpha)$, for some $n \in \omega$. Then $\overline{S}^{(n)}(X)$
  is recursive in $(X \join Z)^{(d_N)}$.
\end{cor}

\begin{proof}
There is a $k$ as follows.  For a given $\omega$-copy $A$ of a transitive set $a$ such that $a$ is in the model coded by $Z$, say represented by parameter $b$, and a given isomorphism $f$ between $A$ and the encoding given by $b$ in $Z$, if $S(a)$ is coded in $Z$ then uniformly recursively in $A\oplus f\oplus Z^{(k)}$  there is an isomorphism between 
$\overline{S}(A)$ and $b.$

We obtain the isomorphism as follows. We start with the $\omega$-copy $B$ of the set $a$ as encoded by $b$ inside $Z$. Using the notation of the proof of Lemma~\ref{lem:arithmetic-S-copy}, we form the term set copy $T_S(B)$ of $T_S(a)$. Arithmetically in $Z$, we can define maps $g_1$ from $T_S(B)$ to $\overline{S}(B)$ and $g_2$ from $\overline{S}(B)$ to $Z$'s copy of $S(a)$. Recursively in $f \oplus g_1$, we can obtain an isomorphism between $\overline{S}(A)$ and $\overline{S}(B)$: the isomorphism between the term sets is recursive, and $g_1$ provides the necessary information about identification of terms in $\overline{S}$. 
Finally, application of $g_2$ provides the isomorphism between $\overline{S}(A)$ and $Z$'s copy of $S(a)$. 

Now we turn to the proof of the corollary. Since $Z$ encodes $S^{(n)}(J_\alpha)$, it also encodes $J_\alpha$. Uniformly arithmetically in $Z \oplus X$, we can find the element in $Z$ that encodes $J_\alpha$ and the isomorphism $f$ between $X$ and $Z$'s $\omega$-copy of $J_\alpha$. 

By the above, recursively in $X \oplus f \oplus Z^{(k)}$, we can find an isomorphism, $f_1$, between $\overline{S}(X)$ and $Z$'s $\omega$-copy of $S(J_\alpha)$. By iterating the process, we obtain, recursively in $X \oplus f \oplus Z^{(k)}$, an isomorphism between $\overline{S}^{(n)}(X)$ and $Z$. Since $f$ is arithmetic in $X \oplus Z$, the corollary follows.
\end{proof}

It is worth highlighting that in the previous corollary we obtain a fixed arithmetic bound relative to $Z$ which is independent of $n$.

\medskip
For fixed $N \in \omega$, let
\begin{multline*}
    \PC_N = \{ X \colon X \text{ is an $n$-pseudocopy for some $n \leq N$} \\
    \text{and for every $Y \in \vec{J}^X$, $Y$ is a $k$-pseudocopy for some $k \leq N$} \}
\end{multline*}
	
This is an arithmetic set of reals.

Working inside $\PC_N$, we can also use $\varphi_{\Op{comp}}$ to arithmetically
define a pre-order $\prec$ on pseudocopies. The idea is that $X \prec Y$ if $X$
embeds its structure into $Y$.

\begin{lem} \label{lem:J-rigid}
\begin{enumerate}
    \item Suppose $X \in \PC_N$ then
  the structure $\Struc{X}$ has no non-trivial automorphism. 
  
  \item If for some $z_1, z_2 \in \vec{J}^X$, $\Set_X(z_1)$ is isomorphic
  to $\Set_X(z_2)$, then $z_1 = z_2$.
 
\end{enumerate}
\end{lem}

\begin{proof}
(1) We first argue there is a definable parameter and a mapping, which is $\Sigma_1$ relative to this parameter, from the subsets of the greatest cardinal $\omega_n^X$, where $n \leq N$, onto $\Struc{X}$. By item (4)(c) of Definition~\ref{def:pseudocopy} there is a $\Sigma_1$-definable (relative to some parameter) mapping from the ordinals of $X$ onto $X$. By item (4)(e) among the parameters which can be used to define such a mapping there is a $\varphi_<$-least one (which thus makes the parameter definable).  
Since this parameter must be fixed by any automorphism, to prove (1) it is sufficient to find a mapping from subsets of the greatest cardinal onto the ordinals. 

Given  $a \subseteq \omega_n^X$, using the bijection between $\omega_n^X$ and $\omega_n^X \times \omega_n^X$ (guaranteed by item (4)(c) in Definition~\ref{def:pseudocopy}), construe $a$ as a binary relation on a subset of $\omega_n^X$. Map $a$ to ordinal $\gamma$ if there is an isomorphism between this relation and $\gamma$. Otherwise, map $a$ to $0$. A failure of this mapping's being well-defined would yield an existentially definable isomorphism between two distinct ordinals, contradicting item (4)(b) in Definition~\ref{def:pseudocopy}.
Moreover, the mapping is onto because for every ordinal there is an injection of it into $\omega_n^X$.

Now, by induction on $j \leq n$, an automorphism of the structure $\Struc{X}$ has to fix every subset of $\omega_j^X$. The case $j = 0$ follows from the fact that pseudocopies are $\omega$-models. Hence the only automorphism of $\Struc{X}$ is the identity.

(2) Assume $z_1, z_2 \in \vecJ^X$. By item (4)(d) in the definition of pseudocopies, without loss of generality, $\Set_X(z_1) \subseteq \Set_X(z_2)$. By the argument for (1), an isomorphism from $\Set_X(z_1)$ to $\Set_X(z_2)$ has to be the identity on all subsets of the greatest cardinal of $\Set_X(z_1)$, and hence the identity on all of $\Set_X(z_1)$. Since an isomorphism is surjective, $z_1 = z_2$.
\end{proof}

\begin{lem} \label{lem:isomophism-arithmetic-PCs}
  Suppose $X$ and $Y$ are in $\PC_N$. If $X$ and $Y$ code isomorphic structures ($X \cong Y$), the isomorphism is unique and uniformly arithmetically definable from the pair $(X,Y)$.
\end{lem}

\begin{proof}
Assume $X$ and $Y$ code isomorphic structures. Two different isomorphisms between the structures would yield a non-trivial automorphism of each, contradicting Lemma~\ref{lem:J-rigid}. 

Let $N_0$ be such that there are $N_0$-many infinite cardinals in the structure coded by $X$ and $Y$.   
By induction on $n \leq N_0$ the isomorphism restricted to $\P^{(n)}(\omega)$ is arithmetic in $X \oplus Y$. The base of the induction is provided by the fact that both $X$ and $Y$ are $\omega$-models. Inside each structure, there is a map $\P^{(N_0)}(\omega)$ to its greatest cardinal, $\kappa$, by which the isomorphism restricted to these greatest cardinals is arithmetic in $X \oplus Y$. Consequently, the isomorphism restricted to subsets the greatest cardinals, respectively, is arithmetic in $X \oplus Y$. Since every ordinal in each structure determined by a subset of the greatest cardinal, the isomorphism between the ordinals in each structure is arithmetic in $X \oplus Y$. Finally, to obtain an arithmetic isomorphism between $X$ and $Y$, we use the existence of a $\Sigma_1$-map from the ordinals to the whole structure, as given by item (4)(c) in Definition~\ref{def:pseudocopy}. 

By the previous paragraph, there is a $k$, depending on $N$, by taking into account all the $N_0 \leq N$, such that the isomorphism between the structures code by $X$ and $Y$, respectively, is recursive in $(X\oplus Y)^{(k)}$. Finding an index for this isomorphism is uniformly computable in $(X\oplus Y)^{(k+2)}$. 
\end{proof}

\begin{cor} \label{cor:iso-arithmetic}
    Suppose $X$ and $Y$ are in $\PC_N$. Whether $X \cong Y$ is a uniformly arithmetic property of the pair $(X,Y)$.  
\end{cor}

\begin{defn} \label{def:pc-order}
	We define
	\[
		X \prec_N Y \quad : \Leftrightarrow X,Y \in \PC_N \and \exists z \, ( z \in
    \vecJ^Y \and X \cong \Set_Y(z)).
	\]
\end{defn}
By Corollary~\ref{cor:iso-arithmetic}, this is an arithmetic property of the pair $(X,Y)$. We let $X \preceq_N Y$ if $X
\prec_N Y$ or $X \cong Y$. Clearly, $\preceq_N$ is reflexive. Since the property of being an element of the $J$-hierarchy is determined by a $\Sigma_1$-formula, and these properties are upwards absolute, $\preceq_N$ is transitive. Therefore,
$\preceq_N$ defines a partial order on $\PC_N$.

If both $X$ and $Y$ are well-founded pseudocopies in $\PC_N$, we have either $X
\prec_N Y$ or $X \cong Y$ or $Y \prec_N X$, that is, ``true'' pseudocopies
(i.e., those that code a $J_\alpha$) are linearly pre-ordered by $\preceq_N$. Hence comparability can only fail if (at least) one of the
pseudocopies is not well-founded.

Provided with a countable subset of $\PC_N$, such as all the elements of $\PC_N$
recursive in a real $Z$, we will want to arithmetically define a subset that is
linearly ordered by $\preceq_N$ by excluding some ill-founded pseudocopies.

\begin{defn}
  Given a real $Z$, let $\PC_N(Z)$ be the intersection of $\PC_N$ with the set of reals computable in $Z$.
\end{defn}

\begin{lem} \label{lem:lin-pseudocode}
  For every natural number $N$ there is an arithmetic predicate such that for every real $Z$, the predicate defines a set of reals $\PC^*_N(Z) \subseteq \PC_N(Z)$ with the following properties:
  \begin{enumerate}[(1)]
   \item For every $X \in \PC^*_N(Z)$ and $z \in F_X$, if $z \in \vecJ^X$, then $\Set_X(z) \in \PC^*_N(Z)$.
    \item $\preceq_N$ is a total preorder on $\PC^*_N(Z)$.
  \item If $X \in \PC_N(Z)$ is well-founded, then $X \in \PC^*_N(Z)$.
  \end{enumerate}
\end{lem}

\begin{proof}
To start, note that if $z \in \vecJ^X$ and $X \in \PC_N(Z)$, then $\Set_X(z)$ is recursive in $Z$ and then by applying Lemma~\ref{lem:reflection}, $\Set_X(z) \in \PC_N(Z)$.

We show that if $X,Y$ are $\preceq_N$-incomparable, at least one of two must be ill-founded. Further, this identification of ill-foundedness is uniformly arithmetic in the pair $(X,Y)$. 

Suppose $X,Y$ are $\preceq_N$-incomparable.
Consider the predicate
\begin{equation*}
\Op{Iso}(x,y,X,Y) : \Leftrightarrow  x \in \vecJ^X \and y \in \vecJ^Y \and \Set_X(x) \cong \Set_Y(y).  
\end{equation*}
It yields an arithmetic partial function from $\vecJ^X$ to $\vecJ^Y$. It is single-valued by Lemma~\ref{lem:J-rigid}. Denote the domain of this function by $D_X$ and the range by $R_Y$. Both $D_X$ and $R_Y$ are linearly ordered by item (4)(d) of Definition~\ref{def:pseudocopy}.

Now we apply the incomparability of $X$ and $Y$ to show that there must be at least one instance of ill-foundedness in $X$ or in $Y$.

\begin{description}
\item[Case 1] \emph{$D_X$ is cofinal in $X$, $R_Y$ is cofinal in $Y$}.\\
This means, by the definition of the function for which $D_X$ and $R_Y$ are domain and range, respectively, the complete internal $J$-hierarchies of $X$ and $Y$, respectively, are pairwise isomorphic. Furthermore, these isomorphisms are compatible by Lemma~\ref{lem:J-rigid}. Their union hence exhibits an isomorphism between the structure coded by $X$ and the structure coded by $Y$, contradicting $\preceq_N$-incomparability.

\item[Case 2] \emph{$D_X$ is cofinal in $X$, $R_Y$ is bounded in $Y$}.\\
	Since $X \not\cong \Set_Y(z)$ for any $z \in \vecJ^Y$, $Y$ must omit $\bigcup_{z \in R_Y} \Set_Y(z)$ and hence the $J$-hierarchy in $Y$ is ill-founded. The case when $D_X$ is bounded and $R_Y$ is cofinal is analogous.

\item[Case 3] \emph{Both $D_X, R_Y$ are bounded in $X,Y$, respectively}.\\
In this case, $D_X$ and $R_Y$ are cuts in $\vecJ^X$ and $\vecJ^Y$, respectively. If these cuts
were principal in both structures, it would contradict the definition of $D_X$ and $R_Y$ by adding a new element to each set, as follows. In the limit case, reason as in Case 1: the union of the $\Set_X(x)$, $x \in D_x$, as evaluated in $X$, maps to the union of the $\Set_Y(y)$, $y \in R_Y$, as evaluated in $Y$. In the successor case, given an isomorphism between $\Set_X(x)$ and $\Set_Y(y)$, because $X$ and $Y$ code $\omega$-models, there also exists an isomorphism between $S(\Set_X(x))$, as evaluated in $X$, and $S(\Set_Y(y))$, as evaluated in $Y$.

Therefore, at least one of the two cuts is not principal, thereby exhibiting an instance of non-wellfoundedness.
\end{description}

By inspection of the proof, the case distinction and the identification of the ill-founded structure is uniformly arithmetic in the pair $(X,Y)$.

We obtain the arithmetic set $\PC^*_N(Z)$ by considering all incomparable pairs in $\PC_N(Z)$ and discarding all elements of $\PC_N(Z)$ such that they or some element of their $J$-hierarchy are shown to be ill-founded by the above analysis. This ensures property (1). Since all pairs are being considered,  $\PC^*_N(Z)$ is linearly pre-ordered by $\preceq_N$, ie. property (2) holds.
For property (3) of the lemma, note that any element removed from $\PC_N(Z)$ in this process is ill-founded.
\end{proof}

\subsubsection*{Taking limits of $\omega$-copies}

We can use the ordering $\preceq_N$ to construct limits of $\omega$-copies. This will be needed in the proof of Theorem~\ref{thm:mastercodes-in-NCR}.

\begin{lem} \label{lem:limit-copy}
    For every $N$, there exists a number $e_N$, which can be computed uniformly from $N$, such that the following holds. Suppose $\overline{X} = \{X_i \colon i \in \omega\}$ is a family of well-founded pseudocopies from $\PC_N$, in other words, 
  each $X_i$ codes a countable $J_{\alpha_i}$ in which there are at most $N$ uncountable cardinals. Let $\gamma$ be the supremum of the $\alpha_i$. Then there exists an $\omega$-copy of
  $J_\gamma$ recursive in $\overline{X}^{(e_N)}$.
\end{lem}

\begin{proof}
    We first note that if $X_i$ and $X_j$ code $J_{\alpha_i}$ and $J_{\alpha_j}$, respectively, and $\alpha_i<\alpha_j$, then the embedding of $J_{\alpha_i}$ as coded by  $X_i$ into $J_{\alpha_j}$ as coded by  $X_j$ is uniformly arithmetic in $\overline{X}.$  See Lemma~\ref{lem:isomophism-arithmetic-PCs}.  Thus, the directed system consisting of the structures coded by the $X_i$'s and the maps between them is arithmetic in $\overline{X}.$  It follows that the direct limit is also arithmetic in $\overline{X},$ and uniformly so.
\end{proof}

\bigskip

\subsection{Canonical copies are not random for continuous measures}
\label{sub:ncr-not-rand}

We now want to use the framework of $\omega$-copies to show that for any $\alpha
< \beta_N$, the canonical copy of a standard $J$-structure $\JA{k}$ cannot be
$K$-random for a continuous measure, with $K$ sufficiently large. 

The argument rests mostly on various applications of the stair trainer technique
introduced in Propositions~\ref{pro:jump-non-random-1} and
\ref{pro:jump-non-random-2}, adapted to the notions of codings of countable
$J$-structures developed in the previous sections. For convenience, we briefly
review the core concepts.
\begin{description}
\item[$\omega$-copy] A coding of a countable set-theoretic structure
  $\Tup{S,A}$, $A \subseteq S$, as a subset of $\omega$; see
  Definitions~\ref{def:relational_structure} and~\ref{def:omega-copy}.
\item[Canonical copy] The copy of a $J$-structure $\JA{n}$ with projectum
  $\rho^n_\alpha = 1$ by means of a canonical bijection $V_\omega
  \leftrightarrow \omega$; see Definition~\ref{def:canonical-copy}.
\item[Effective copy] An $\omega$-copy of a $J$-structure $\JA{n}$ from which
  the internal fine structure hierarchy can effectively be recovered; see
  Definition~\ref{def:effective-copy} and
  Corollary~\ref{cor:effective-copies-from-above}. A canonical copy is always
  effective.
\item[Pseudocopy] An $\omega$-copy of a rud-closed structure that is an $\omega$-model 
  satisfying $\varphi_{V=L}$ and other $J$-like properties. The set of pseudocopies is arithmetically
  definable. A pseudocopy may be ill-founded. If it is well-founded, it codes
  some countable level $J_\alpha$ of the $J$-hierarchy; see
  Definition~\ref{def:pseudocopy} and Lemma~\ref{pro:pseudocopy-arithmetic}. By
  comparing their internal $J$-hierarchies, a subset of pseudocopies can be
  linearly ordered (up to isomorphism). This linear ordering is arithmetic, too;
  see Lemma~\ref{lem:lin-pseudocode}.
\end{description}

We also fix some notation for the rest of this section. 
Given $N \in \omega$, we fix $c \in \omega$ to be a sufficiently large number. It will be greater
than the complexity of all arithmetic definitions ($N$-pseudocopies, comparison
of pseudocopies and $S$-operators, linearization) introduced in the previous
sections. In particular, Corollary~\ref{cor:transfer-function} yields that, if
$X$ and $Y$ are well-founded $\omega$-copies of a $J_\alpha$ where $\alpha <
\beta_N$, then the isomorphism between the two coded structures is recursive in
$(X \join Y)^{(c)}$. It will also be greater than the respective constants from Corollary~\ref{cor:transfer-S-copy} and Lemma~\ref{lem:limit-copy}. After proving a couple of auxiliary results (Lemmas~\ref{lem:arith-S-tower} and \ref{lem:stair-J-successor}), we will also assume $c$ to be greater than the constants appearing in these lemmas.

\medskip
We give a first application of the stair trainer technique (as used in the proofs of Propositions~\ref{pro:jump-non-random-1} and \ref{pro:jump-non-random-2}) in the  context of $\omega$-copies and pseudocopies. 

\begin{lem} \label{lem:arith-S-tower}
    For each $N$, there exist numbers $d,e \in \omega$ such that the following holds. Suppose $\mu$ is a continuous measure, $\alpha < \beta_N$,
    and $X$ is an $\omega$-copy of $J_\alpha$,  recursive in $\mu^{(m)}$, for some $m \in \omega$. Suppose
    further that $R$ is $(m+e)$-random with respect to $\mu$ and
    computes an $\omega$-copy of $J_{\alpha+1}$. Then there exists an $\omega$-copy of $J_{\alpha+1}$ recursive in $\mu^{(m+d)}$.
\end{lem}

\begin{proof}
  By Lemma~\ref{prop:copy-computes-S-below}, $R$ computes an $\omega$-copy of $S^{(n)}(J_\alpha)$, for all $n \in \omega$. 
  By Corollary~\ref{cor:transfer-S-copy}, we can fix a constant $d_N$ so that $\overline{S}^{(n)}(X)$ is recursive in $(R \join \mu)^{(m+d_N)}$. 
  By Lemma~\ref{lem:arithmetic-S-copy}, we can fix $a$ so that  for any $\omega$-copy of a transitive set, the $a$-th jump of the $\omega$-copy computes its $\overline{S}$. Therefore $\overline{S}(X)$ is recursive in $\mu^{(m+a)}.$ We may assume that $R$ is $(m+d_N+a+1)$-random for $\mu$.
  By Proposition~\ref{pro:random-join-jump} and Lemma~\ref{lem:stair-with-jump}, $\overline{S}(X)$ is recursive in $\mu^{(m+d_N)}$. Again by Lemma~\ref{lem:arithmetic-S-copy}, $\overline{S}^{(2)}(X)$ is recursive in $\mu^{(m+d_N+a)}$. Since $\overline{S}^{(2)}(X)$ is also recursive in $(R \join \mu)^{(m+d_N)}$, we can apply Proposition~\ref{pro:random-join-jump} and Lemma~\ref{lem:stair-with-jump} again and obtain that $\overline{S}^{(2)}(X)$ is recursive in $\mu^{(m+d_N)}$. We can continue inductively and obtain that for each $n \in \omega$, $\overline{S}^{(n)}(X)$ is recursive in $\mu^{(m+d_N)}$. Now apply Corollary~\ref{cor:successor-copy}.
\end{proof}

The lemma shows that with the help of a sufficiently random real that  computes an $\omega$-copy of the next level of the $J$-hierarchy, $\mu$ can reach a copy of this level arithmetically, too. Combined with Lemma~\ref{lem:limit-copy}, this will be the key ingredient in proving that canonical copies of standard codes cannot be random with respect to a continuous measure.   

\medskip
The next lemma establishes a similar fact for the standard $J$-structures $\JA[\delta]{n}$ over a given $J_\delta$.

\begin{lem} \label{lem:stair-J-successor}
  There exist numbers $d,e \in \omega$ such that the following holds. Suppose $\mu$ is a continuous measure
    and $X$ is an $\omega$-copy of
  $J_\delta$. Further suppose 
  $X$ is recursive in $\mu^{(m)}$. Finally, suppose that $R$ is
 $(m+e)$-random with respect to $\mu$ and
  $n$ is such that $R$ computes an $\omega$-copy of $\JA[\delta]{n}$. Then there exists an $\omega$-copy $\EM[\delta]{n}$ of $\JA[\delta]{n}$
  recursive in $\mu^{(m+d)}$.
\end{lem}

\begin{proof}
	The proof is similar to that of Lemma~\ref{lem:arith-S-tower}, inductively using Proposition~\ref{pro:random-join-jump}, Lemma~\ref{lem:stair-with-jump}, Corollary~\ref{lem:successor-copy}, and Corollary~\ref{cor:transfer-function}. We choose $d$ large enough so that the isomorphism between $X$'s and $R$'s copy of $J_\delta$ is computable in $(\mu \oplus R)^{(m+d)}$. The number $e$ then guarantees sufficient randomness in $R$ to apply the stair trainer technique.
\end{proof}

\medskip
From now on, we assume that $c$ is also greater than the respective constants
from Lemmas~\ref{lem:limit-copy}, \ref{lem:arith-S-tower} and 
\ref{lem:stair-J-successor}. We define $G(N) = (N+2)(3c+6)$.

\begin{thm} \label{thm:mastercodes-in-NCR}
  Suppose $N \geq 0$, $\alpha <
  \beta_N$, and for some $m > 0$, $\rho^m_\alpha = 1$. Then the canonical copy
  of the standard $J$-structure $\JA{m}$ is not $G(N)$-random with respect to
  any continuous measure.
\end{thm}

\begin{proof}
We fix $R$ and $m_R$ so that, when $R$ is interpreted as a
pair of reals, $R$ is the canonical copy of some standard $J$-structure
$\Tup{J_{\rho^{m_R}_\alpha}, A^{m_R}_\alpha}$, where $\rho^{m_R}_\alpha = 1$. We assume for
the sake of a contradiction that $R$ is $G(N)$-random with respect to a
continuous measure $\mu$.

\medskip
To obtain a contradiction similar to the proofs of
Propositions~\ref{pro:jump-non-random-1} and \ref{pro:jump-non-random-2}, we
inductively define a hierarchy of indices of (pseudo)copies arithmetic in $\mu$.

\begin{defn}
  For each $k$ with $0 \leq k \leq N$, we let
\begin{multline*}
  S_k = \{ e \in \omega\colon \Phi^{\mu^{(k(3c+6))}}_e \text{ is total  and } \Phi^{\mu^{(k(3c+6))}}_e \in \PC_N(\mu^{(k(3c+6))}) \\
  \qquad \text{ and  for all $d < e$, $\Phi^{\mu^{(k(3c+6))}}_d \neq \Phi^{\mu^{(k(3c+6))}}_e$ } \}.
\end{multline*}
  The relation $\preceq_N$ induces an ordering on the indices in each $S_k$, which will be denoted by $\preceq_N$, too. The linearly ordered subsets corresponding to $\PC^*_N$ are given as
  \[
    L_k = \{e \in S_k \colon \Phi^{\mu^{(k(3c+6))}}_e \in \PC^*_N(\mu^{(k(3c+6))})\}.
  \]
  Finally, we let
  \[
    I_k = \{ e \in S_k \colon \Phi^{\mu^{(k(3c+6))}}_e \text{ is well-founded} \}.
  \]
\end{defn}

By Lemma~\ref{lem:lin-pseudocode}, the sets $S_k$ and $L_k$ are arithmetic in $\mu^{(k(3c+6))}$, for all $k \leq N$. In particular, by choice of $c$, $L_k$ is recursive in $\mu^{(k(3c+6)+c)}$. 

The following lemma shows that $I_k$ is the longest well-founded initial segment of $L_k$.

\begin{lem}\label{lem:longest-wf-segment}
  Given $k \leq N$, let $I$ be a well-founded initial segment of $L_k$. Then, for every $e \in I$, $\Phi^{\mu^{(k(3c+6))}}_e$ is a well-founded pseudocopy.
\end{lem}

\begin{proof}[Proof of Lemma~\ref{lem:longest-wf-segment}]
    For any $i \in S_k$, denote by $M_e$ the structure coded by $\Phi^{\mu^{(k(3c+6))}}_e$. Since $I$ is well-founded, we can use induction on the ordinal height of $I$. Let $e \in I$. As induction hypothesis assume that for any $d \prec_N e$, $M_d$ is well-founded. Since $\Phi^{\mu^{(k(3c+6))}}_e$ is a pseudocopy, $M_e$ satisfies $\varphi_{V=L}$. Suppose $x$ is an element of $M_e$. Since $M_e$ satisfies $\varphi_{V=L}$, $x$ must be in the rud-closure of a member $y$ of $M_e$'s internal $J$-hierarchy. $y$ is transitive in $M_e$. Similarly, the rud-closure of $y$ is transitive in $M_e$ (it may be all of $M_e$). 
    
    By Lemma~\ref{lem:lin-pseudocode}, there is a $d \prec_N e$ such that $y$ with the membership relation of $M_e$ is isomorphic to $M_d$. The rud-closure of $M_d$ is isomorphic to the rud-closure of $y$. Hence the rud-closure of $y$ is well-founded. Therefore, $x$ belongs to the rud-closure of $y$, which is transitive in $M_e$ and well-founded.
\end{proof}

Note that, conversely, if an $\omega$-copy of $J_\beta$ is recursive in $\mu^{(k(3c+6))}$, $J_\beta$ will be indexed in $L_k$, since we only exclude ill-founded structures when passing to $\PC^*_N(\mu^{(k(3c+6))})$. Since it is well-founded, it will be indexed in $I_k$, too.

We will need an additional property of $I_k$.

\begin{lem} \label{lem:Ik_1}
  If $J_\beta$ is represented in $I_k$, then $\beta < \alpha$.
\end{lem}

\begin{proof}[Proof of Lemma~\ref{lem:Ik_1}]
  Suppose $J_\alpha$ is represented in $I_k$. By Lemma~\ref{lem:stair-J-successor}, there exists an $\omega$-copy $\XM$ of $\JA{m_R}$ recursive in $\mu^{(k(3c+6)+c)}$. (Recall $m_R$ is such that $\rho^{m_R}_\alpha =1$ and $R$ represents $\Tup{J_1, A^{m_R}_\alpha}$.) Comparing the canonical encoding of $J_1$ with $X$, we obtain that $R$ is recursive in $\mu^{(k(3c+6)+2c)}$, contradicting the randomness of $R$.  
\end{proof}

We continue the proof of Theorem~\ref{thm:mastercodes-in-NCR}. We apply Lemma~\ref{lem:recog-wf-segment} to $L_k$ to conclude $I_k$ is recursive in $\mu^{(k(3c+6)+c+4)}$ as follows. Since $R$ is a canonical copy, by Corollary~\ref{cor:effective-copies-from-above} it computes an $\omega$-copy of $J_\alpha$. We can use that copy to test for any pseudocopy $M$
with an index in $L_k$ whether $M$ embeds into $J_\alpha$ (using Corollary~\ref{cor:iso-arithmetic}). This can be done
recursively in $(\mu^{(k(3c+6))} \join R)^{(c)}$ by Lemma~
\ref{lem:lin-pseudocode} and the choice of $c$. By Lemma~\ref{lem:longest-wf-segment}, every pseudocopy with an index in $I_k$ will embed into $J_\alpha$.
Consequently, $I_k$ is recursive in $(\mu \join R)^{(k(3c+6)+c)}$. By choice of
$G$, $R$
is at least
$(k(3c+6)+c+5)$-random for $\mu$, so Lemma~\ref{lem:recog-wf-segment} implies
that $I_k$ is
recursive in $\mu^{(k(3c+6)+c+4)}$.

\medskip
Next we define by recursion a sequence of ordinals $\gamma_0, \ldots, \gamma_K$,
where $K$ is at most $N+1$. 

\begin{itemize}
  \item Let $\gamma_0 = \omega$ and $\xi_0 = 1$.
  \item Given $\gamma_k$,
we check whether $\gamma_k$ is a cardinal in each of the structures represented in $I_k$ to which it belongs. 
  \begin{itemize}
    \item If so, we let
\begin{align*}
  \qquad \gamma_{k+1}  & = \sup \{ \beta \colon \exists e \in I_k  \: ( \text{$\beta$ has cardinality at most $\gamma_k$} \\
  & \qquad \qquad\qquad\text{in the structure indexed by $e$})\}, \\
  \xi_{k+1} & = \sup \{ \beta \colon J_\beta \text{ is represented in } I_k \},
\end{align*}
and we continue the recursion. 
  
  \item Otherwise,  $\gamma_k$ is not a cardinal inside some structure $J_\delta$ represented in $I_k$. 
Since the recursion made it to step $k$, $\delta$ is greater than any $\beta$ such that $J_\beta$ has a representation in $I_{k-1}$. We terminate the recursion and let $K = k$. 

  \end{itemize}

\item If we reach $\gamma_{N+1}$, we terminate the recursion.
\end{itemize}
 
By way of an example, consider $\gamma_1$.  It is the supremum of the ordinals which appear and are seen to be countable in one of the structures indexed in $I_0$.  Either it is the $\omega_1$ of cofinally many of these structures or it is the supremum of their ordinals.  If at step $k+1$ there were a $J_\delta$ indexed in $I_k$ in which $\gamma_1$ is countable, and (by a condensation argument for the $J$-hierarchy as in Section~\ref{sub:projecta}) every $\gamma_j$ with $j\leq k$ is similarly countable, then the recursion ends.  For $\beta$ between $\gamma_1$ and the least such $\delta,$ $\gamma_1=(\omega_1)^{J_\beta}.$  More generally, for all $k\leq K$, and $\beta$ between $\gamma_k$ and the least $\delta$ as in the specification of the recursion,  $\gamma_k=(\omega_k)^{J_\beta}$.

\begin{lem} \label{lem:xi_k}
  For every $0 \leq k \leq K$, $J_{\xi_k+1}$ is represented in $I_k$.
\end{lem}

\begin{proof}[Proof of Lemma~\ref{lem:xi_k}]
  We first prove that $J_2$ is represented in $I_0$. The canonical copy of $J_1$ is recursive. By Lemma~\ref{lem:arith-S-tower}, there exists an $\omega$-copy of $J_2$ recursive in $\mu^{(c)}$. Note that $\rho_2 = 1$, hence the standard $J$-structure of $J_2$ is of the form $\Tup{J_1, A_2}$. The canonical copy of $\Tup{J_1,A_2}$ is recursive in $R$ by Lemma~\ref{cor:effective-copies-from-above}. By Lemma~\ref{lem:stair-J-successor}, there exists an $\omega$-copy of $\Tup{J_1,A_{2}}$ recursive in $\mu^{(2c)}$. The isomorphism between the $J_1$ of this copy and the canonical copy of $J_1$ is recursive in $\mu^{(3c)}$. Therefore, by Lemma~\ref{prop-non-accel}, the canonical copy of $\Tup{J_1,A_{2}}$ is recursive in $\mu$. By Corollary~\ref{cor:effective-copies-from-above}, there exists an $\omega$-copy of $J_2$ recursive in $\mu$.

  Now assume $k > 0$.
  We show that $J_{\xi_k}$ has a representation recursive in $\mu^{((k-1)(3c+6)+c)}$ as follows.
  If $\xi_k$ is the maximum of the $\beta$ for which $J_\beta$ is represented in $I_{k-1}$, there is an $\omega$-copy of $J_{\xi_k}$ recursive in $\mu^{((k-1)(3c+6))}$. Otherwise, since $I_{k-1}$ is recursive in $\mu^{(k-1)(3c+6)+4},$ Lemma~\ref{lem:limit-copy} implies that there exists an $\omega$-copy of $J_{\xi_k}$ recursive in $\mu^{((k-1)(3c+6)+4+c)}$. 
  
  Next, we can apply Lemma~\ref{lem:arith-S-tower} to obtain an $\omega$-copy of  $J_{\xi_k+1}$ recursive in $\mu^{((k-1)(3c+6)+4+2c)}$, which implies that it recursive in $\mu^{k(3c+6)}$ and hence that it is represented in $I_k$.
\end{proof}

The lemma implies that for each $i < K$, $\gamma_i$ appears as a cardinal in some structure represented in $I_i$.

\medskip
\emph{Case 1:} $K < N+1$, that is, the recursion terminates early.  Let $\delta$ be fixed to be least as it appeared in the definition of the $\gamma$ and $\xi$ sequences.

\smallskip
In this case 
either the ultimate projectum $\rho_\delta$ of $J_\delta$ is $1$ or there is
an $1< i < K$ such that $\rho_\delta$ is $\gamma_{i}$. This is because for every infinite ordinal less than $\gamma_K$, there is a structure represented in $I_{K-1}$ in which this ordinal is in one-to-one correspondence with some $\gamma_i$, $i< K$.

\medskip
\emph{Case 1a:} $\rho_\delta = 1$.
The canonical copy of $J_1$ is recursive. By Lemma~\ref{lem:stair-J-successor},
there exists an $\omega$-copy of $\Tup{J_{1}, A_\delta}$ recursive in $\mu^{(K
(3c+6)+c)}$. Lemma~\ref{pro:compute-canonical-copy} implies that the canonical copy $\XM$ of $\Tup{J_{1}, A_\delta}$ is recursive in $\mu^{(K (3c+6)+c+1)}$. Since $\alpha >
\delta$, $R$ computes $\XM$, by
Lemma~\ref{prop:copy-computes-below}. By Lemma~\ref{prop-non-accel}, $\mu$ computes $\XM$. Since $\XM$ is an effective copy, Corollary~\ref{cor:effective-copies-from-above} implies $\mu$ computes an $\omega$-copy of $J_\delta$. But this means $J_\delta$ is represented in $I_0$, which contradicts the definition of $\gamma_1$.

\medskip
\emph{Case 1b:} $\rho_\delta=\gamma_i > 1 $. We established above that $I_{i-1}$ is recursive in $\mu^{((i-1)(3c+6)+c+4)}.$ 
Then by
Lemma~\ref{lem:limit-copy}, there exists an $\omega$-copy $Y$ of $J_{\rho_\delta}$ recursive in $\mu^{((i-1)(3c+6)+2c+4)}$.

Since $J_\delta$ is represented in $I_K$, there exists an
$\omega$-copy of $J_\delta$ recursive in $\mu^{(K(3c+6))}$. By 
Lemma~\ref{lem:stair-J-successor}, there exists an $\omega$-copy $\Tup{X,M_X}$ of
$\Tup{J_{\rho_\delta}, A_\delta}$ recursive in $\mu^{(K (3c+6)+c)}$.
We think of $Y$ as a simple representation of $J_{\rho_\delta}$ relative to $\mu$, whereas $X$ is a complicated representation. 

By Lemma~\ref{lem:isomophism-arithmetic-PCs} and choice of $c$, we can retrieve the isomorphism between $X$'s and $Y$'s representations of $J_{\rho_\delta}$ recursively in $\mu^{(K 
(3c+6)+c+c)} = \mu^{(K
(3c+6)+2c)}$ and transfer $M_X$ (the coding of $A_\delta$ in $X$)
to $Y$. This gives us an $\omega$-copy $\Tup{Y,M_Y}$ of $\Tup{J_{\rho_\delta}, A_\delta}$
recursive in $\mu^{(K (3c+6)+2c)}$.

Since $\alpha> \delta$, $R$ computes, by
Lemma~\ref{prop:copy-computes-below}, another $\omega$-copy of
$\Tup{J_{\rho_\delta}, A_\delta}$, say $\Tup{Z,M_Z}$. This copy is trivial relative to $R$. Using at most $c$ jumps, the join of
$R$ and $\mu^{((i-1)(3c+6)+2c+4)}$ can compare $Y$ and $Z$ and map $M_Z$, the encoding of
$A_\delta$ in $Z$, to $M_Y$, the encoding of $A_\delta$ in $Y$. This way we obtain that the $\omega$-copy $\Tup{Y,M_Y}$
of
$\Tup{J_{\rho_\delta}, A_\delta}$ recursive in  $(R \join \mu^{((i-1) 
(3c+6)+2c+4)})^{(c)}$, i.e.\ recursive in $R \join \mu^{((i-1)
(3c+6)+3c+4)}$. By Lemma~\ref{lem:stair-with-jump}, $\Tup{Y,M_Y}$ is recursive in $\mu^{((i-1)(3c+6)+3c+4)}$.
The complicated predicate $M_Y$ on the simple structure $Y$ is in fact relatively simple.

By Corollary~
\ref{cor:copies-from-above}, $\mu^{((i-1)(3c+6)+3c+4 + 2)} = \mu^{(i(3c+6))}$
computes an
$\omega$-copy of $J_\delta$. But this implies $J_\delta$ is represented in
$S_{i}$. In particular, $J_\delta$ is represented in
$I_{i}$. This contradicts the fact that $\delta$ is greater than any $\beta$ such that $J_\beta$ has a representation in $I_{K-1}$.

\medskip
\emph{Case 2:} $K = N+1$.

\smallskip
The analysis is similar to Case 1. $\gamma_{N+1}$ is defined. Since $R$ cannot be represented in any $S_N$, $\gamma_{N+1} < \beta_N$. Hence $\gamma_{N+1}$ is not a cardinal in $J_{\gamma_{N+1}}$ and thus $J_{\gamma_{N+1}}$ must be projectable, say to $\gamma_i$. There is an $\omega$-copy of $J_{\gamma_{N+1}}$ recursive in $\mu^{(N(3c+6)+2c+4)}$. By the same argument as above, there is a $\omega$-copy of $\Tup{J_{\gamma_{i}}, A_{\gamma_{N+1}}}$ recursive in $\mu^{((i-1)(3c+6)+3c+4)}$, which in turn yields a an $\omega$-copy of $J_{\gamma_{N+1}}$ recursive in $\mu^{(i(3c+6))}$. This contradicts that $J_{\gamma_{N+1}}$ is not represented in any $I_i$, for $i \leq N$.

\medskip
This is sufficient to complete the proof of Theorem~\ref{thm:mastercodes-in-NCR}.
\end{proof}

\subsection{Finishing the proof of Theorem~\ref{thm:second-main}} % (fold)
\label{sub:finishing_the_proof_of_theorem_2}

We restate Theorem~\ref{thm:second-main}. Let $G(n)$ be the recursive function defined before the statement of Theorem~\ref{thm:mastercodes-in-NCR} in Section~\ref{sub:ncr-not-rand}. 

\begin{thm1.2a}
	For every $n \in \Nat$,
	\[
		\ZFC^{-}_n \nvdash \text{ ``$\NCR_{G(n)}$ is countable.'' }
 	\]
\end{thm1.2a}

\begin{proof}
    By Theorem \ref{thm:mastercodes-in-NCR}, the set $\mathcal{X}$ of canonical copies of standard $J$-structures is a subset of $\NCR_{G(n)}$. The non-existence of the representation of a measure with an arithmetic property is downwards absolute between $\omega$-models.  Therefore, to prove the theorem, it is sufficient to show that $\mathcal{X}$ is not countable in $L_{\beta_n}$.
    
    Every real inside $L_{\beta_n}$ is computable from some element in $\mathcal{X}$. This follows from Jensen's analysis of projecta: when a new real is constructed so is a  subset of $\omega$ that is a standard code for the $J$-structure in which the real appears.  By the usual Cantor diagonal argument, the set of reals in $L_{\beta_n}$ is uncountable in $L_{\beta_n}$, so $\mathcal{X}$ is not countable in $L_{\beta_n}$, too.

\end{proof}


\begin{thebibliography}{27}
\providecommand{\natexlab}[1]{#1}
\providecommand{\url}[1]{\texttt{#1}}
\expandafter\ifx\csname urlstyle\endcsname\relax
  \providecommand{\doi}[1]{doi: #1}\else
  \providecommand{\doi}{doi: \begingroup \urlstyle{rm}\Url}\fi

\bibitem{Ackermann:1937a}
W. Ackermann.
\newblock Die {W}iderspruchsfreiheit der {A}llgemeinen {M}engenlehre.
\newblock {\em Mathematische Annalen}, 114(1):305--315, 1937.


\bibitem{Bienvenu-Porter:2012a}
L. Bienvenu and C. Porter.
\newblock Strong reductions in effective randomness.
\newblock {\em Theoretical Computer Science}, 459(0):55 -- 68, 2012.


\bibitem{boolos-putnam:1968}
G. Boolos and H. Putnam.
\newblock Degrees of unsolvability of constructible sets of integers.
\newblock \emph{J. Symbolic Logic}, 33:\penalty0 497--513, 1968.
\newblock ISSN 0022-4812.

\bibitem{Day-Miller:2013a}
A. Day and J. Miller.
\newblock Randomness for non-computable measures.
\newblock {\em Trans. Amer. Math. Soc.}, 365:3575--3591, 2013.

\bibitem{Leeuw:1956a}
K.~de~Leeuw, E.~F. Moore, C.~E. Shannon, and N.~Shapiro.
\newblock Computability by probabilistic machines.
\newblock In {\em Automata studies}, Annals of mathematics studies, no. 34,
  pages 183--212. Princeton University Press, Princeton, N. J., 1956.


\bibitem{demuth:1988}
O. Demuth.
\newblock Remarks on the structure of tt-degrees based on constructive measure
  theory.
\newblock \emph{Comment. Math. Univ. Carolin.}, 29\penalty0 (2):\penalty0
  233--247, 1988.

\bibitem{devlin:1984}
K. J. Devlin.
\newblock \emph{Constructibility}.
\newblock Perspectives in Mathematical Logic. Springer-Verlag, Berlin, 1984.

\bibitem{downey-nies-weber-yu:2006}
R. Downey, A. Nies, R. Weber, and L. Yu.
\newblock Lowness and {$\Pi\sp 0\sb 2$} nullsets.
\newblock \emph{J. Symbolic Logic}, 71\penalty0 (3):\penalty0 1044--1052, 2006.

\bibitem{downey-hirschfeldt:2010}
R. G. Downey and D. R. Hirschfeldt.
\newblock {\em Algorithmic randomness and complexity}.
\newblock Springer, 2010.



\bibitem{enderton-putnam:note-hyperarithmetical_1970}
H. Enderton and H. Putnam.
\newblock {A note on the hyperarithmetical hierarchy.}
\newblock {\em J. Symb. Log.}, 35:429--430, 1970.

\bibitem{friedman:1970}
H. M. Friedman.
\newblock Higher set theory and mathematical practice.
\newblock \emph{Ann. Math. Logic}, 2\penalty0 (3):\penalty0 325--357, 1970.

\bibitem{gacs:1986}
P. G{\'a}cs.
\newblock Every sequence is reducible to a random one.
\newblock {\em Information and Control}, 70(2-3):186--192, 1986.

\bibitem{gacs:2005}
P. G{\'a}cs.
\newblock Uniform test of algorithmic randomness over a general space.
\newblock {\em Theoretical Computer Science}, 341:91--137, 2005.

\bibitem{gale-stewart:1953}
D. Gale and F. M. Stewart.
\newblock Infinite games with perfect information.
\newblock In \emph{Contributions to the theory of games, vol. 2}, Annals of
  Mathematics Studies, no. 28, pages 245--266. Princeton University Press,
  Princeton, N. J., 1953.

\bibitem{glasner:2003}
E. Glasner.
\newblock {\em Ergodic theory via joinings}, volume 101 of {\em Mathematical
  Surveys and Monographs}.
\newblock American Mathematical Society, Providence, RI, 2003.

\bibitem{hachtman:sub}
S. Hachtman
\newblock Calibrating determinacy strength in levels of the {B}orel hierarchy.
\newblock {\em   J. Symb. Log.}, 82(2):510--548, 2017.


\bibitem{Haken:2014a}
I.~R. Haken.
\newblock {\em Randomizing {R}eals and the {F}irst-{O}rder {C}onsequences of
  {R}andoms}.
\newblock PhD thesis, University of California, Berkeley, 2014.


\bibitem{harrington-kechris:1975}
L. A. Harrington and A. S. Kechris.
\newblock A basis result for {$\Sigma \sp{0}\sb{3}$} sets of reals with an
  application to minimal covers.
\newblock \emph{Proc. Amer. Math. Soc.}, 53\penalty0 (2):\penalty0 445--448,
  1975.


\bibitem{Hodes:1980a}
H.~T. Hodes.
\newblock Jumping through the transfinite: The master code hierarchy of turing
  degrees.
\newblock {\em The Journal of Symbolic Logic}, 45(2):pp. 204--220, 1980.


\bibitem{hoyrup-rojas:computabilityprobability_2009}
M. Hoyrup and C. Rojas.
\newblock Computability of probability measures and {M}artin-{L}{\"o}f
  randomness over metric spaces.
\newblock {\em Information and Computation}, 207(7):830--847, 2009.

\bibitem{jech:2003}
T. Jech.
\newblock \emph{Set Theory}.
\newblock Springer Monographs in Mathematics. Springer-Verlag, Berlin, 2003.
\newblock The third millennium edition, revised and expanded.

\bibitem{jensen:1972}
R. B. Jensen.
\newblock The fine structure of the constructible hierarchy.
\newblock \emph{Ann. Math. Logic}, 4:\penalty0 229--308; erratum, ibid. 4
  (1972), 443, 1972.

\bibitem{Jensen-Karp:1971a}
R. B. Jensen and C. Karp.
\newblock Primitive recursive set functions.
\newblock In {\em Axiomatic set theory}, volume 1, page 143. American
  Mathematical Soc., 1971.

\bibitem{Jockusch-Simpson:1976a}
C.~G. Jockusch, Jr. and S.~G. Simpson.
\newblock A degree-theoretic definition of the ramified analytical hierarchy.
\newblock {\em Ann. Math. Logic}, 10(1):1--32, 1976.


\bibitem{kautz:1991}
S. M. Kautz.
\newblock \emph{Degrees of Random sequences}.
\newblock PhD thesis, Cornell University, 1991.

\bibitem{kechris:1995}
A. S. Kechris.
\newblock {\em Classical Descriptive Set Theory}.
\newblock Springer, 1995.


\bibitem{kucera:1985}
Anton{\'{\i}}n Ku{\v{c}}era.
\newblock Measure, {$\Pi\sp 0\sb 1$}-classes and complete extensions of
  {$\mathrm{PA}$}.
\newblock In \emph{Recursion theory week (Oberwolfach, 1984)}, volume 1141 of
  \emph{Lecture Notes in Math.}, pages 245--259. Springer, Berlin, 1985.

\bibitem{kunen:2011}
K. Kunen.
\newblock \emph{Set theory}, volume 34 of \emph{Studies in Logic (London)}.
\newblock College Publications, London, 2011.

\bibitem{kurtz:1981}
S. Kurtz.
\newblock \emph{Randomness and Genericity in the Degrees of Unsolvability}.
\newblock PhD thesis, University of Illinois at Urbana-Champaign, 1981.

\bibitem{levin:1976}
L. A. Levin.
\newblock Uniform tests for randomness.
\newblock {\em Dokl. Akad. Nauk SSSR}, 227(1):33--35, 1976.

\bibitem{martin:1968}
D. A. Martin.
\newblock The axiom of determinateness and reduction principles in the
  analytical hierarchy.
\newblock \emph{Bull. Amer. Math. Soc.}, 74:\penalty0 687--689, 1968.

\bibitem{martin:1975}
D. A. Martin.
\newblock Borel determinacy.
\newblock \emph{Ann. of Math. (2)}, 102\penalty0 (2):\penalty0 363--371, 1975.

\bibitem{martin:1985}
D. A. Martin.
\newblock A purely inductive proof of {B}orel determinacy.
\newblock In {\em Recursion theory ({I}thaca, {N}.{Y}., 1982)}, volume 42 of
  {\em Proc. Sympos. Pure Math.}, pages 303--308. Amer. Math. Soc., Providence,
  RI, 1985.

\bibitem{martinloef:1966}
P. Martin-L{\"o}f.
\newblock The definition of random sequences.
\newblock \emph{Information and Control}, 9:\penalty0 602--619, 1966.

\bibitem{moschovakis:1980}
Y. N.  Moschovakis.
\newblock \emph{Descriptive set theory}.
\newblock North-Holland Publishing Co., Amsterdam, 1980.

\bibitem{nies:2009}
A. Nies.
\newblock Computability and randomness.
\newblock Oxford University Press, 2009.

\bibitem{oxtoby:1970}
J. C. Oxtoby.
\newblock Homeomorphic measures in metric spaces.
\newblock \emph{Proc. Amer. Math. Soc.}, 24:\penalty0 419--423, 1970.

\bibitem{parthasarathy:1967}
K. R. Parthasarathy.
\newblock \emph{Probability measures on metric spaces}.
\newblock Probability and Mathematical Statistics, No. 3. Academic Press Inc.,
  New York, 1967.

\bibitem{posner-robinson:1981}
D. B. Posner and R. W. Robinson.
\newblock Degrees joining to {${\bf 0}\sp{\prime} $}.
\newblock \emph{J. Symbolic Logic}, 46\penalty0 (4):\penalty0 714--722, 1981.

\bibitem{reimann:apal}
J. Reimann.
\newblock Effectively closed sets of measures and randomness.
\newblock {\em Ann. Pure Appl. Logic}, 156(1):170--182, 2008.

\bibitem{reimann-slaman:tams}
J. Reimann and T. A. Slaman.
\newblock Measures and their random reals.
\newblock {\em Transactions of the American Mathematical Society},
  367(7):5081--5097, 2015.

\bibitem{Sacks:1963a}
G.~E. Sacks.
\newblock {\em Degrees of unsolvability}.
\newblock Princeton University Press, Princeton, N.J., 1963.


\bibitem{sacks:1990}
G. E. Sacks.
\newblock \emph{Higher Recursion Theory}.
\newblock Perspectives in Mathematical Logic. Springer-Verlag, Berlin, 1990.

\bibitem{schnorr:1971}
C.-P. Schnorr.
\newblock {\em Zuf{\"a}lligkeit und {W}ahrscheinlichkeit. {E}ine algorithmische
  {B}egr{\"u}n\-dung der {W}ahrscheinlichkeitstheorie}.
\newblock Springer-Verlag, Berlin, 1971.


\bibitem{shoenfield:1961}
J. R. Shoenfield.
\newblock The problem of predicativity.
\newblock In \emph{Essays on the foundations of mathematics}, pages 132--139.
  Magnes Press, Hebrew Univ., Jerusalem, 1961.

\bibitem{shore-slaman:1999}
R.\ A. Shore and T. A. Slaman.
\newblock Defining the {T}uring jump.
\newblock \emph{Math. Res. Lett.}, 6\penalty0 (5-6):\penalty0 711--722, 1999.

\bibitem{woodin:2008}
W. H. Woodin.
\newblock A tt version of the {P}osner-{R}obinson theorem.
\newblock In \emph{Computational prospects of infinity. {P}art {II}.
  {P}resented talks}, volume 15 of \emph{Lect. Notes Ser. Inst. Math. Sci.
  Natl. Univ. Singap.}, pages 355--392. World Sci. Publ., Hackensack, NJ, 2008j.

\bibitem{zvonkin-levin:1970}
A. K. Zvonkin and L. A. Levin.
\newblock The complexity of finite objects and the basing of the concepts of
  information and randomness on the theory of algorithms.
\newblock \emph{Uspehi Mat. Nauk}, 25\penalty0 (6(156)):\penalty0 85--127,
  1970.






\end{thebibliography}
\end{document}